\documentclass[preprint,noinfoline]{imsart}
\usepackage{amsfonts}
\usepackage{amsbsy,amsmath,amsthm,amssymb}
\usepackage{geometry}

\setattribute{journal}{name}{}
\numberwithin{equation}{section}
\theoremstyle{plain}
\newtheorem{theorem}{Theorem}

\newtheorem{proposition}{Proposition}
\theoremstyle{remark}
\newtheorem{remark}{Remark}

\begin{document}

\begin{frontmatter}
\title{Adaptive Minimax Estimation over Sparse $\ell_q$-Hulls}
\runtitle{Adaptive Minimax Estimation over Sparse $\ell_q$-Hulls}
\begin{aug}
\author[a]{\fnms{Zhan}  \snm{Wang}\ead[label=e1]{wangx607@stat.umn.edu}},
\author[b]{\fnms{Sandra} \snm{Paterlini}\ead[label=e2]{sandra.paterlini@unimore.it} \thanksref{t1}},
\author[c]{\fnms{Fuchang} \snm{Gao}\ead[label=e3]{fuchang@uidaho.edu}},
\and
\author[a]{\fnms{Yuhong} \snm{Yang}\ead[label=e4]{yyang@stat.umn.edu} \thanksref{t2}}

\vspace{-0.2cm}
\address[a]{School of Statistics, University of Minnesota, USA,\\ \printead{e1,e4}}
\vspace{-0.2cm}
\address[b]{Department of Economics,  University of Modena and Reggio E., Italy,\\ \printead{e2}}
\vspace{-0.2cm}
\address[c]{Department of Mathematics, University of Idaho, USA, \\ \printead{e3}}
\thankstext{t1}{ Sandra Paterlini's research was partially supported by Fondazione Cassa di Risparmio di Modena.}
\thankstext{t2}{ Yuhong Yang's research was partially supported by NSF grant DMS-1106576.}

\runauthor{Z.Wang, S.Paterlini, F. Gao and Y.Yang}
\end{aug}
\begin{abstract} 
Given a dictionary of $M_n$ initial estimates of the unknown true regression function, we aim to construct linearly aggregated estimators 
that target the best performance among all the linear combinations under a sparse $q$-norm ($0 \leq q \leq 1$) constraint on the linear coefficients. 
Besides identifying the optimal rates of aggregation for these $\ell_q$-aggregation problems, our multi-directional (or universal)
 aggregation strategies by model mixing 
or model selection achieve the optimal
rates simultaneously over the full range of $0\leq q \leq 1$ for
general $M_n$ and upper bound $t_n$ of the $q$-norm. Both random and fixed designs, 
with known or unknown error variance, are handled, and the $\ell_q$-aggregations examined in this work cover major types of aggregation
problems previously studied in the literature.  Consequences on minimax-rate adaptive regression under $\ell_q$-constrained true
coefficients ($0 \leq q \leq 1$) are also provided. 

Our results show that the minimax rate of $\ell_q$-aggregation ($0 \leq q \leq 1$) is basically determined by an effective model
size, which is a sparsity index that depends on $q$, $t_n$, $M_n$, and the sample size $n$ in an easily interpretable way based on a classical model
selection theory that deals with a large number of models. In addition, in the fixed design case, the model selection approach is seen to
yield optimal rates of convergence not only in expectation but also with exponential decay of deviation probability. In contrast, the model mixing approach can have leading 
constant one in front of the target risk in the oracle inequality while not offering optimality in deviation probability.\\
\vspace{-0.4cm}
\end{abstract}

\begin{keyword}
\kwd{minimax risk}
\kwd{adaptive estimation}
\kwd{sparse $\ell_q$-constraint}
\kwd{linear combining}
\kwd{aggregation}
\kwd{model mixing}
\kwd{model selection}
\end{keyword}

\end{frontmatter}


\section{Introduction}

\label{introduction}

The idea of sharing strengths of different estimation procedures by
combining them instead of choosing a single one has led to fruitful and
exciting research results in statistics and machine learning. In statistics,
the theoretical advances have centered on optimal risk bounds that require
almost no assumption on the behaviors of the individual estimators to be
integrated (see, e.g., \cite{Yang1996, Yang2000, Catoni1997, Catoni2004,
JuditskyNemirovski2000, Nemirovski2000, Yang2004, Tsybakov2003} for early
representative work). While there are many different ways that one can
envision to combine the advantages of the candidate procedures, the
combining methods can be put into two main categories: those intended for 
\textit{combining for adaptation}, which aims at combining the procedures to
perform adaptively as well as the best candidate procedure no matter what
the truth is, and those for \textit{combining for improvement}, which aims
at improving over the performance of all the candidate procedures in certain
ways. Whatever the goal is, for the purpose of estimating a target function
(e.g., the true regression function), we expect to pay a price: the risk of
the combined procedure is typically larger than the target risk. The
difference between the two risks (or a proper upper bound on the difference)
is henceforth called \textit{risk regret} of the combining method.

The research attention is often focused on one but the main step in the
process of combining procedures, namely, \textit{aggregation of estimates},
wherein one has already obtained estimates by all the candidate procedures
(based on initial data, most likely from data splitting, or previous
studies), and is trying to aggregate these estimates into a single one based
on data that are independent of the initial data. The performance of the
aggregated estimator (conditional on the initial estimates) plays the most
important role in determining the total risk of the whole combined
procedure, although the proportion of the initial data size and the later
one certainly also influences the overall performance. In this work, we will
mainly focus on the aggregation step.

It is now well-understood that given a collection of procedures, although
combining procedures for adaptation and selecting the best one share the
same goal of achieving the best performance offered by the candidate
procedures, the former usually wins when model selection uncertainty is high
(see, e.g., \cite{YuanYang2005}). Theoretically, one only needs to pay a
relatively small price for aggregation for adaptation (\cite{Yang2000b,
Catoni2004, Tsybakov2003}). In contrast, aggregation for improvement under a
convex constraint or $\ell _{1}$-constraint on coefficients is associated
with a higher risk regret (as shown in \cite{JuditskyNemirovski2000,
Nemirovski2000, Yang2004, Tsybakov2003}). Several other directions of
aggregation for improvement, defined via proper constraints imposed on the $%
\ell _{0}$-norm alone or in conjunction with the $\ell _{1}$-norm of the
linear coefficients, have also been studied, including linear aggregation
(no constraint, \cite{Tsybakov2003}), aggregation to achieve the best
performance of a linear combination of no more than a given number of
initial estimates (\cite{Buneaetal2007}) and also under an additional
constraint on the $\ell _{1}$-norm of these coefficients (\cite{Lounici2007}%
). Interestingly, combining for adaptation has a fundamental role for
combining for improvement: it serves as an effective tool in constructing
multi-directional (or universal) aggregation methods that simultaneously
achieve the best performance in multiple specific directions of aggregation
for improvement. This strategy was taken in section 3 of \cite{Yang2004},
where aggregations of subsets of estimates are then aggregated to be
suitably aggressive and conservative in an adaptive way. Other uses of
subset models for universal aggregation have been handled in \cite%
{Buneaetal2007, RigolletTsybakov2011}.

The goal of this paper is to propose aggregation methods that achieve the
performance (in risk with/without a multiplying factor), up to a multiple of
the optimal risk regret as defined in \cite{Tsybakov2003}, of the best
linear combination of the initial estimates under the constraint that the $q$%
-norm ($0\leq q\leq 1$) of the linear coefficients is no larger than some
positive number $t_{n}$ (henceforth the \textit{$\ell _{q}$-constraint}). We
call this type of aggregation \textit{$\ell _{q}$-aggregation}. It turns out
that the optimal rate is simply determined by an \textit{effective model
size }$m_{\ast }$, which roughly means that only $m_{\ast }$ terms are
really needed for effective estimation. We strive to achieve the optimal $%
\ell _{q}$-aggregation simultaneously for all $q$ ($0\leq q\leq 1$) and $%
t_{n}$ ($t_{n}>0$). From the work in \cite{JuditskyNemirovski2000, Yang2004,
Tsybakov2003, AudibertCatoni2010}, it is known that by suitable aggregation
methods, the squared $L_{2}$ risk is no larger than that of the best linear
combination of the initial $M_{n}$ estimates with the $\ell _{1}$-norm of
the coefficients bounded by 1 plus the order $(\log (M_{n}/\sqrt{n}%
)/n)^{1/2} $ when $M_{n}\geq \sqrt{n}$ or $M_{n}/n$ when $M_{n}<\sqrt{n}$.
Two important features are evident here: 1) When $M_{n}$ is large, its
effect on the risk enlargement is only through a logarithmic fashion; 2) No
assumption is needed at all on how the initial estimates are possibly
correlated. The strong result comes from the $\ell _{1}$-constraint on the
coefficients.

Indeed, in the last decade of the twentieth century, the fact that $\ell
_{1} $-type of constraints induce sparsity has been used in different ways
for statistical estimation to attain relatively fast rates of convergence as
a means to overcome the curse of dimensionality. Among the most relevant
ones, Barron \cite{Barron1994} studied the use of $\ell _{1}$-constraint in
construction of estimators for fast convergence with neural nets; Tibshirani 
\cite{Tibshirani1996} introduced the Lasso; Chen, Donoho and Saunders \cite%
{Chenetal1998} proposed the basis pursuit with over complete bases.
Theoretical advantages have also been pointed out. Barron \cite{Barron1993}
showed that for estimating a high-dimensional function that has integrable
Fourier transform or a neural net representation, accurate approximation
error is achievable. Together with model selection over finite dimensional
neural network models, relatively fast rates of convergence, e.g., $[(d\log
n)/n]^{1/2},$ where $d$ is the input dimension, are obtained (see, e.g., 
\cite{Barron1994} with parameter discretization, section III.B in \cite%
{YangBarron1998} and section 4.2 in \cite{Barronetal1999} with continuous
models). Donoho and Johnstone \cite{DonohoJohnstone1994} identified how the $%
\ell _{q}$-constraint ($q>0$) on the mean vector affects estimation accuracy
under $\ell _{p}$ loss ($p\geq 1$) in an illustrative Gaussian sequence
model. For function estimation, Donoho \cite{Donoho1993} studied sparse
estimation with unconditional orthonormal bases and related the essential
rate of convergence to a sparsity index. In that direction, for a special
case of function classes with unconditional basis defined basically in terms
of bounded $q$-norm on the coefficients of the orthonormal expansion, the
rate of convergence $(\log n/n)^{1-q/2}$ was given in \cite{YangBarron1998}
(section 5). The same rate also appeared in the earlier work of Donoho and
Johnstone \cite{DonohoJohnstone1994} in some asymptotic settings. Note that
when $q=1,$ this is exactly the same rate of the risk regret for $\ell _{1}$%
-aggregation when $M_{n}$ is of order $n^{\kappa }$ for $1/2\leq \kappa
<\infty $.

General model selection theories on function estimation intend to work with
general and possibly complicatedly dependent terms. Considerable research
has been built upon subset selection as a natural way to pursue sparse and
flexible estimation. When exponentially many or more models are entertained,
optimality theories that handle a small number of models (e.g., \cite%
{Shibata1981, Li1987}) are no longer suitable. General theories were then
developed for estimators based on criteria that add an additional penalty to
the AIC type criteria, where the additional penalty term prevents
substantial overfitting that often occurs when working with exponentially
many models by standard information criteria, such as AIC and BIC. A
masterpiece of work with tremendous breadth and depth is Barron, Birg\'{e}
and Massart \cite{Barronetal1999}, and some other general results in
specific contexts of density estimation and regression with fixed or random
design are in \cite{YangBarron1998, Yang1999, BirgeMassart2001, Baraud2000,
Baraud2002, Birge2004}.

These model selection theories are stated for nonparametric scenarios where
none of the finite-dimensional approximating models is assumed to hold but
they are used as suitable sieves to deliver good estimators when the size of
the sieve is properly chosen (see, e.g., \cite{ShenWong1994, vandegeer1995,
BirgeMassart1998} for non-adaptive sieve theories). If one makes the
assumption that a subset model of at most $k_{n}$ terms holds ($\ell _{0}$%
-constraint), then the general risk bounds mentioned in the previous
paragraph immediately give the order $k_{n}\log \left( M_{n}/k_{n}\right) /n$
for the risk of estimating the target function under quadratic type losses.

Thus, the literature shows that both $\ell _{0}$- and $\ell _{1}$%
-constraints result in fast rates of convergence (provided that $M_{n}$ is
not too large and $k_{n}$ is relatively small), with hard-sparsity directly
coming from that only a small number of terms is involved in the true model
under the $\ell _{0}$-constraint, and soft-sparsity originating from the
fact that there can only be a few large coefficients under the $\ell _{1}$%
-constraint. In this work, with new approximation error bounds in $\ell
_{q,t_{n}}^{M_{n}}$-hulls (defined in section \ref{notationanddefintion})
for $0<q\leq 1$, from a theoretical standpoint, we will see that model
selection or model combining with all subset models in fact simultaneously
exploits the advantage of sparsity induced by $\ell _{q}$-constraints for $%
0\leq q\leq 1$ to the maximum extent possible.

Clearly, all subset selection is computationally infeasible when the number
of terms $M_{n}$ is large. To overcome this difficulty, an interesting
research direction is based on greedy approximation, where terms are added
one after another sequentially (see, e.g., \cite{Barronetal2008}). Some
general theoretical results are given in the recent work of \cite%
{Huangetal2008}, where a theory on function estimation via penalized squared
error criteria is established and is applicable to several greedy
algorithms. The associated risk bounds yield optimal rate of convergence for
sparse estimation scenarios. For aggregation methods based on exponential
weighting under fixed design, practical algorithms based on Monte Carlo
methods have been given in \cite{DalalyanTsybakov2009, RigolletTsybakov2011}.

Considerable recent research has focused on $\ell _{1}$-regularization,
producing efficient algorithms and related theories. Interests are both on
risk of regression estimation and on variable selection. Some estimation
risk bounds are in \cite{Bickeletal2009, GreenshteinRitov2004,
Koltchinskii2009a, Koltchinskii2009b, MeinshausenBuhlmann2006,
MeinshausenYu2009, Wainwright2009, vandegeer2008, ZhangHuang2008,
ZhangCH2010, ZhangT2009, YeZhang2010}.

The $\ell _{q}$-constraint, despite being non-convex for $0<q<1$, poses an
easier optimization challenge than the $\ell _{0}$-constraint, which is
known to define a NP-hard optimization problem and be hardly tractable for
large dimensions. Although a few studies have devoted to the algorithmic
developments of the $\ell _{q}$-constraint optimization problem, such as
multi-stage convex relaxation algorithm (\cite{Zhang2010}) and the DC
programming approach (\cite{Gassoetal2009}), little work has been done with
respect to the theoretical analysis of the $\ell _{q}$-constrained framework.

Sparse model estimation by imposing the $\ell _{q}$-constraint has found
consensus among academics and practitioners in many application fields,
among which, just to mention a few, compressed sensing, signal and image
compression, gene-expression, cryptography and recovery of loss data. The $%
\ell _{q}$-constraints do not only promote sparsity but also are often
approximately satisfied on natural classes of signal and images, such as the
bounded variation model for images and the bump algebra model for spectra (%
\cite{Donoho2006}).

Our $\ell _{q}$-aggregation risk upper bounds require no assumptions on
dependence of the initial estimates in the dictionary and the true
regression function is arbitrary (except that it has a known sup-norm upper
bound in the random design case). The results readily give minimax rate
optimal estimators for a regression function that is representable as a
linear combination of the predictors subject to $\ell _{q}$-constraints on
the linear coefficients.

Two recent and interesting results are closely related to our work, both
under fixed design only. Raskutti, Wainwright and Yu \cite{Raskuttietal2010}
derived in-probability minimax rates of convergence for estimating the
regression functions in $\ell _{q,t_{n}}^{M_{n}}$-hulls with minimal
conditions for the full range of $0\leq q\leq 1.$ In addition, in an
informative contrast, they have also handled the quite different problem of
estimating the coefficients under necessarily much stronger conditions.
Rigollet and Tsybakov \cite{RigolletTsybakov2011} nicely showed that
exponential mixing of least squares estimators by an algorithm of Leung and
Barron \cite{LeungBarron2006} over subset models achieves universal
aggregation of five different types of aggregation, which involve $\ell _{0}$%
- and/or $\ell _{1}$-constraints. Furthermore, they implemented a MCMC based
algorithm with favorable numerical results. As will be seen, in this context
of regression under fixed design, our theoretical results are broader with
improvements in several different ways.

Our theoretical work emphasizes adaptive minimax estimation under the mean
squared risk. Building upon effective estimators and powerful risk bounds
for model selection or aggregation for adaptation, we propose several
aggregation/combining strategies and derive the corresponding oracle
inequalities or index of resolvability bounds. Upper bounds for $\ell _{q}$%
-aggregations and for linear regression with $\ell _{q}$-constraints are
then readily obtained by evaluating the index of resolvability for the
specific situations, incorporating an approximation error result that
follows from a new and precise metric entropy calculation on function
classes of $\ell _{q,t_{n}}^{M_{n}}$-hulls. Minimax lower bounds that match
the upper rates are also provided in this work. Whatever the relationships
between the dictionary size $M_{n}$, the sample size $n$, and upper bounds
on the $\ell _{q}$-constraints, our estimators automatically take advantage
of the best sparse $\ell _{q}$-representation of the regression function in
a proper sense.

By using classical model selection theory, we have a simple explanation of
the minimax rates, by considering the effective model size $m_{\ast }$,
which provides the best possible trade-off between the approximation error,
the estimation error, and the additional price due to searching over not
pre-ordered terms. The optimal rate of risk regret for $\ell _{q}$%
-aggregation, under either hard or soft sparsity (or both together), can
then be unifyingly expressed as 
\begin{equation*}
REG(m_{\ast })=1\wedge \frac{m_{\ast }\left( 1+\log \frac{M_{n}}{m_{\ast }}%
\right) }{n},
\end{equation*}%
which can then be interpreted as the log number of models of size $m_{\ast }$
divided by the sample size ($\wedge 1$), as was previously suggested for the
hard sparsity case $q=0$ (e.g., Theorem 1 of \cite{YangBarron1998}, Theorems
1 and 4 of \cite{Yang1999}).

The paper is organized as follows. In section \ref{preliminaries}, we
introduce notation and some preliminaries of the estimators and aggregation
algorithms that will be used in our strategies. In addition, we derive
metric entropy and approximation error bounds for $\ell _{q,t_{n}}^{M_{n}}$%
-hulls that play an important role in determining the minimax rate of
convergence and adaptation. In section \ref{lqaggregation}, we derive
optimal rates of $\ell _{q}$-aggregation and show that our methods achieve
multi-directional aggregation. We also briefly talk about $\ell _{q}$-combination of
procedures. In section \ref{linearregressionrandom}, we
derive the minimax rate for linear regression with $\ell _{q}$-constrained
coefficients also under random design. In section \ref{linearregressionfixed}%
, we handle $\ell _{q}$-regression/aggregation under fixed design with known
or unknown variance. A discussion is then reported in section \ref%
{discussion}. In section \ref{generaloracle}, oracle inequalities are given
for the random design. Proofs of the results are provided in section \ref%
{proofs}. We note that some upper and lower bounds in the last two sections
may be of independent interest.

\vspace*{.2in}

\section{Preliminaries}

\label{preliminaries}

Consider the regression problem where a dictionary of $M_{n}$ prediction
functions ($M_{n}\geq 2$ unless stated otherwise) are given as initial
estimates of the unknown true regression function. The goal is to construct
a linearly combined estimator using these estimates to pursue the
performance of the best (possibly constrained) linear combinations. A
learning strategy with two building blocks will be considered. First, we
construct candidate estimators from subsets of the given estimates. Second,
we aggregate the candidate estimators using aggregation algorithms or model
selection methods to obtain the final estimator.  

\subsection{Notation and definition}

\label{notationanddefintion}

Let $(\mathbf{X}_{1},Y_{1}),\ldots ,(\mathbf{X}_{n},Y_{n})$ be $n$ ($n\geq 2$%
) i.i.d. observations where $\mathbf{X}_{i}=(X_{i,1},\ldots ,X_{i,d})$, $%
1\leq i\leq n,$ take values in $\mathcal{X}\subset \mathbb{R}^{d}$ with a
probability distribution $P_{X}$. We assume the regression model 
\begin{equation}
Y_{i}=f_{0}(\mathbf{X}_{i})+\varepsilon _{i},\hspace*{0.4in}i=1,\ldots n,
\label{regression}
\end{equation}%
where $f_{0}$ is the unknown true regression function to be estimated. The
random errors $\varepsilon _{i}$, $1\leq i\leq n$, are independent of each
other and of $\mathbf{X}_{i}$, and have the probability density function $%
h(x)$ (with respect to the Lebesgue measure or a general measure $\mu $)
such that $E(\varepsilon _{i})=0$ and $E(\varepsilon _{i}^{2})=\sigma
^{2}<\infty $. The quality of estimating $f_{0}$ by using the estimator $%
\hat{f}$ is measured by the squared $L_{2}$ risk (with respect to $P_{X} $) 
\begin{equation*}
R(\hat{f};f_{0};n)=E\Vert \hat{f}-f_{0}\Vert ^{2}=E\left( \int (\hat{f}%
-f_{0})^{2}dP_{X}\right) ,
\end{equation*}%
where, as in the rest of the paper, $\Vert \cdot \Vert $ denotes the $L_{2}$%
-norm with respect to the distribution of $P_{X}$.

Let $F_{n}=\{f_{1},f_{2},\ldots ,f_{M_{n}}\}$ be a dictionary of $M_{n}$
initial estimates of $f_{0}$. In this paper, unless stated otherwise, $\Vert
f_{j}\Vert \leq 1,$ $1\leq j\leq M_{n}$. Consider the constrained linear
combinations of the estimates $\mathcal{F}=\left\{ f_{\theta
}=\sum_{j=1}^{M_{n}}\theta _{j}f_{j}:\theta \in \Theta _{n},f_{j}\in
F_{n}\right\} $, where $\Theta _{n}$ is a subset of $\mathbb{R}^{M_{n}}$.
The problem of constructing an estimator $\hat{f}$ that pursues the best
performance in $\mathcal{F}$ is called \textit{aggregation of estimates}. We
consider aggregation of estimates with sparsity constraints on $\theta $.
For any $\theta =(\theta _{1},\ldots ,\theta _{M_{n}})^{\prime }$, define
the $\ell _{0}$-norm and the $\ell _{q}$-norm ($0<q\leq 1$) by 
\begin{equation*}
\Vert \theta \Vert _{0}=\sum_{j=1}^{M_{n}}I(\theta _{j}\neq 0),%
\mbox{\space and
\space }\Vert \theta \Vert _{q}=\left( \sum_{j=1}^{M_{n}}|\theta
_{j}|^{q}\right) ^{1/q},
\end{equation*}%
where $I(\cdot )$ is the indicator function. Note that for $0<q<1$, $\Vert
\cdot \Vert _{q}$ is not a norm but a quasinorm, and for $q=0,$ $\Vert \cdot
\Vert _{0}$ is not even a quasinorm. But we choose to refer them as norms
for ease of exposition. For any $0\leq q\leq 1$ and $t_{n}>0$, define the $%
\ell _{q}$-ball
\begin{equation*}
B_{q}(t_{n};M_{n})=\left\{ \mathbf{\theta }=(\theta _{1},\theta _{2},\ldots
,\theta _{M_{n}})^{\prime }:\Vert \mathbf{\theta }\Vert _{q}\leq
t_{n}\right\} .
\end{equation*}%
When $q=0,$ $t_{n}$ is understood to be an integer between $1$ and $M_{n},$
and sometimes denoted by $k_{n}$ to be distinguished from $t_{n}$ when $q>0.$
Define the \textit{$\ell _{q,t_{n}}^{M_{n}}$-hull} of $F_{n}$ to be the
class of linear combinations of functions in $F_{n}$ with the $\ell _{q}$%
-constraint 
\begin{equation*}
\mathcal{F}_{q}(t_{n})=\mathcal{F}_{q}(t_{n};M_{n};F_{n})=\left\{ f_{\theta
}=\sum_{j=1}^{M_{n}}\theta _{j}f_{j}:\theta \in B_{q}(t_{n};M_{n}),f_{j}\in
F_{n}\right\} ,0\leq q\leq 1,t_{n}>0.
\end{equation*}%
One of our goals is to propose an estimator $\hat{f}_{F_{n}}=%
\sum_{j=1}^{M_{n}}\hat{\theta}_{j}f_{j}$ such that its risk is upper bounded
by a multiple of the smallest risk over the class $\mathcal{F}_{q}(t_{n})$
plus a small risk regret term 
\begin{equation*}
R(\hat{f}_{F_{n}};f_{0};n)\leq C\inf_{f_{\theta }\in \mathcal{F}%
_{q}(t_{n})}\Vert f_{\theta }-f_{0}\Vert ^{2}+REG_{q}(t_{n};M_{n}),
\end{equation*}%
where $C$ is a constant that does not depend on $f_{0}$, $n$, and $M_{n}$,
or $C=1$ under some conditions. We aim to obtain the optimal order of
convergence for the risk regret term.

\subsection{Two starting estimators}

\label{somestartingestimators}

A key step of our strategy is the construction of candidate estimators using
subsets of the initial estimates. The following two estimators (T- and
AC-estimators) were chosen because of the relatively mild assumptions for
them to work with respect to the squared $L_{2}$ risk. Under the data
generating model (\ref{regression}) and i.i.d. observations $(\mathbf{X}%
_{1},Y_{1}),\ldots ,(\mathbf{X}_{n},Y_{n})$, suppose we are given $%
(g_{1},\ldots ,g_{m})$ terms for the regression problem.

When working on the minimax upper bounds in random design settings, we will
always make the following assumption on the true regression function.

\vspace*{.06in} {\scshape Assumption BD}: There exists a known constant $L >
0$ such that $\Vert f_{0}\Vert _{\infty }\leq L<\infty $.

\vspace*{0.1in} \textbf{(T-estimator)} Birg\'{e} \cite{Birge2004}
constructed the T-estimator and derived its $L_{2}$ risk bounds under the
Gaussian regression setting. The following proposition is a simple
consequence of Theorem 3 in \cite{Birge2004}. Suppose \newline
{\scshape T1}. The error distribution $h(\cdot )$ is normal; \newline
{\scshape T2}. $0<\sigma <\infty $ is known.


\begin{proposition}
\label{Birge} Suppose Assumptions {\scshape BD} and {\scshape T1}, {\scshape %
T2} hold. We can construct a T-estimator $\hat{f}^{(T)}$ such that 
\begin{equation*}
E\Vert \hat{f}^{(T)}-f_{0}\Vert ^{2}\leq C_{L,\sigma}\left( \inf_{\vartheta
\in \mathbb{R}^{m}}\left\Vert \sum_{j=1}^{m}\vartheta
_{j}g_{j}-f_{0}\right\Vert ^{2}+\frac{m}{n}\right) ,
\end{equation*}%
where $C_{L,\sigma}$ is a constant depending only on $L$ and $\sigma$.
\end{proposition}

\vspace*{0.1in} \textbf{(AC-estimator)} For our purpose, consider the class
of linear combinations with the $\ell _{1}$-constraint $\mathcal{G}%
=\{g=\sum_{j=1}^{m}\vartheta _{j}g_{j}:\Vert \vartheta \Vert _{1}\leq s\}$
for some $s>0$. Audibert and Catoni proposed a sophisticated AC-estimator $%
\hat{f}_{s}^{(AC)}$ (\cite{AudibertCatoni2010}, page 25). The following
proposition is a direct result from Theorem 4.1 in \cite{AudibertCatoni2010}
under the following conditions. \newline
{\scshape AC1}. There exists a constant $H>0$ such that $\sup_{g,g^{^{\prime
}}\in \mathcal{G},\mathbf{x}\in \mathcal{X}}|g(\mathbf{x})-g^{^{\prime }}(%
\mathbf{x})|=H<\infty .$ \newline
{\scshape AC2}. There exists a constant $\sigma ^{\prime }>0$ such that $%
\sup_{\mathbf{x}\in \mathcal{X}}E\left( (Y-g^{\ast }(\mathbf{X}))^{2}|%
\mathbf{X}=\mathbf{x}\right) \leq \left( \sigma ^{\prime }\right)
^{2}<\infty $, where $g^{\ast }=\inf_{g\in \mathcal{G}}\left\Vert
g-f_{0}\right\Vert ^{2}$.



\begin{proposition}
\label{Audibert} Suppose Assumptions {\scshape AC1} and {\scshape AC2} hold.
For any $s>0$, we can construct an AC-estimator $\hat{f}_{s}^{(AC)}$ such
that 
\begin{equation*}
E\Vert \hat{f}_{s}^{(AC)}-f_{0}\Vert ^{2}\leq \inf_{g \in \mathcal{G}}
\left\Vert g -f_{0}\right\Vert ^{2}+c \left( 2 \sigma^{\prime }+H \right)^{2}%
\frac{m}{n},
\end{equation*}%
where $c$ is a pure constant.
\end{proposition}

Note that under the assumption $\Vert f_{0}\Vert _{\infty }\leq L,$ we can
always enforce the estimators $\hat{f}^{(T)}$ and $\hat{f}_{s}^{(AC)}$ to be
in the range of $[-L,L]$ with the same risk bounds in the propositions.

\vspace*{0.1in}

\subsection{Two aggregation algorithms for adaptation}

\label{aggregationalgorithms}

Suppose $N$ estimates $\check{f}_{1},\ldots ,\check{f}_{N}$ are obtained
from $N$ candidate procedures based on some initial data. Two aggregation
algorithms, the ARM algorithm (Adaptive Regression by Mixing, Yang \cite%
{Yang2001}) and Catoni's algorithm (Catoni \cite{Catoni2004}), can be used
to construct the final estimator $\hat{f}$ by aggregating the candidate
estimates $\check{f}_{1},\ldots ,\check{f}_{N}$ based on $n$ additional
i.i.d. observations $(\mathbf{X}_{i},Y_{i})_{i=1}^{n}$. The ARM algorithm
requires knowing the form of the error distribution but it allows heavy tail
cases. In contrast, Catoni's algorithm does not assume any functional form
of the error distribution, but demands exponential decay of the tail
probability.

\vspace*{0.1in} \textbf{(The ARM algorithm)} Suppose \newline
{\scshape Y1}. There exist two known constants $\underline{\sigma }$ and $%
\overline{\sigma }$ such that $0<\underline{\sigma }\leq \sigma \leq 
\overline{\sigma }<\infty $; \newline
{\scshape Y2}. The error density function $h(x)$ has a finite fourth moment
and for each pair of constants $R_{0}>0$ and $0<S_{0}<1$, there exists a
constant $B_{S_{0},R_{0}}$ (depending on $S_{0}$ and $R_{0}$) such that for
all $|R|<R_{0}$ and $S_{0}\leq S\leq S_{0}^{-1}$, 
\begin{equation*}
\int h(x)\log \frac{h(x)}{S^{-1}h((x-R)/S)}dx\leq
B_{S_{0},R_{0}}((1-S)^{2}+R^{2}).
\end{equation*}%
We can construct an estimator $\hat{f}^{Y}$ which aggregates $\check{f}%
_{1},\ldots ,\check{f}_{N}$ by the ARM algorithm as
described below.

\begin{itemize}
\item[Step 1.] Split the data into two parts $Z^{(1)} = ( \mathbf{X}_i,
Y_i)_{i = 1}^{n_1}$, $Z^{(2)} = ( \mathbf{X}_i, Y_i)_{i = n_1 + 1}^{n}$.
Take $n_1 = \lceil n / 2 \rceil$.

\item[Step 2.] Estimate $\sigma ^{2}$ for each $\check{f}_{k}$
using the data $Z^{(1)}$,  
\begin{equation*}
\hat{\sigma}_{k}^{2}=\frac{1}{n_{1}}\sum_{i=1}^{n_{1}}\left( Y_{i}-\check{f}%
_{k}(\mathbf{X}_{i})\right) ^{2},\mbox{\space for }1\leq k\leq N.
\end{equation*}%
Clip the estimate $\hat{\sigma}_{k}^{2}$ into the range $[\underline{\sigma }%
^{2},\overline{\sigma }^{2}]$ if needed.

\item[Step 3.] Evaluate predictions for each $k$. For $n_{1}+1\leq l \leq n$%
, predict $Y_{l }$ by $\check{f}_{k}(\mathbf{X}_{l })$ and compute 
\begin{equation*}
E_{k,l }= \frac{ \prod_{i = n_1 + 1}^{l} h\left( (Y_i - \check{f}_k (\mathbf{%
X}_i)) / \hat{\sigma}_{k} \right) }{\hat{\sigma}_{k}^{l - n_1} }.
\end{equation*}

\item[Step 4.] Compute the final estimate $\hat{f}^{Y}=\sum_{k=1}^{N}W_{k}%
\check{f}_{k}$ with 
\begin{equation*}
W_{k}=\frac{1}{n-n_{1}}\sum_{l =n_{1}+1}^{n}W_{k,l }%
\mbox{\space
\space \space and \space }W_{k,l }=\frac{\pi _{k}E_{k,l }}{\sum_{j=1}^{N}\pi
_{j}E_{j,l }},
\end{equation*}%
where $\pi _{k}$ are prior probabilities such that $\sum_{k=1}^{N}\pi _{k}=1$%
.
\end{itemize}


\begin{proposition}
\label{Yang2004} (Yang \cite{Yang2004}, Proposition 1) Suppose Assumptions {%
\scshape BD} and {\scshape Y1}, {\scshape Y2} hold, and $\Vert \check{f}%
_{k}\Vert _{\infty }\leq L<\infty $ with probability $1$, $1\leq k\leq N$.
The estimator $\hat{f}^{Y}$ by the ARM algorithm has the risk 
\begin{equation*}
R(\hat{f}^{Y};f_{0};n)\leq C_{Y}\inf_{1\leq k\leq N}\left( \Vert \check{f}%
_{k}-f_{0}\Vert ^{2}+\frac{\sigma ^{2}}{n}\left( 1+\log \frac{1}{\pi _{k}}%
\right) \right) ,
\end{equation*}%
where $C_{Y}$ is a constant that depends on $\underline{\sigma },\overline{%
\sigma },L,$ and also $h$ (through the fourth moment of the random error and 
$B_{S_{0},R_{0}}$ with $S_{0}=\underline{\sigma }/\overline{\sigma },R_{0}=L$%
).
\end{proposition}

\begin{remark}
If $\sigma $ is known or other estimators of $\sigma $ are available, the
data splitting is not required, and the ARM algorithm consists of only Steps
3 and 4.
\end{remark}

\vspace*{.1in} \textbf{(Catoni's algorithm)} Suppose for some positive
constant $\alpha < \infty$, there exist known constants $U_{\alpha},
V_{\alpha} < \infty$ such that \newline
{\scshape C1}. $E(\exp ( \alpha | \varepsilon_i | ) ) \leq U_{\alpha}$;
\newline
{\scshape C2}. $\frac{ E( \varepsilon_i^2 \exp ( \alpha | \varepsilon_i | ) )%
}{E(\exp ( \alpha | \varepsilon_i | ) )} \leq V_{\alpha}$. 

The estimator built using Catoni's algorithm is $\hat{f}%
^{C}=\sum_{k=1}^{N}W_{k}\check{f}_{k}$ with 
\begin{equation*}
W_{k}=\frac{1}{n}\sum_{l=1}^{n}\frac{\pi _{k}\left(
\prod_{i=1}^{l}q_{k}(Y_{i}|\mathbf{X}_{i})\right) }{\sum_{j=1}^{N}\pi _{j}
\left( \prod_{i=1}^{l}q_{j}(Y_{i}|\mathbf{X}_{i}) \right) },%
\mbox{\space and
\space}q_{k}(Y_{i}|\mathbf{X}_{i})=\sqrt{\frac{\lambda _{C}}{2\pi }}\exp
\left\{ -\frac{\lambda _{C}}{2}(Y_{i}-\check{f}_{k}(\mathbf{X}%
_{i}))^{2}\right\} ,
\end{equation*}%
where $\lambda _{C}=\min \{\frac{\alpha }{2L},(U_{\alpha
}(17L^{2}+3.4V_{\alpha }))^{-1}\}$, and $\pi _{k}$ is the prior for $\check{f%
}_{k}$, $1\leq k\leq N$, such that $\sum_{k=1}^{N}\pi _{k}=1$.


\begin{proposition}
\label{Catoni} (Catoni \cite{Catoni2004}, Theorem 3.6.1) Suppose Assumptions 
{\scshape BD} and {\scshape C1}, {\scshape C2} hold, and $\Vert \check{f}%
_{k}\Vert _{\infty }\leq L<\infty $, $1\leq k\leq N$. The estimator $\hat{f}%
^{C}$ that aggregates $\check{f}_{1},\ldots ,\check{f}_{N}$ by Catoni's
algorithm has the risk 
\begin{equation*}
R(\hat{f}^{C};f_{0};n)\leq \inf_{1\leq k\leq N}\left( \Vert \check{f}%
_{k}-f_{0}\Vert ^{2}+\frac{2}{n\lambda _{C}}\log \frac{1}{\pi _{k}}\right) .
\end{equation*}
\end{proposition}

\begin{remark}
In the risk bound above, the multiplying constant in front of $\Vert \check{f%
}_{k}-f_{0}\Vert ^{2}$ is one, which can be important sometimes. Catoni \cite%
{Catoni2004} provided results under weaker assumptions than {\scshape C1}
and {\scshape C2}. In particular, $\varepsilon _{i}$ and $\mathbf{X}_{i}$ do
not have to be independent.
\end{remark}


\subsection{Metric entropy and sparse approximation error of $\ell
_{q,t_{n}}^{M_{n}}$-hulls}

\label{metricandsparse}

It is well-known that the metric entropy plays a fundamental role in
determining minimax-rates of convergence, as shown, e.g., in \cite%
{Birge1986, YangBarron1999}.

For each $1\leq m\leq M_{n}$ and each subset $J_{m}\subset \{1,2,\ldots
,M_{n}\}$ of size $m$, define $\mathcal{F}_{J_{m}}=\{\sum_{j\in J_{m}}\theta
_{j}f_{j}:\theta \in \mathbb{R}^{m}\}.$ Let 
\begin{equation*}
d^{2}(f_{0};\mathcal{F})=\inf_{f_{\theta }\in \mathcal{F}}\Vert f_{\theta
}-f_{0}\Vert ^{2}
\end{equation*}%
denote the smallest approximation error to $f_{0}$ over a function class $%
\mathcal{F}$.


\begin{theorem}
\label{Th1} (Metric entropy and sparse approximation bound for $\ell
_{q,t_{n}}^{M_{n}}$-hulls) Suppose $F_{n}=\{f_{1},f_{2},...,f_{M_{n}}\}$
with $\Vert f_{j}\Vert _{L^{2}(\nu )}\leq 1$, $1\leq j\leq M_{n}$, where $%
\nu $ is a $\sigma $-finite measure.

(i) For $0<q\leq 1$, there exists a positive constant $c_{q}$ depending only
on $q$, such that for any $0<\epsilon <t_{n}$, $\mathcal{F}_{q}(t_{n})$
contains an ${\epsilon }$-net $\{e_{j}\}_{j=1}^{N_{\epsilon }}$ in the $%
L_{2}(\nu )$ distance with $\Vert e_{j}\Vert _{0}\leq 5(t_{n}\epsilon
^{-1})^{2q/(2-q)}+1$ for $j=1,2,...,N_{\epsilon }$, where $N_{\epsilon }$
satisfies 
\begin{equation}
\log N_{\epsilon }\leq \left\{ 
\begin{array}{ll}
c_{q}\left( t_{n}\epsilon ^{-1}\right) ^{\frac{2q}{2-q}}\log (1+M_{n}^{\frac{%
1}{q}-\frac{1}{2}}t_{n}^{-1}\epsilon ) & {\mbox{ if }\epsilon }>t_{n}M_{n}^{%
\frac{1}{2}-\frac{1}{q}}, \\ 
{c_{q}M_{n}\log (1+M_{n}^{\frac{1}{2}-\frac{1}{q}}t_{n}\epsilon ^{-1})} & {%
\mbox{ if }{\epsilon }\leq t_{n}M_{n}^{\frac{1}{2}-\frac{1}{q}}.}%
\end{array}%
\right.  \label{entropy}
\end{equation}

(ii) For any $1\leq m\leq M_{n}$, $0<q\leq 1$, $t_{n}>0$, there exists a
subset $J_{m}$ and $f_{\theta ^{m}}\in \mathcal{F}_{J_{m}}$ with $\Vert
\theta ^{m}\Vert _{1}\leq t_{n}$ such that the sparse approximation error is
upper bounded as follows 
\begin{equation}
\Vert f_{\theta ^{m}}-f_{0}\Vert ^{2}-d^{2}(f_{0};\mathcal{F}%
_{q}(t_{n}))\leq 2^{2/q-1}t_{n}^{2}m^{1-2/q}.  \label{approx}
\end{equation}
\end{theorem}

The metric entropy estimate (\ref{entropy}) is the best possible. Indeed, if 
$f_{j}$, $1\leq j\leq M_{n}$, are orthonormal functions, then (\ref{entropy}%
) is sharp in order for any $\epsilon $ satisfying that $\epsilon /t_{n}$ is
bounded away from 1 (see \cite{Kuhn2001}). Also note that if we let $\nu
_{0} $ be the discrete measure $\frac{1}{n}\sum_{i=1}^{n}\delta _{x_{i}},$
where $\mathbf{x}_{1}$, $\mathbf{x}_{2}$, ..., $\mathbf{x}_{n}$ are fixed
points in a fixed design, then $\Vert g\Vert _{L^{2}(\nu _{0})}=(\frac{1}{n}%
\sum_{i=1}^{n}|g(\mathbf{x}_{i})|^{2})^{1/2}$. Thus, part (i) of Theorem \ref%
{Th1} implies Lemma 3 of \cite{Raskuttietal2010}, with an improvement of a $%
\log (M_{n})$ factor when $\epsilon \approx t_{n}M_{n}^{\frac{1}{2}-\frac{1}{%
q}}$, and an improvement from $(t_{n}\epsilon ^{-1})^{\frac{2q}{2-q}}\log
(M_{n}) $ to $M_{n}\log (1+M_{n}^{\frac{1}{q}-\frac{1}{2}}t_{n}\epsilon
^{-1})$ when $\epsilon <t_{n}M_{n}^{\frac{1}{2}-\frac{1}{q}}$. These
improvements are useful to derive the exact minimax rates for some of the
possible situations in terms of $M_{n},$ $q,$ and $t_{n}.$

With the tools provided in Yang and Barron \cite{YangBarron1999}, given fixed $q$ and $t_n$, one can
derive minimax rates of convergence for $\ell _{q}$-aggregation problems and
also for linear regression with $\ell _{q}$-constraints. However, the goal
for this work is to obtain adaptive estimators that simultaneously work for $%
\mathcal{F}_{q}(t_{n})$ with any choice of $0\leq q\leq 1$ and $t_{n}$, and
more. 


\subsection{An insight from the sparse approximation bound based on
classical model selection theory}

\label{aninsight}

Consider general $M_{n},t_n$ and $0<q\leq 1$. With the approximation error
bound in Theorem \ref{Th1}, classical model selection theory can 
provide key insight on what to expect regarding the minimax rate of
convergence for estimating a function in  $\ell _{q,t_{n}}^{M_{n}}$-hull.

Suppose $J_{m}$ is the best subset model of size $m$ in terms of having the
smallest $L_{2}$ approximation error to $f_{0}$. Then the estimator based on 
$J_{m}$ is expected to have the risk (under some squared error loss) of
order 
\begin{equation*}
t_{n}^{2}m^{1-2/q}+\frac{\sigma ^{2}m}{n}.
\end{equation*}%
Minimizing this bound over $m,$ we get the best choice (in order) in the
range $1\leq m\leq M_{n}\wedge n:$ 
\begin{equation*}
m^{\ast }=m^{\ast }(q,t_{n})=\left\lceil \left( nt_{n}^{2}\tau \right)
^{q/2}\right\rceil \wedge M_{n}\wedge n,
\end{equation*}%
where $\tau =\sigma ^{-2}$ is the precision parameter.
When $q=0$ with $t_{n}=k_{n},$ $m^{\ast }$ should be taken to be $%
k_{n}\wedge n.$ It is the \textit{ideal model size} (in order) under the $%
\ell _{q}$-constraint because it provides the best possible trade-off
between the approximation error and estimation error when $1\leq m\leq
M_{n}\wedge n$. The ratio $m^{\ast }/M_{n}$ is called a sparsity index in 
\cite{YangBarron1998} (section III.D) that characterizes, up to a log
factor, how much sparse estimation by model selection improves the
estimation accuracy based on nested models only. The calculation of
balancing the approximation error and the estimation error is well-known to
lead to the minimax rate of convergence for general full approximation sets
of functions with pre-determined order of the terms in an approximation
system (see section 4 of \cite{YangBarron1999}). However, when the terms are
not pre-ordered, there are many models of the same size $m^{\ast },$ and one
must pay a price for dealing with exponentially many or more models (see,
e.g., section 5 of \cite{YangBarron1999}). The classical model selection
theory that deals with searching over a large number of models tells us that
the price of searching over ${\binom{{M_{n}}}{{m}^{\ast }}}$ many models is
the addition of the term $\log {\binom{{M_{n}}}{{m}^{\ast }}}/n$ (e.g., \cite%
{BarronCover1991, YangBarron1998, Barronetal1999, Yang1999,
BirgeMassart2001, Baraud2002}). That is, the risk (under squared error type
of loss) of the estimator based on subset selection with a model descriptive
complexity term of order $\log {\binom{{M_{n}}}{{m}}}$ added to the AIC-type
of criteria is typically upper bounded in order by the smallest value of%
\begin{equation*}
\text{(squared) approximation error}_{m}+\frac{\sigma ^{2}m}{n}+\frac{\sigma
^{2}\log {\binom{{M_{n}}}{{m}}}}{n}
\end{equation*}%
over all the subset models, which is called the index of the resolvability
of the function to be estimated. Note that $\frac{m}{n}+\frac{\log {\binom{{%
M_{n}}}{{m}}}}{n}$ is uniformly of order $m\left( 1+\log \left( \frac{{M_{n}}%
}{m}\right) \right) /n$ over $0\leq m\leq M_{n}.$ Evaluating the above bound
at $m^{\ast }$ in our context yields a quite sensible rate of convergence.
Note also that $\log {\binom{{M_{n}}}{{m}^{\ast }}}/n$ (price of searching)
is of a higher order than $\frac{m^{\ast }}{n}$ (price of estimation) when $%
m^{\ast }\leq M_{n}/2$. Define 
\begin{equation*}
SER(m)=1+\log \left( \frac{{M_{n}}}{m}\right) \asymp \frac{m+\log {\binom{{%
M_{n}}}{{m}}}}{m},\hspace*{0.3in}1\leq m\leq M_{n},
\end{equation*}%
to be the ratio of the price with searching to that without searching (i.e.,
only the price of estimation of the parameters in the model). Here
\textquotedblleft $\asymp $\textquotedblright\ means of the same order as $%
n\rightarrow \infty $. Observe that reducing $m^{\ast }$ slightly will
reduce the order of searching price $\frac{m^{\ast }SER\left( m^{\ast
}\right) }{n}$ (since $x(1+\log \left( M_{n}/x\right) )$ is an increasing
function for $0<x<M_{n}$) and increase the order of the squared bias plus
variance (i.e., $t_{n}^{2}m^{1-2/q}+\frac{\sigma ^{2}m}{n}$). The best
choice will typically make the approximation error $t_{n}^{2}m^{1-2/q}$ of
the same order as $\frac{m(1+\log \frac{M_{n}}{m})}{n}$. Define \begin{equation*}
m_{\ast }=m_{\ast }(q,t_{n})=\left\{ 
\begin{array}{ll}
m^{\ast } & \text{if }m^{\ast }=M_{n}\wedge n, \\ 
\left\lceil \frac{m^{\ast }}{\left( 1+\log \frac{M_{n}}{m^{\ast }}\right)
^{q/2}}\right\rceil =\left\lceil \frac{m^{\ast }}{SER(m^{\ast })^{q/2}}%
\right\rceil & \text{otherwise.}%
\end{array}%
\right.
\end{equation*}

We call this the \textit{effective model size} (in order) under the $\ell
_{q}$-constraint because evaluating the index of resolvability expression
from our oracle inequality at the best model of this size gives the minimax
rate of convergence, as will be seen. When $m^{\ast }=n,$ the minimax risk
is of order 1 (or higher sometimes) and thus does not converge. Note that
the down-sizing factor $SER(m^{\ast })^{q/2}$ from $m^{\ast }$ to $m_{\ast }$
depends on $q$: it becomes more severe as $q$ increases; when $q=1,$ the
down-sizing factor reaches the order $\left( 1+\log \left( \frac{M_{n}}{%
m^{\ast }}\right) \right) ^{1/2}$. Since the risk of the ideal model and
that by a good model selection rule differ only by a factor of $\log
(M_{n}/m^{\ast })$, as long as $M_{n}$ is not too large, the price of
searching over many models of the same size is small, which is a fact well
known in the model selection literature (see, e.g., \cite{YangBarron1998},
section III.D).

For $q=0,$ under the assumption of at most $k_{n}\leq M_{n}\wedge n$ nonzero
terms in the linear representation of the true regression function, the risk
bound immediately yields the rate $\left( 1+\log {\binom{{M_{n}}}{{k}_{n}}}%
\right) /n\asymp $ $\frac{k_{n}\left( 1+\log \frac{M_{n}}{k_{n}}\right) }{n}$%
. Thus, from all above, we expect that $\frac{m_{\ast }SER\left( m_{\ast
}\right) }{n}\wedge 1$ is the unifying optimal rate of convergence for
regression under the $\ell _{q}$-constraint for $0\leq q\leq 1.$

The aforementioned rates of convergence for estimating functions in $\ell
_{q,t_{n}}^{M_{n}}$-hulls for $0\leq q\leq 1$ will be confirmed, and our
estimators will achieve the rates adaptively in some generality. From the insight
gained above, to construct a multi-directional (or universal) aggregation
method that works for all $0\leq q\leq 1,$ it suffices to aggregate the
estimates from the subset models for adaptation, which will automatically
lead to simultaneous optimal performance in $\ell
_{q,t_{n}}^{M_{n}}$-hulls.

\vspace*{0.2in}

\section{$\ell_q$-aggregation of estimates}

\label{lqaggregation}

Consider the setup from section \ref{notationanddefintion}. We focus on the
problem of aggregating the estimates in $F_{n}$ to pursue the best
performance in $\mathcal{F}_{q}(t_{n})$ for $0\leq q\leq 1$, $t_{n}>0$,
which we call \textit{$\ell _{q}$-aggregation of estimates}. To be more
precise, when needed, it will be called $\ell _{q}(t_{n})$-aggregation, and
for the special case of $q=0,$ we call it $\ell _{0}(k_{n})$-aggregation for 
$1\leq k_{n}\leq M_{n}.$

\subsection{The strategy}

\label{thestrategy}

For each $1\leq m\leq M_{n}\wedge n$ and each subset model $J_{m}\subset
\{1,2,\ldots ,M_{n}\}$ of size $m$, let $\mathcal{F}_{J_{m}}$ be as defined
in section \ref{metricandsparse}, and let $\mathcal{F}_{J_{m},s}^{L}=\{f_{%
\theta }=\sum_{j\in J_{m}}\theta _{j}f_{j}:\Vert \theta \Vert _{1}\leq
s,\Vert f_{\theta }\Vert _{\infty }\leq L\}$ be the class of $\ell _{1}$%
-constrained linear combinations in $F_{n}$ with a sup-norm bound on $%
f_{\theta }$. Our strategy is as follows.

\begin{itemize}
\item[Step I.] Divide the data into two parts: $Z^{(1)} = ( \mathbf{X}_i,
Y_i)_{i = 1}^{n_1}$ and $Z^{(2)} = (\mathbf{X}_i, Y_i)_{i = n_1 + 1}^{n}$.

\item[Step II.] Based on data $Z^{(1)}$, obtain a T-estimator for each
function class $\mathcal{F}_{J_{m}}$, or obtain an AC-estimator for each
combination of $s\in \mathbb{N}$ and function class $\mathcal{F}%
_{J_{m},s}^{L}$.

\item[Step III.] Based on data $Z^{(2)}$, combine all estimators obtained in
step II and the null model ($f\equiv 0$) using Catoni's or the ARM
algorithm. Let $p_{0}$ be a small positive number in $(0,1)$. In all, we
have to combine $\sum_{m=1}^{M_{n}\wedge n}{\binom{{M_{n}}}{{m}}}$
T-estimators with the weight $\pi _{J_{m}}=(1-p_{0})\left( (M_{n}\wedge n){%
\binom{{M_{n}}}{{m}}}\right) ^{-1}$ and the null model with the weight $\pi
_{0}=p_{0}$, or combine countably many AC-estimators with the weight $\pi
_{J_{m},s}=(1-p_{0})\left( (1+s)^{2}(M_{n}\wedge n){\binom{{M_{n}}}{{m}}}%
\right) ^{-1}$ and the null model with the weight $\pi _{0}=p_{0}$. (Note
that sub-probabilities on the models do not affect the validity of the risk
bounds to be given.)
\end{itemize}

For simplicity of exposition, from now on and when relevant, we assume $n$
is even and choose $n_{1}=n/2$ in our strategy. However, similar results
hold for other values of $n$ and $n_{1}$.

We use the expression \textquotedblleft \textbf{E}-\textbf{G} strategy" for
ease of presentation where \textbf{E} = \textbf{T} or \textbf{AC} represents
the estimators constructed in Step II, and \textbf{G} = \textbf{C} or 
\textbf{Y} stands for the aggregation algorithm used in Step III. By our
construction, Assumption AC1 is automatically satisfied: for each $J_{m}$, $%
H_{J_{m},s}=\sup_{f,f^{\prime }\in \mathcal{F}_{J_{m},s}^{L},\mathbf{x}\in 
\mathcal{X}}|f(\mathbf{x})-f^{\prime }(\mathbf{x})|\leq 2L$. Assumption AC2
is met with $\left( \sigma ^{\prime }\right) ^{2}=\sigma ^{2}+4L^{2}.$

We assume the following conditions are satisfied for each strategy,
respectively. 
\begin{eqnarray*}
\mbox{{\scshape A}}_{\mathbf{T}-\mathbf{C}}\mbox{ and }\mbox{{\scshape A}}_{%
\mathbf{T}-\mathbf{Y}} &:&\mbox{{\scshape BD}, {\scshape T1}, {\scshape T2}}.
\\
\mbox{{\scshape A}}_{\mathbf{AC}-\mathbf{C}} &:&%
\mbox{{\scshape BD},
 {\scshape C1}, {\scshape C2}}. \\
\mbox{{\scshape A}}_{\mathbf{AC}-\mathbf{Y}} &:&%
\mbox{{\scshape BD},
 {\scshape Y1}, {\scshape Y2}}.
\end{eqnarray*}

Given that {\scshape T1}, {\scshape T2} are stronger than {\scshape C1}, {%
\scshape C2} and {\scshape Y1}, {\scshape Y2}, it is enough to require their
satisfaction in $\mbox{{\scshape A}}_{\mathbf{T}-\mathbf{C}}$ and $%
\mbox{{\scshape A}}_{\mathbf{T} -\mathbf{Y}}$.
\vspace*{.1in}

\subsection{Minimax rates for $\ell_q$-aggregation of estimates}

\label{minimaxratesforlqaggregation}

\textbf{\ }Let $\mathcal{F}_{q}^{L}(t_{n})=\mathcal{F}_{q}(t_{n})\cap
\{f:\Vert f\Vert _{\infty }\leq L\}$ for $0\leq q\leq 1$. In the previous
section, we have defined $m_{\ast }=m_{\ast }(q,t_{n})$ to be the effective
model size for $0<q\leq 1.$ Now, for ease of presentation, we extend the
definition to 
\begin{equation*}
m_{\ast }^{\mathcal{F}}=\left\{ 
\begin{array}{ll}
m_{\ast }(q,t_{n}) & \text{for case 1: }\mathcal{F=F}_{q}(t_{n}),0<q\leq 1,
\\ 
k_{n}\wedge n & \text{for case 2: }\mathcal{F=F}_{0}(k_{n}), \\ 
m_{\ast }(q,t_{n})\wedge k_{n} & \text{for case 3: }\mathcal{F=F}%
_{q}(t_{n})\cap \mathcal{F}_{0}(k_{n}),0<q\leq 1.%
\end{array}%
\right. 
\end{equation*}%
Note that in the third case, we are simply taking the smaller one between
the effective model sizes from the soft sparsity constraint ($\ell _{q}$%
-constraint with $0<q\leq 1$) and the hard sparsity one ($\ell _{0}$%
-constraint), and this smaller size defines the final sparsity. Define%
\begin{equation*}
REG(m_{\ast }^{\mathcal{F}})=\sigma ^{2}\left( 1\wedge \frac{m_{\ast }^{%
\mathcal{F}}\cdot \left( 1+\log \left( \frac{M_{n}}{m_{\ast }^{\mathcal{F}}}%
\right) \right) }{n}\right) ,
\end{equation*}%
which will be shown to be typically the optimal rate of the risk regret for $%
\ell _{q}$-aggregation.
 In particular, Theorems \ref{UpAgg} and \ref{LowAgg}
provide upper and lower bounds to determine the order of the risk regret for 
$\ell _{q}$-aggregation of estimates. The specific behaviors of $REG(m_{\ast
}^{\mathcal{F}})$ for the three different cases will be precisely discussed
later.

For case 3, we intend to achieve the best performance of linear combinations
when both $\ell _{0}$- and $\ell _{q}$-constraints are imposed on the linear
coefficients, which results in $\ell _{q}$-aggregation using just a subset
of the initial estimates and will be called $\ell _{0}\cap \ell _{q}$%
-aggregation. For the special case of $q=1,$ this $\ell _{0}\cap \ell _{1}$%
-aggregation is studied in Yang \cite{Yang2004} (page 36) for
multi-directional aggregation and in Lounici \cite{Lounici2007} (called $D$%
-convex aggregation) more formally, giving also lower bounds. Our results
below not only handle $q<1$ but also close a gap of a logarithmic factor in
upper and lower bounds in \cite{Lounici2007}.

For ease of presentation, we may use the same symbol (e.g., $C$) to denote
possibly different constants of the same nature.


\begin{theorem}
\label{UpAgg} Suppose $\mbox{{\scshape A}}_{\mathbf{E}-\mathbf{G}}$ holds
for the \textbf{E}-\textbf{G} strategy respectively. Our estimator $\hat{f}%
_{F_{n}}$ simultaneously has the following properties.

\begin{itemize}
\item[(i)] For \textbf{T-} strategies, for $\mathcal{F=F}_{q}(t_{n})$ with $%
0<q\leq 1,$ or $\mathcal{F=F}_{0}(k_{n}),$ or $\mathcal{F=F}_{q}(t_{n})\cap 
\mathcal{F}_{0}(k_{n})$ with $0<q\leq 1,$ we have%
\begin{equation*}
R(\hat{f}_{F_{n}};f_{0};n)\leq \left[ C_{0}d^{2}(f_{0};\mathcal{F}%
)+C_{1}REG(m_{\ast }^{\mathcal{F}})\right] \wedge \left[ C_{0}\left(
\left\Vert f_{0}\right\Vert ^{2}\vee \frac{C_2 \sigma ^{2}}{n}\right) \right]
.
\end{equation*}

\item[(ii)] For \textbf{AC-} strategies, for $\mathcal{F=F}_{q}(t_{n})$ with 
$0<q\leq 1,$ or $\mathcal{F=F}_{0}(k_{n}),$ or $\mathcal{F=F}_{q}(t_{n})\cap 
\mathcal{F}_{0}(k_{n})$ with $0<q\leq 1,$ we have%
\begin{eqnarray*}
&&R(\hat{f}_{F_{n}};f_{0};n)\leq C_{1} REG(m_{\ast }^{\mathcal{F}})+ \\
&&C_{0}\left\{ 
\begin{array}{ll}
d^{2}(f_{0};\mathcal{F}_{q}^{L}(t_{n}))+ \frac{C_2 \sigma ^{2}\log (1+t_{n})%
}{n} & \text{for case 1,} \\ 
\inf_{s\geq 1}\left( \inf_{\{\theta :\Vert \theta \Vert _{1}\leq s,\Vert
\theta \Vert _{0}\leq k_{n},\Vert f_{\theta }\Vert _{\infty }\leq
L\}}\left\Vert f_{\theta }-f_{0}\right\Vert ^{2}+ \frac{C_2 \sigma ^{2}\log
(1+s)}{n}\right) & \text{for case 2,} \\ 
d^{2}(f_{0};\mathcal{F}_{q}^{L}(t_{n})\cap \mathcal{F}_{0}^{L}(k_{n}))+ 
\frac{C_2 \sigma ^{2}\log (1+t_{n})}{n} & \text{for case 3.}%
\end{array}%
\right.
\end{eqnarray*}%
Also, $R(\hat{f}_{F_{n}};f_{0};n)\leq C_0 \left( \left\Vert f_{0}\right\Vert
^{2}\vee \frac{C_{2}\sigma ^{2}}{n} \right).$
\end{itemize}

For all these cases, $C_{0}$ and $C_2$ do not depend on $%
n,f_{0},t_{n},q,k_{n},M_{n}$; $C_{1}$ does not depend on $%
n,f_{0},t_{n},k_{n},M_{n}$. These constants may depend on $L,$ $p_{0},$ $%
\sigma^{2}$ or $\overline{\sigma }^{2}/\underline{\sigma }^{2},$ $\alpha ,$ $%
U_{\alpha },V_{\alpha }$ when relevant. An exception is that $C_{0}=1$ for
the \textbf{AC-C} strategy.
\end{theorem}

\begin{remark}
When $q=1,$ our theorem covers some important previous aggregation results.
With $t_{n}=1,$ Juditsky and Nemirovski \cite{JuditskyNemirovski2000}
obtained the optimal result for large $M_{n}$; Yang \cite{Yang2004} gave
upper bounds for all $M_{n},$ but the rate is slightly sub-optimal (by a
logarithmic factor) when $M_{n}=O(\sqrt{n})$ and with a factor larger than 1
in front of the approximation error; Tsybakov \cite{Tsybakov2003} presented
the optimal rate for both large and small $M_{n},$ but under the assumption
that the joint distribution of $\{f_{j}(\mathbf{X}),j=1,...,M_{n}\}$ is
known. For the case $M_{n}=O(\sqrt{n}),$ Audibert and Catoni \cite%
{AudibertCatoni2010} have improved over \cite{Yang2004} and \cite%
{Tsybakov2003} by giving an optimal risk bound. Even when $q=1,$ our result
is more general in that $t_{n}$ is allowed to be arbitrary.
 Note also that
in some specific cases, the induced sparsity with $\ell _{1}$-constraint was
explored earlier in e.g., \cite{DonohoJohnstone1994, Barron1994,
YangBarron1998}. The latter two papers dealt with nonparametric situations
with mild assumptions on the terms in the approximation systems. In
particular, when the true function has a finite-order linear expression, the
estimators achieve the minimax optimal rate $\sqrt{(\log n)/n}$ when $M_{n}$
grows polynomially fast in $n.$
\end{remark}

\begin{remark}
The upper rate for $q=0$ as well as its interpretation is not new in the
literature (see, e.g., Theorem 1 of \cite{YangBarron1998}, Theorems 1 and 4
of \cite{Yang1999}): by noticing that there are ${\binom{M_{n}}{k_{n}}}$
subsets of size $k_{n}$ and that $\log {\binom{M_{n}}{k_{n}}}\leq
k_{n}\left( 1+\log (M_{n}/k_{n})\right) $, the rate for $q=0,$ which
directly imposes hard sparsity on the maximum number of relevant terms, is just the log number of models of size $k_{n}$ divided
by the sample size.
\end{remark}

\begin{remark}
Note that an extra term of $\log (1+t_{n})/n$ is present in the upper bounds
of the estimator obtained by \textbf{AC-} strategies. For case 1, if $%
t_{n}\leq e^{cn} \wedge e^{cm_* (1 + \log (M_n / m_*))}$ for a pure constant 
$c$, then $\log (1+t_{n})/n$ is upper bounded by a multiple of $REG(m_{\ast
}^{\mathcal{F}_{q}(t_{n})})$. Then, under the condition that the
approximation errors involved in the risk bounds are of the same order, 
\textbf{AC-} strategies have the same upper bound orders as \textbf{T-}
strategies. For case 2, the same is true if for some $s\leq e^{cn} \wedge
e^{ck_{n}( 1 + \log (M_{n}/k_{n})) } $, the $\ell _{1}$ norm constraint does
not enlarge the approximation error order. 
\end{remark}

\begin{remark}
For case 2, the boundedness assumption of $\Vert f_{j}\Vert \leq 1$, $1\leq
j\leq M_{n}$ is not necessary.
\end{remark}

\begin{remark}
If the true function $f_{0}$ happens to have a small $L_{2}$ norm such that $%
\Vert f_{0}\Vert ^{2}\vee \frac{\sigma ^{2}}{n}$ is of a smaller order than $%
REG(m_{\ast }^{\mathcal{F}}),$ then its inclusion in the risk bounds may
improve the rate of convergence.
\end{remark}

Next, we show that the upper rates in Theorem \ref{UpAgg} cannot be
generally improved by giving a theorem stating that the lower bounds of the
risk are of the same order in some situations, as is typically done in the
literature on aggregation of estimates. The following theorem implies that
the estimator by our strategies is indeed minimax adaptive for $\ell _{q}$%
-aggregation of estimates. Let $f_{1},\ldots ,f_{M_{n}}$ be an orthonormal
basis with respect to the distribution of $\mathbf{X}$. Since the earlier
upper bounds are obtained under the assumption that the true regression
function $f_{0}$ satisfies $\Vert f_{0}\Vert _{\infty }\leq L$ for some
known (possibly large) constant $L>0,$ for our lower bound result below,
this assumption will also be considered. For the last result in part (iii)
below under the sup-norm constraint on $f_{0},$ the functions $f_{1},\ldots
,f_{M_{n}}$ are specially constructed on $[0,1]$ and $P_{X}$ is the uniform
distribution on $[0,1]$. See the proof for details.

In order to give minimax lower bounds without any norm assumption on $f_{0},$
let $\widetilde{m}_{\ast }^{\mathcal{F}}$ be defined the same as $m_{\ast }^{%
\mathcal{F}}$ except that the ceiling of $n$ is removed. Define 
\begin{equation*}
\overline{REG}(\widetilde{m}_{\ast }^{\mathcal{F}})=\frac{\sigma ^{2}%
\widetilde{m}_{\ast }^{\mathcal{F}}\cdot \left( 1+\log \left( \frac{M_{n}}{%
\widetilde{m}_{\ast }^{\mathcal{F}}}\right) \right) }{n}\wedge \left\{ 
\begin{array}{ll}
t_{n}^{2} & \text{for cases 1 and 3,} \\ 
\infty  & \text{for case 2,}%
\end{array}%
\right. 
\end{equation*}%
\begin{equation*}
\underline{REG}(m_{\ast }^{\mathcal{F}})=REG(m_{\ast }^{\mathcal{F}})\wedge
\left\{ 
\begin{array}{ll}
t_{n}^{2} & \text{for cases 1 and 3,} \\ 
\infty  & \text{for case 2.}%
\end{array}%
\right. 
\end{equation*}


\begin{theorem}
\label{LowAgg} Suppose the noise $\varepsilon$ follows a normal distribution
with mean $0$ and variance $\sigma ^{2}>0$.

\begin{description}
\item[(i)] For any aggregated estimator $\hat{f}_{F_{n}}$ based on an
orthonormal dictionary $F_{n}=\{f_{1},\ldots ,f_{M_{n}}\}$, for $\mathcal{F=F%
}_{q}(t_{n}),$ or $\mathcal{F=F}_{0}(k_{n}),$ or $\mathcal{F=F}%
_{q}(t_{n})\cap \mathcal{F}_{0}(k_{n})$ with $0<q\leq 1,$ one can find a
regression function $f_{0}$ (that may depend on $\mathcal{F}$) such that%
\begin{equation*}
R(\hat{f}_{F_{n}};f_{0};n)-d^{2}(f_{0};\mathcal{F})\geq C\cdot \overline{REG}%
(\widetilde{m}_{\ast }^{\mathcal{F}}),
\end{equation*}%
where $C$ may depend on $q$ (and only $q$) for cases 1 and 3 and is an
absolute constant for case 2.

\item[(ii)] Under the additional assumption that $\Vert f_{0}\Vert \leq L$
for a known $L>0,$ the above lower bound becomes $C^{^{\prime }}\cdot 
\underline{REG}(m_{\ast }^{\mathcal{F}})$ for the three cases, where $%
C^{^{\prime }}$ may depend on $q$ and and $L$ for cases 1 and 3 and on $L$
for case 2.

\item[(iii)] With the additional knowledge $\Vert f_{0}\Vert _{\infty }\leq L
$ for a known $L>0,$ the lower bound $C^{^{\prime \prime }}\cdot \underline{%
REG}(m_{\ast }^{\mathcal{F}})$ also holds for the following situations: 1)
for $\mathcal{F=F}_{q}(t_{n})$ with $0<q\leq 1,$ if $\sup_{f_{\theta }\in 
\mathcal{F}_{q}(t_{n})}\Vert f_{\theta }\Vert _{\infty }\leq L;$ 2) for $%
\mathcal{F=F}_{0}(k_{n}),$ if $\sup_{1\leq j\leq M_{n}}\Vert f_{j}\Vert
_{\infty }\leq L<\infty $ and $\frac{k_{n}^{2}}{n}(1+\log \frac{M_{n}}{k_{n}}%
)$ are bounded above; 3) for $\mathcal{F=F}_{0}(k_{n}),$ if $M_{n}/\left(
1+\log \frac{M_{n}}{k_{n}}\right) \leq bn$ for some constant $b>0$ and the
orthonormal basis is specially chosen.
\end{description}
\end{theorem}

For satisfaction of $\sup_{f_{\theta }\in \mathcal{F}_{q}(t_{n})}\Vert
f_{\theta }\Vert _{\infty }\leq L,$ consider uniformly bounded functions $%
f_{j}$, then for $0<q\leq 1,$ 
\begin{equation*}
\Vert \sum_{j=1}^{M_{n}}\theta _{j}f_{j}\Vert _{\infty }\leq
\sum_{j=1}^{M_{n}}|\theta _{j}|\Vert f_{j}\Vert _{\infty }\leq \left(
\sup_{1\leq j\leq M_{n}}\Vert f_{j}\Vert _{\infty }\right) \Vert \theta
\Vert _{1}\leq \left( \sup_{1\leq j\leq M_{n}}\Vert f_{j}\Vert _{\infty
}\right) \Vert \theta \Vert _{q}.
\end{equation*}%
Thus, under the condition that $\left( \sup_{1\leq j\leq M_{n}}\Vert
f_{j}\Vert _{\infty }\right) t_{n}$ is upper bounded, $\sup_{f_{\theta }\in 
\mathcal{F}_{q}(t_{n})}\Vert f_{\theta }\Vert _{\infty }\leq L$ is met.

The lower bounds given in part (iii) of the theorem for the three cases of $%
\ell _{q}$-aggregation of estimates are of the same order of the upper
bounds in the previous theorem, respectively, unless $t_{n}$ is too small.
Hence, under the given conditions, the minimax rates for $\ell _{q}$%
-aggregation are identified. When no restriction is imposed on the norm of $%
f_{0},$ the lower bounds can certainly approach infinity (e.g., when $t_{n}$
is really large). That is why $\overline{REG}(\widetilde{m}_{\ast }^{%
\mathcal{F}})$ is introduced. The same can be said for later lower bounds.

For the new case $0<q<1,$ the $\ell _{q}$-constraint imposes a type of
soft-sparsity more stringent than $q=1$: even more coefficients in the
linear expression are pretty much negligible. For the discussion below,
assume $m^{\ast }<n.$ When the radius $t_{n}$ increases or $q\rightarrow 1$, 
$m^{\ast }$ increases given that the $\ell _{q}$-ball enlarges. When $%
m_{\ast }=m^{\ast }=M_{n}<n$, the $\ell _{q}$-constraint is not tight enough
to impose sparsity: $\ell _{q}$-aggregation is then simply equivalent to
linear aggregation and the risk regret term corresponds to the estimation
price of the full model, $M_{n}\sigma ^{2}/n$. In contrast, when $1<m_{\ast
}<M_{n}\wedge n$, the rate for $\ell _{q}$-aggregation can be expressed in
different ways:%
\begin{equation*}
\sigma ^{2-q}t_{n}^{q}\left( \frac{\log \left( 1+\frac{M_{n}}{(n\tau
t_{n}^{2})^{q/2}}\right) }{n}\right) ^{1-q/2}\asymp \frac{m_{\ast }}{n}%
SER\left( m_{\ast }\right) \asymp \frac{m_{\ast }}{n}SER\left( m^{\ast
}\right) \asymp \frac{m^{\ast }}{n}SER\left( m^{\ast }\right) ^{1-\frac{q}{2}%
}.
\end{equation*}%
The second expression is transparent in interpretation: due to the sparsity
condition, we only need to consider models of the effective size $m_{\ast }$
and the risk goes with the searching price $\frac{m_{\ast }}{n}SER\left(
m_{\ast }\right) $ (the estimation error of $m_{\ast }$ parameters is being
dominated in order). The last expression means that we can do better than
searching over the models of the ideal model size $m^{\ast },$ which has the
risk $\frac{m^{\ast }}{n}SER\left( m^{\ast }\right) .$ The minimax risk is
deflated by a factor of $SER\left( m^{\ast }\right) ^{\frac{q}{2}},$ which
becomes larger as $q\rightarrow 1$, pointing out that the factor $%
SER(m^{\ast })$ has to be downsized more as the $\ell _{q}$-ball becomes
larger. When $m^{\ast }=M_{n}$ (the full model), $SER(m^{\ast })$ reduces to
1. When $m^{\ast }\leq \left( 1+\log (M_{n}/m^{\ast })\right) ^{q/2}$ or
equivalently $m_{\ast }=1$, the $\ell _{q}$-constraint restricts the search
space of the optimization problem so much that it suffices to consider at
most one $f_{j}$ and the null model may provide a better risk.

Now let us explain that our $\ell _{q}$-aggregation includes the commonly
studied aggregation problems in the literature. First, when $q=1,$ we have
the well-known convex or $\ell _{1}$-aggregation (but now with the $\ell
_{1} $-norm bound allowed to be general). Second, when $q=0,$ with $%
k_{n}=M_{n}\leq n,$ we have the linear aggregation. For other $k_{n}<$ $%
M_{n}\wedge n,$ we have the aggregation to achieve the best linear
performance of only $k_{n}$ initial estimates. The case $q=0$ and $k_{n}=1$
has a special implication. Observe that from Theorem \ref{UpAgg}, we deduce
that for both the \textbf{T-} strategies and \textbf{AC-} strategies, under
the assumption $\sup_{j}\Vert f_{j}\Vert _{\infty }\leq L,$ our estimator
satisfies 
\begin{equation*}
R(\hat{f}_{F_{n}};f_{0};n)\leq C_{0}\inf_{1\leq j\leq M_{n}}\left\Vert
f_{j}-f_{0}\right\Vert ^{2}+C_{1}\sigma ^{2}\left( 1\wedge \frac{1+\log M_{n}%
}{n}\right) ,
\end{equation*}%
where $C_{0}=1$ for the \textbf{AC-C }strategy. Together with the lower
bound of the order $\sigma ^{2}\left( 1\wedge \frac{1+\log M_{n}}{n}\right) $
on the risk regret of aggregation for adaptation given in \cite{Tsybakov2003}%
, we conclude that $\ell _{0}(1)$-aggregation directly implies the
aggregation for adaptation (model selection aggregation). As mentioned
earlier, $\ell _{0}(k_{n})\cap \ell _{q}(t_{n})$-aggregation pursues the
best performance of the linear combination of at most $k_{n}$ initial
estimates with coefficients satisfying the $\ell _{q}$-constraint, which
includes the $D$-convex aggregation as a special case (with $q=1$).

\vspace*{0.1in}

\subsection{$\ell _{q}$-combination of procedures}

\label{lqcombination}

Suppose we start with a collection of estimation procedures $\Delta
=\{\delta _{1},\ldots ,\delta _{M_{n}}\}$ instead of a dictionary of
estimates. Let $\hat{f}_{j}$ be the estimator of the unknown true regression
function based on the procedure $\delta _{j}$, $1\leq j\leq M_{n}$, at a
certain sample size. Our goal is to combine the estimators $\{\hat{f}%
_{j}:1\leq j\leq M_{n}\}$ to achieve the best performance in 
\begin{equation*}
\mathcal{F}_{q}(t_{n};\Delta )=\left\{ \hat{f}_{\theta
}=\sum_{j=1}^{M_{n}}\theta _{j}\hat{f}_{j}:\Vert \theta \Vert _{q}\leq
t_{n}\right\} ,\mathit{\space}0\leq q\leq 1,t_{n}>0.
\end{equation*}

We split the data $(\mathbf{X}_{1},Y_{1}),\ldots ,(\mathbf{X}_{n},Y_{n})$
into three parts: $Z^{(1)}=(\mathbf{X}_{i},Y_{i})_{i=1}^{n_{1}}$, $Z^{(2)}=(%
\mathbf{X}_{i},Y_{i})_{i=n_{1}+1}^{n_{1}+n_{2}}$ and $Z^{(3)}=(\mathbf{X}%
_{i},Y_{i})_{i=n_{1}+n_{2}+1}^{n}$. Use the data $Z^{(1)}$ to obtain
estimators $\hat{f}_{1},\ldots ,\hat{f}_{M_{n}}$ and use the data $Z^{(2)}$
to construct T-estimators or AC-estimators based on subsets of $\hat{f}%
_{1},\ldots ,\hat{f}_{M_{n}}$. The data $Z^{(3)}$ are used to construct the
final estimator $\hat{f}_{\Delta }$ by aggregating the T-estimators or
AC-estimators and the null model using Catoni's or the ARM algorithm as done in
the previous section. For simplicity, assume $n$ is a multiple of $4$ and
choose $n_{1}=n/2$, $n_{2}=n/4$. Upper bounds for combining procedures by
our strategy are obtained similarly. The only difference is that $%
d^{2}(f_{0};\mathcal{F})$ is replaced by the risk of the best constrained
linear combination of the estimators $\hat{f}_{1,n/2},\ldots ,\hat{f}%
_{M_{n},n/2},$ where we add the second subscript $n/2$ to emphasize that the
estimators are constructed with a reduced sample size. For example, by 
\textbf{T-} strategies, we have that for any $0<q\leq 1$ and $t_{n}>0$, 
\begin{equation*}
R(\hat{f}_{\Delta };f_{0};n)\leq C_{0}\inf_{\theta \in
B_{q}(t_{n};M_{n})}E\left\Vert f_{0}-\sum_{j=1}^{M_{n}}\theta _{j}\hat{f}%
_{j,n/2}\right\Vert ^{2}+C_{1}\cdot REG(m_{\ast }^{\mathcal{F}_{q}(t_{n})}),
\end{equation*}%
and again such risk bounds simultaneously hold for $0\leq q\leq 1$ and $%
t_{n}>0.$

Note that these risk bounds involve the accuracies of the candidate
procedures at a reduced sample size $n/2$ due to data splitting to come up
with the estimates to be aggregated. Ideally, we want to have $C_{0}=1$ and $%
\hat{f}_{j,n/2}$ replaced by $\hat{f}_{j,n}.$ At this time, we are unaware
of any such risk bound that holds for combining general estimators (in fixed
design case, Leung and Barron's algorithm does not involve data splitting,
but it works only for least squares estimators). Because of this, the
theoretical attractiveness that the constant $C_{0}$ being $1$ in the
aggregation stage, unfortunately, disappears since the remaining parts in
the risk bounds also depend on the data splitting and there seems to be no
reason to expect with certainty that an aggregation method with $C_{0}=1$
has a better risk, even asymptotically, than another one with $C_{0}>1$.
Therefore, for combining general statistical procedures, it is unclear how
useful $C_{0}=1$ is even from a theoretical perspective. (It seems that
there is one scenario that one can argue otherwise: the candidate estimates
are truly provided. In the application of combining forecasts sequentially,
the candidate forecasts may be provided by other experts/commercial
companies and the statistician does not have access to the data based on
which the forecasts are built. In this context, since no data splitting is
needed, $C_{0}=1$ leads to a theoretical advantage compared to $C_{0}>1$.)
For this reason, in our view, results with $C_{0}>1$ (but not too large) are
also important for combining procedures. Indeed, such results often have
strengths in other aspects such as allowing heavy tail distributions for the
errors and allowing dependence of the observations.

Nonetheless, regardless of the degree of practical relevance, limiting
attention to the aggregation step and pursuing $C_{0}=1$ in that local goal
is certainly not without a theoretical appeal.

Some additional interesting results on combining procedures are in \cite%
{Audibert2009, Birge2004, BuneaNobel2008, DalalyanTsybakov2007,
DalalyanTsybakov2009, Giraud2008, Goldenshluger2009, GyorfiOttucsak2007,
Gyorfietal2002, Wegkamp2003, Yang2001}.

\vspace*{0.2in}

\section{Linear regression with $\ell _{q}$-constrained coefficients under
random design}

\label{linearregressionrandom}

Let's consider the linear regression model with $M_{n}$ predictors $%
X_{1},\ldots ,X_{M_{n}}$. Suppose the data are drawn i.i.d. from the
following model 
\begin{equation}
Y=f_{0}(\mathbf{X})+\varepsilon =\sum_{j=1}^{M_{n}}\theta
_{j}X_{j}+\varepsilon .  \label{linearregression}
\end{equation}%
As previously defined, for a function $f(x_{1},\ldots ,x_{M_{n}}):\mathcal{X}%
\rightarrow \mathbb{R}$, the $L_{2}$-norm $\Vert f\Vert $ is the square root
of $Ef^{2}(X_{1},\ldots ,X_{M_{n}})$, where the expectation is taken with
respect to $P_{X}$, the distribution of $\mathbf{X}$. Denote the $\ell
_{q,t_{n}}^{M_{n}}$-hull in this context by 
\begin{equation*}
\mathcal{F}_{q}(t_{n};M_{n})=\left\{ f_{\theta }=\sum_{j=1}^{M_{n}}\theta
_{j}x_{j}:\Vert \theta \Vert _{q}\leq t_{n}\right\} ,\mbox{\space}0\leq
q\leq 1,\mbox{\space}t_{n}>0.
\end{equation*}%
For linear regression, we assume coefficients of the true regression
function $f_{0}$ have a sparse $\ell _{q}$-representation ($0<q\leq 1$) or $%
\ell _{0}$-representation or both, i.e. $f_{0}\in \mathcal{F}$ where $%
\mathcal{F}=\mathcal{F}_{q}(t_{n};M_{n})$, $\mathcal{F}_{0}(k_{n};M_{n})$ or 
$\mathcal{F}_{q}(t_{n};M_{n})\bigcap \mathcal{F}_{0}(k_{n};M_{n})$.

Assumptions BD and $A_{\mathbf{E-G}}$ are still relevant in this section. As
in the previous section, for AC-estimators, we consider $\ell _{1}$- and
sup-norm constraints.

For each $1\leq m\leq M_{n}\wedge n$ and each subset $J_{m}$ of size $m$,
let $\mathcal{G}_{J_{m}}=\{\sum_{j\in J_{m}}\theta _{j}x_{j}:\theta \in 
\mathbb{R}^{m}\}$ and $\mathcal{G}_{J_{m},s}^{L}=\{\sum_{j\in J_{m}}\theta
_{j}x_{j}:\Vert \theta \Vert _{1}\leq s,\Vert f_{\theta }\Vert _{\infty
}\leq L\}$. We introduce now the adaptive estimator $\hat{f}_{A}$, built
with the same strategy used to construct $\hat{f}_{F_{n}}$ except that we
now consider $\mathcal{G}_{J_{m}}$ and $\mathcal{G}_{J_{m},s}^{L}$ instead
of $\mathcal{F}_{J_{m}}$ and $\mathcal{F}_{J_{m},s}^{L}$.

\vspace*{.1in}

\subsection{Upper bounds}

\label{upperbounds}

We give upper bounds for the risk of our estimator assuming 
$f_{0}\in 
\mathcal{F}_{q}^{L}(t_{n};M_{n})$, $\mathcal{F}_{0}^{L}(k_{n};M_{n})$, or $%
\mathcal{F}_{q}^{L}(t_{n};M_{n})\cap \mathcal{F}_{0}^{L}(k_{n};M_{n})$,
where $\mathcal{F}^{L}=\{f:f\in \mathcal{F},\Vert f\Vert _{\infty }\leq L\}$
for a positive constant $L$. Let $\alpha _{n}=$ $\sup_{f\in \mathcal{F}%
_{0}^{L}(k_{n};M_{n})}\inf \{\Vert \theta \Vert _{1}:$ $f_{\theta }=f\}$ be
the maximum smallest $\ell _{1}$-norm needed to represent the functions in $%
\mathcal{F}_{0}^{L}(k_{n};M_{n}).$ For ease of presentation, define $\Psi ^{%
\mathcal{F}}$ as follows:%
\begin{equation*}
\Psi ^{\mathcal{F}_{q}^{L}(t_{n};M_{n})}=\left\{ 
\begin{array}{ll}
\sigma ^{2} & \text{if $m_{\ast }=n$,} \\ 
\frac{\sigma ^{2}M_{n}}{n} & \text{if $m_{\ast }=M_{n}<n$,} \\ 
\sigma ^{2-q}t_{n}^{q}\left( \frac{1+\log \frac{M_{n}}{(nt_{n}^{2}\tau
)^{q/2}}}{n}\right) ^{1-q/2}\wedge \sigma ^{2} & \text{if $1<m_{\ast }<M_{n}$%
}\wedge n\text{,} \\ 
\left( t_{n}^{2}\vee \frac{\sigma ^{2}}{n}\right) \wedge \sigma ^{2} & \text{%
if $m_{\ast }=1$,}%
\end{array}%
\right.
\end{equation*}%
\begin{equation*}
\Psi ^{\mathcal{F}_{0}^{L}(k_{n};M_{n})}=\sigma ^{2}\left( 1\wedge \frac{%
k_{n}\left( 1+\log \frac{M_{n}}{k_{n}}\right) }{n}\right) ,
\end{equation*}%
\begin{equation*}
\Psi ^{\mathcal{F}_{q}^{L}(t_{n};M_{n})\cap \mathcal{F}%
_{0}^{L}(k_{n};M_{n})}=\Psi ^{\mathcal{F}_{q}^{L}(t_{n};M_{n})}\wedge \Psi ^{%
\mathcal{F}_{0}^{L}(k_{n};M_{n})}.
\end{equation*}%
In addition, for lower bound results, let $\underline{\Psi }^{\mathcal{F}%
_{q}^{L}(t_{n};M_{n})}$ ($0\leq q\leq 1$) and $\underline{\Psi }^{\mathcal{F}%
_{q}^{L}(t_{n};M_{n})\cap \mathcal{F}_{0}^{L}(k_{n};M_{n})}$ ($0<q\leq 1$)
be the same as $\Psi ^{\mathcal{F}_{q}^{L}(t_{n};M_{n})}$ and $\Psi ^{%
\mathcal{F}_{q}^{L}(t_{n};M_{n})\cap \mathcal{F}_{0}^{L}(k_{n};M_{n})}$,
respectively, except that when $0<q\leq 1$ and $m_{\ast }=1,$ $\underline{%
\Psi }^{\mathcal{F}_{q}^{L}(t_{n};M_{n})}$ takes the value $\sigma
^{2}\wedge t_{n}^{2}$ instead of $\sigma ^{2}\wedge \left( t_{n}^{2}\vee 
\frac{\sigma ^{2}}{n}\right) $ and $\underline{\Psi }^{\mathcal{F}%
_{q}^{L}(t_{n};M_{n})\cap \mathcal{F}_{0}^{L}(k_{n};M_{n})}$ is modified the
same way.


\begin{theorem}
\label{UpReg} Suppose $A_{\mathbf{E}-\mathbf{G}}$ holds for the \textbf{E}-%
\textbf{G} strategy respectively, and $\sup_{1\leq j\leq M_{n}}\Vert
X_{j}\Vert _{\infty }\leq 1$. The estimator $\hat{f}_{A}$ simultaneously has
the following properties.

\begin{itemize}
\item[(i)] For \textbf{T-} strategies, for $\mathcal{F=F}_{q}^{L}(t_n;M_{n})$
with $0<q\leq 1,$ or $\mathcal{F=F}_{0}^{L}(k_{n};M_{n}),$ or $\mathcal{F=F}%
_{q}^{L}(t_n;M_{n})\cap \mathcal{F}_{0}^{L}(k_{n};M_{n})$ with $0<q\leq 1,$
we have 
\begin{equation*}
\sup_{f_{0}\in \mathcal{F}}R(\hat{f}_{A};f_{0};n)\leq C_{1}\Psi ^{\mathcal{F}%
},
\end{equation*}%
where the constant $C_{1}$ does not depend on $n.$

\item[(ii)] For \textbf{AC-} strategies, for $\mathcal{F=F}%
_{q}^{L}(t_n;M_{n}) $ with $0<q\leq 1,$ or $\mathcal{F=F}%
_{0}^{L}(k_{n};M_{n}),$ or $\mathcal{F=F}_{q}^{L}(t_n;M_{n})\cap \mathcal{F}%
_{0}^{L}(k_{n};M_{n})$ with $0<q\leq 1,$ we have 
\begin{equation*}
\sup_{f_{0}\in \mathcal{F}}R(\hat{f}_{A};f_{0};n)\leq C_{1}\Psi ^{\mathcal{F}%
}+C_{2}\left\{ 
\begin{tabular}{ll}
$\frac{\sigma ^{2} \log (1+\alpha _{n})}{n}$ & for $\mathcal{F=F}%
_{0}^{L}(k_{n};M_{n}),$ \\ 
$\frac{\sigma ^{2} \log (1+t_n)}{n}$ & otherwise,%
\end{tabular}%
\right.
\end{equation*}%
where the constants $C_{1}$ and $C_2$ do not depend on $n.$
\end{itemize}
\end{theorem}

\begin{remark}
\vspace{1pt}The constants $C_{1}$ and $C_2$ may depend on $L,$ $p_{0},$ $%
\sigma ^{2}$, $\overline{\sigma }^{2}/\underline{\sigma }^{2},$ $\alpha ,$ $%
U_{\alpha },V_{\alpha }$ when relevant.
\end{remark}

\begin{remark}
The rate $\left( \frac{\log n}{n}\right) ^{1-q/2}$ for $0<q<1$ has appeared
in related regression or normal mean problems, e.g., in \cite%
{DonohoJohnstone1994} (Theorem 3), \cite{YangBarron1999} (section 5), \cite%
{Huangetal2008} (section 6), and \cite{Ing2010}. For function classes
defined in terms of infinite order orthonormal expansion with bounded $q$%
-norm of the coefficients and with $\ell _{2}$-norm of the tail coefficients
decaying at a polynomial order, the rate of convergence $(\log n/n)^{1-q/2}$
is derived in \cite{YangBarron1998} (page 1588) (when the tail of the
coefficients decays fast, the rate is improved to $(1/n)^{1-q/2}$). Note
that only the upper rates are given there.
\end{remark}

\vspace*{0.1in}

\subsection{Lower bounds}

\label{lowerbounds}

To derive lower bounds, we make the following near orthogonality assumption
on sparse sub-collections of the predictors. Such an assumption, similar to
the sparse Riesz condition (SRC) (Zhang \cite{Zhang2010}) under fixed
design, is used only for lower bounds but not for upper bounds.

{\scshape Assumption SRC}: For some $\gamma >0$, there exit two positive
constants $\underline{a}$ and $\overline{a}$ that do not depend on $n$ such
that for every $\theta $ with $\Vert \theta \Vert _{0}\leq \min (2\gamma
,M_{n})$ we have 
\begin{equation*}
\underline{a}\Vert \theta \Vert _{2}\leq \Vert f_{\theta }\Vert \leq 
\overline{a}\Vert \theta \Vert _{2}.
\end{equation*}


\begin{theorem}
\label{LowReg} Suppose the noise $\varepsilon $ follows a normal
distribution with mean $0$ and variance $0<\sigma ^{2}<\infty $.

\begin{itemize}
\item[(i)] For $0<q\leq 1$, under Assumption SRC with $\gamma =m_{\ast }$,
we have 
\begin{equation*}
\inf_{\hat{f}}\sup_{f_{0}\in \mathcal{F}_{q}(t_{n};M_{n})}E\Vert \hat{f}%
-f_{0}\Vert ^{2}\geq c\underline{\Psi }^{\mathcal{F}_{q}^{L}(t_{n};M_{n})}.
\end{equation*}

\item[(ii)] Under Assumption SRC with $\gamma =k_{n}$, we have 
\begin{equation*}
\inf_{\hat{f}}\sup_{f_{0}\in \mathcal{F}_{0}(k_{n};M_{n})\cap \{f_{\theta
}:\Vert \theta \Vert _{2}\leq a_{n}\}}E\Vert \hat{f}-f_{0}\Vert ^{2}\geq
c^{^{\prime }}\left\{ 
\begin{array}{ll}
\underline{\Psi }^{\mathcal{F}_{0}^{L}(k_{n};M_{n})} & \text{if $a_{n}\geq 
\tilde{c}\sigma \sqrt{\frac{k_{n}\left( 1+\log \frac{M_{n}}{k_{n}}\right) }{n%
}}$,} \\ 
a_{n}^{2} & \text{if $a_{n}<\tilde{c}\sigma \sqrt{\frac{k_{n}\left( 1+\log 
\frac{M_{n}}{k_{n}}\right) }{n}}$.}%
\end{array}%
\right.
\end{equation*}%
where $\tilde{c}$ is a pure constant.

\item[(iii)] For any $0<q\leq 1$, under Assumption SRC with $\gamma = k_{n}
\wedge m_{\ast } $, we have 
\begin{equation*}
\inf_{\hat{f}}\sup_{f_{0}\in \mathcal{F}_{0}(k_{n};M_{n})\cap \mathcal{F}%
_{q}(t_{n};M_{n})}E\Vert \hat{f}-f_{0}\Vert ^{2}\geq c^{^{\prime \prime }}%
\underline{\Psi }^{\mathcal{F}_{q}^{L}(t_{n};M_{n})\cap \mathcal{F}%
_{0}^{L}(k_{n};M_{n})}.
\end{equation*}
\end{itemize}

For all cases, $\hat{f}$ is over all estimators and the constants $c,$ $%
c^{^{\prime }}$ and $c^{^{\prime \prime }}$ may depend on $\underline{a}$, $%
\overline{a},$ $q$ and $\sigma ^{2}$.
\end{theorem}

\begin{remark}
\vspace{1pt}Note that in (i), at the transition from $m_{\ast }>1$ to $%
m_{\ast }=1,$ i.e., $nt_{n}^{2}\tau \approx 1+\log \frac{M_{n}}{%
(nt_{n}^{2}\tau )^{q/2}},$ we see continuity: 
\begin{equation*}
\sigma ^{2-q}t_{n}^{q}\left( \frac{1+\log \frac{M_{n}}{(nt_{n}^{2}\tau
)^{q/2}}}{n}\right) ^{1-q/2}\approx \frac{\sigma ^{2}\left( 1+\log \frac{%
M_{n}}{(nt_{n}^{2}\tau )^{q/2}}\right) }{n}\asymp t_{n}^{2}.
\end{equation*}
\end{remark}

For the second case (ii), the lower bound is stated in a more informative
way because the effect of the bound on $\Vert \theta \Vert _{2}$ is clearly
seen. Normality of the errors is not essential at all for the lower bounds.
With some additional efforts, one can show that these lower rates are also
valid under Assumption Y2, which we will not give here.

\subsection{The minimax rates of convergence}

Combining the upper and lower bounds, we give a representative minimax rate
result with the roles of the key quantities $n,$ $M_{n},$ $q,$ and $k_{n}$
explicitly seen in the rate expressions. Below \textquotedblleft $\asymp $%
\textquotedblright\ means of the same order when $L,$ $L_{0},$ $q,$ $t_{n}=t,
$ and $\overline{\sigma }^{2}$ ( $\overline{\sigma }^{2}$ is defined in
Theorem \ref{MinimaxReg} below) are held constant in the relevant
expressions.


\begin{theorem}
\label{MinimaxReg} Suppose the noise $\varepsilon $ follows a normal
distribution with mean $0$ and variance $\sigma^2$, and there exists a known
constant $\overline{\sigma }$ such that $0<\sigma \leq \overline{\sigma }%
<\infty $. Also assume there exists a known constant $L_{0}>0$ such that $%
\sup_{1\leq j\leq M_{n}}\Vert X_{j}\Vert _{\infty }\leq L_{0}<\infty $.

(i) For $0<q\leq 1$, under Assumption SRC with $\gamma =m_{\ast }$, 
\begin{equation*}
\inf_{\hat{f}}\sup_{f_{0}\in \mathcal{F}_{q}^{L}(t;M_{n})}E\Vert \hat{f}%
-f_{0}\Vert ^{2}\asymp 1\wedge \left\{ 
\begin{array}{ll}
1 & \text{if }m_{\ast }\text{$=n$,} \\ 
\frac{M_{n}}{n} & \text{if }m_{\ast }\text{$=M_{n}<n$,} \\ 
\left( \frac{1+\log \frac{M_{n}}{(nt^{2}\tau )^{q/2}}}{n}\right) ^{1-q/2} & 
\text{if $1\leq m_{\ast }<M_{n}$}\wedge n\text{.}%
\end{array}%
\right. 
\end{equation*}

(ii) If there exists a constant $K_{0}>0$ such that $\frac{k_{n}^{2}\left(
1+\log \frac{M_{n}}{k_{n}}\right) }{n}\leq K_{0}$, then under Assumption SRC
with $\gamma =k_{n}$, 
\begin{equation*}
\inf_{\hat{f}}\sup_{f_{0}\in \mathcal{F}_{0}^{L}(k_{n};M_{n})\cap
\{f_{\theta }:\Vert \theta \Vert _{\infty }\leq L_{0}\}}E\Vert \hat{f}%
-f_{0}\Vert ^{2}\asymp 1\wedge \frac{k_{n}\left( 1+\log \frac{M_{n}}{k_{n}}%
\right) }{n}.
\end{equation*}

(iii) If $\sigma >0$ is actually known, then under the condition $\frac{%
k_{n}^{2}\left( 1+\log \frac{M_{n}}{k_{n}}\right) }{n}\leq K_{0}$ and
Assumption SRC with $\gamma =k_{n}$, we have%
\begin{equation*}
\inf_{\hat{f}}\sup_{f_{0}\in \mathcal{F}_{0}^{L}(k_{n};M_{n})}E\Vert \hat{f}%
-f_{0}\Vert ^{2}\asymp 1\wedge \frac{k_{n}\left( 1+\log \frac{M_{n}}{k_{n}}%
\right) }{n},
\end{equation*}%
and for any $0<q\leq 1$, under Assumption SRC with $\gamma = k_{n} \wedge
m_{\ast } $, we have 
\begin{equation*}
\inf_{\hat{f}}\sup_{f_{0}\in \mathcal{F}_{0}^{L}(k_{n};M_{n})\cap \mathcal{F}%
_{q}^{L}(t;M_{n})}E\Vert \hat{f}-f_{0}\Vert ^{2}\asymp 1\wedge \left\{ 
\begin{array}{ll}
\frac{k_{n}\left( 1+\log \frac{M_{n}}{k_{n}}\right) }{n} & \text{if $m_{\ast
}>k_{n}$,} \\ 
\left( \frac{1+\log \frac{M_{n}}{(nt^{2}\tau )^{q/2}}}{n}\right) ^{1-q/2}
& \text{if $1\leq m_{\ast }\leq k_{n}$.}%
\end{array}%
\right.
\end{equation*}

\begin{remark}
When considering jointly the $\ell _{q}$-constraint for a fixed $0<q\leq 1$
and $q=0$, since the associated function classes are not nested, one cannot
immediately deduct the optimal rate of convergence for their intersection.
In our problem, the simple rule works: when the upper bound $k_{n}$ of the $%
\ell _{0}$-constraint is smaller than the effective model size $m_{\ast }$,
the additional $\ell _{q}$-constraint does reduce the parameter searching
space, but this reduction is not essential and the rate is equal to the rate
for $q=0$. In contrast, when the effective model size $m_{\ast }$ is smaller
than $k_{n}$, the $\ell _{0}$-constraint does reduce the parameter searching
space determined by the $\ell _{q}$-constraint, but not essential from the
uniform estimation standpoint and the rate is then $m_{\ast }\log
(1+M_{n}/m_{\ast })/n$. Clearly, both rates can be interpreted as the log
number of models of size $k_{n}$ or $m_{\ast }$ over the sample size.
\end{remark}
\end{theorem}

\vspace*{.1in}

\section{Adaptive minimax estimation under fixed design}

\label{linearregressionfixed}

Consider the linear regression model (\ref{linearregression}) under fixed
design, $Y_{i}=f_{0}(\mathbf{x}_{i})+\varepsilon _{i},$ $i=1,...,n,$ where $%
\mathbf{x}_{i}=(x_{i,1},\ldots ,x_{i,M_{n}})^{\prime }\in \mathcal{X}\subset 
\mathbb{R}^{M_{n}}$ are fixed, $1\leq i\leq n$, and the random errors $%
\varepsilon _{i}$ are i.i.d. $N(0,\sigma ^{2})$. Suppose $\max_{1\leq j\leq
M_{n}}\sum_{i=1}^{n}x_{i,j}^{2}/n\leq 1$. Let ${f}_{0}^{n}=(f_{0}(\mathbf{x}%
_{1}),\ldots ,f_{0}(\mathbf{x}_{n}))^{\prime }$. For any function $f$: $%
\mathcal{X}\rightarrow \mathbb{R}$, define the norm $\Vert \cdot \Vert _{n}$
by $\Vert f\Vert _{n}^{2}=\frac{1}{n}\sum_{i=1}^{n}f^{2}(\mathbf{x}_{i})$.
Our goal is to estimate the regression mean $f_{0}^{n}$ through a linear
combination of the predictors with the coefficients $\theta $ satisfying a $%
\ell _{q}$-constraint ($0\leq q\leq 1$). For an estimate $\hat{f}$ of $f_{0}$%
, define its average squared error to be
\begin{equation*}
ASE(\hat{f})=\Vert \hat{f}-f_{0}\Vert _{n}^{2}.
\end{equation*}%
We consider subset selection based estimators. Let $J_{m}\subset \{1,2,\ldots ,M_{n}\}$ be a model of size $m$ ($1\leq
m\leq M_{n}$). Our strategy is to choose a model using a model selection
criterion, and the resulting least squares estimator is used for $f_{0}^{n}$%
. The loss of a given model $J_{m}$ is $ASE(\hat{f}_{J_{m}})=\Vert \hat{Y}%
_{J_{m}}-{f}_{0}^{n}\Vert _{n}^{2}$ (with a slight abuse of notation), where 
$\hat{Y}_{J_{m}}=(\hat{Y}_{1,J_{m}},\ldots ,\hat{Y}_{n,J_{m}})^{\prime }$ is
the projection onto the column span of the design matrix of model $J_{m}$.
The alternative strategy of model mixing will be taken as well. Although our
estimators do not directly consider the $\ell _{q}$-constraint, it will be
shown to automatically adapt to the sparsity of $f_{0}$ in terms of $\ell
_{q}$-representation by the dictionary.

For a function class $\mathcal{F,}$ for the fixed design, define the
approximation error $d_{n}^{2}(f_{0};\mathcal{F})=\inf_{f\in \mathcal{F}%
}\Vert f-f_{0}\Vert _{n}^{2}.$ We will consider both $\sigma $ known and $%
\sigma $ unknown cases. As will be seen, the results are quite different in
some aspects, and an understanding on what the different assumptions can
lead to is important to reach a deeper insight on the theoretical issues.

\vspace*{0.1in}

\subsection{When $\protect\sigma $ is known}

\label{sigmaknown}

For a model $J_{m}$ of size $m$ ($1\leq m\leq M_{n}$), the ABC criterion
proposed in Yang (1999) is 
\begin{equation*}
ABC(J_{m})=\sum_{i=1}^{n}(Y_{i}-\hat{Y}_{i,J_{m}})^{2}+2r_{J_{m}}\sigma
^{2}+\lambda \sigma ^{2}C_{J_{m}},
\end{equation*}%
where $\lambda $ is a pure constant, $r_{J_{m}}$ is the rank of the design
matrix of $J_{m}$, and $C_{J_{m}}$ is the model index descriptive
complexity. Let $r_{M_{n}}$ denote the rank of the full model $J_{M_{n}},$
which is assumed to be at least $1$.

Let $\bar{J}$ denote the model that gives the full projection matrix $%
I_{n\times n}$ (since the ASE at the design points is the loss of interest,
this identity projection is permitted). We define $ABC(\bar{J})=2n\sigma
^{2}+\lambda \sigma ^{2}C_{\bar{J}}.$ Let $J_{0}$ denote the null model that
only includes the intercept and define $ABC(J_{0})=\sum_{i=1}^{n}(Y_{i}-\bar{%
Y})^{2}+2\sigma ^{2}+\lambda \sigma ^{2}C_{J_{0}}$, where $\bar{Y}%
=\sum_{i=1}^{n}Y_{i}/n$. The model index descriptive complexity $C_{J}$
satisfies $C_{J}>0$ and $\sum_{J}e^{-C_{J}}\leq 1,$ where the summation is
over all the candidate models being considered.

The subset models of size $1\leq m\leq M_{n}\wedge n$, the models $J_0$ and $%
\bar{J}$ are considered with the complexity $C_{J_{m}}=-\log 0.85+\log
\left( (M_{n}-1)\wedge n\right) +\log {\binom{M_{n}}{m}}$ for a subset model
with $m<M_{n}$, $C_{J_{M_{n}}}=-\log 0.05$ for the full model $J_{M_{n}}$, $%
C_{J_{0}}=-\log 0.05$ for the null model $J_{0}$, and $C_{\bar{J}}=-\log
0.05 $ for the full projection model $\bar{J}$. Note that for the purpose of
estimating $f_{0}^{n},$ there is no problem with duplication in the list of
candidate models.

Let $\Gamma _{n}$ denote the set of all the models considered and the model
chosen by the ABC criterion is 
\begin{equation*}
\hat{J}=\arg \min_{J\in \Gamma _{n}}ABC(J).
\end{equation*}

The ABC estimator $\hat{f}_{\hat{J}}$ is the fitted value $\hat{Y}_{\hat{J}}$%
. Let $\bar{f}_{J}=\mathcal{P}_{J}{f}_{0}^{n}$ be the projection of ${f}%
_{0}^{n}$ into the column space of the design matrix of model $J$.

For ease of presentation, define $\Phi ^{\mathcal{F}}$ as follows:%
\begin{equation*}
\Phi ^{\mathcal{F}_{q}(t_{n};M_{n})}=\left\{ 
\begin{array}{ll}
\frac{\sigma ^{2}r_{M_{n}}}{n} & \text{if $m_{\ast }=M_{n}\wedge n$,} \\ 
\sigma ^{2-q}t_{n}^{q}\left( \frac{1+\log \frac{M_{n}}{(nt_{n}^{2}\tau
)^{q/2}}}{n}\right) ^{1-q/2}\wedge \frac{\sigma ^{2}r_{M_{n}}}{n} & \text{if 
$1<m_{\ast }<M_{n}\wedge n$,} \\ 
(t_{n}^{2}\vee \frac{\sigma ^{2}}{n})\wedge \frac{\sigma ^{2}r_{M_{n}}}{n} & 
\text{if $m_{\ast }=1$.}%
\end{array}%
\right.
\end{equation*}%
\begin{equation*}
\Phi ^{\mathcal{F}_{0}(k_{n};M_{n})}=\frac{\sigma ^{2}k_{n}\left( 1+\log 
\frac{M_{n}}{k_{n}}\right) }{n}\wedge \frac{\sigma ^{2}r_{M_{n}}}{n},
\end{equation*}%
\begin{equation*}
\Phi ^{\mathcal{F}_{q}(t_{n};M_{n})\cap \mathcal{F}_{0}(k_{n};M_{n})}=\Phi ^{%
\mathcal{F}_{q}(t_{n};M_{n})}\wedge \Phi ^{\mathcal{F}_{0}(k_{n};M_{n})}.
\end{equation*}%
In addition, for lower bound results, let $\underline{\Phi }^{\mathcal{F}%
_{q}(t_{n};M_{n})}$ ($0\leq q\leq 1$) and $\underline{\Phi }^{\mathcal{F}%
_{q}(t_{n};M_{n})\cap \mathcal{F}_{0}(k_{n};M_{n})}$ ($0<q\leq 1$) be the
same as $\Phi ^{\mathcal{F}_{q}(t_{n};M_{n})}$ and $\Phi ^{\mathcal{F}%
_{q}(t_{n};M_{n})\cap \mathcal{F}_{0}(k_{n};M_{n})}$, respectively, except
that when $0<q\leq 1$ and $m_{\ast }=1,$ $\underline{\Phi }^{\mathcal{F}%
_{q}(t_{n};M_{n})}$ takes the value $t_{n}^{2}\wedge \frac{\sigma
^{2}r_{M_{n}}}{n}$ instead of $(t_{n}^{2}\vee \frac{\sigma ^{2}}{n})\wedge 
\frac{\sigma ^{2}r_{M_{n}}}{n}$ and $\underline{\Phi }^{\mathcal{F}%
_{q}(t_{n};M_{n})\cap \mathcal{F}_{0}(k_{n};M_{n})}$ is modified the same
way. In the fixed design case, the ranks of the design matrices are
certainly relevant in risk bounds (see, e.g., \cite{Yang1999,
RigolletTsybakov2011}).


\begin{theorem}
\label{UpABC} When $\lambda \geq 5.1\log 2$, the ABC estimator $\hat{f}_{%
\hat{J}}$ simultaneously has the following properties.

\begin{itemize}
\item[(i)] For $\mathcal{F=F}_{q}(t_n;M_{n})$ with $0<q\leq 1,$ or $\mathcal{%
F=F}_{0}(k_{n};M_{n})$ with $1\leq k_{n}\leq M_{n},$ or $\mathcal{F=F}%
_{q}(t_n;M_{n})\cap \mathcal{F}_{0}(k_{n};M_{n})$ with $0<q\leq 1$ and $%
1\leq k_{n}\leq M_{n},$ we have%
\begin{equation*}
\sup_{f_{0}\in \mathcal{F}}E(ASE(\hat{f}_{\hat{J}}))\leq B \Phi ^{\mathcal{F}%
},
\end{equation*}%
where the constant $B$ depends only on $q$ and $\lambda $ for the first and
third cases of $\mathcal{F},$ and depends only on $\lambda $ for the second
case.

\item[(ii)] In general, for an arbitrary ${f}_{0}^{n},$ we have%
\begin{eqnarray*}
& & E(ASE(\hat{f}_{\hat{J}})) \leq B\left( \left\Vert \bar{f}_{J_{M_{n}}}-{f}%
_{0}^{n}\right\Vert _{n}^{2}+\inf_{J_{m}:1\leq m<M_{n}}\left( \left\Vert 
\bar{f}_{J_{m}}-\bar{f}_{J_{M_{n}}}\right\Vert _{n}^{2}+\frac{%
\sigma^{2}r_{J_{m}}}{n} \right. \right. \\
& & \left. \left. +\frac{\sigma^{2}\log (M_{n}\wedge n)}{n} +\frac{%
\sigma^{2}\log {\binom{M_{n}}{m}}}{n}\right) \wedge \frac{\sigma^{2}r_{M_{n}}%
}{n}\right) \wedge B\left( \left( \Vert \bar{f}_{J_{0}}-f_{0}^{n}\Vert
_{n}^{2}+\frac{\sigma^{2}}{n}\right) \wedge \sigma^{2}\right) ,
\end{eqnarray*}
\end{itemize}

where the constant $B$ depends only on $\lambda $.
\end{theorem}


\begin{remark}
In (i), the case $\mathcal{F}=\mathcal{F}_{0}(k_{n};M_{n})$ does not require 
$\max_{1\leq j\leq M_{n}}\sum_{i=1}^{n}x_{i,j}^{2}/n\leq 1$.
\end{remark}

\begin{remark}
In pursuing the best performance in each case of $\mathcal{F}$, the general
risk bound in (ii) reduces to $B\Phi ^{\mathcal{F}}$ plus the approximation
error $d_{n}^{2}(f_{0};\mathcal{F})=\inf_{f\in \mathcal{F}}\Vert
f-f_{0}\Vert _{n}^{2}$.
\end{remark}

For the lower bound results, as before, additional conditions are needed.
Let $\Xi $ denote the design matrix of the full model $J_{M_{n}}$.

{\scshape Assumption SRC$^{\prime }$}: For some $\gamma > 0$, there exist
two positive constants $\underline{a}$ and $\overline{a}$ that do not depend
on $n$ such that for every $\theta $ with $\Vert \theta \Vert _{0}\leq \min
(2\gamma ,M_{n})$, we have 
\begin{equation*}
\underline{a}\Vert \theta \Vert _{2}\leq \frac{1}{\sqrt{n}}\Vert \Xi \theta
\Vert _{2}\leq \overline{a}\Vert \theta \Vert _{2}.
\end{equation*}

\vspace{1pt}This condition is slightly weaker than Assumption 2 in \cite%
{Raskuttietal2010}, which was used to derive minimax lower bounds for $%
0<q\leq 1$.


\begin{theorem}
\label{LowABC} Suppose the noise $\varepsilon $ follows a normal
distribution with mean $0$ and variance $0<\sigma ^{2}<\infty $. For $%
\mathcal{F=F}_{q}(t_{n};M_{n})$ with $0<q\leq 1,$ or $\mathcal{F=F}%
_{0}(k_{n};M_{n})$ with $1\leq k_{n}\leq M_{n},$ or $\mathcal{F=F}%
_{q}(t_{n};M_{n})\cap \mathcal{F}_{0}(k_{n};M_{n})$ with $0<q\leq 1$ and $%
1\leq k_{n}\leq M_{n},$ under Assumption SRC$^{\prime }$ with $\gamma
=m_{\ast }$, or $k_{n},$ or $k_{n}\wedge m_{\ast }$ respectively, we have%
\begin{equation*}
\inf_{\hat{f}}\sup_{f_{0}\in \mathcal{F}}E(ASE(\hat{f}))\geq B^{^{\prime }}%
\underline{\Phi }^{\mathcal{F}},
\end{equation*}%
where the estimator $\hat{f}$ is over all estimators, and the constant $%
B^{^{\prime }}$ depends only on $\underline{a}$ and $\overline{a}$ for the
second case of $\mathcal{F}$ and additionally on $q$ for the first and third
cases of $\mathcal{F}$.
\end{theorem}


\begin{remark}
If {SRC$^{\prime }$ is not satisfied on the set of all the predictors but is
satisfied on a subset of }$M_{0}$ predictors, as long as $\log \frac{M_{n}}{%
m_{\ast }},$ $\log \frac{M_{n}}{k_{n}},$ and $\log \frac{M_{n}}{m_{\ast
}\wedge k_{n}}$ are of the same order as $\log \frac{M_{0}}{m_{\ast }},$ $%
\log \frac{M_{0}}{k_{n}},$ and $\log \frac{M_{0}}{m_{\ast }\wedge k_{n}}$,
respectively, we get the same risk lower rates. When $M_{n}$ is really
large, this relaxation of {SRC$^{\prime }$ can be much less stringent for
application.}
\end{remark}

For the case $q=0,$ the achievability of the upper rate is a direct
consequence of \cite{Yang1999}. The lower rates for $q=0$ and/or $1$ are
given in \cite{RigolletTsybakov2011}, where the satisfiability of the SRC$%
^{^{\prime }}$ is also worked out. Raskutti et al. \cite{Raskuttietal2010},
under the assumption that the rank of the full design matrix is $n,$ derived
the minimax rates of convergence $t_{n}^{q}(\log \left( M_{n}\right)
/n)^{1-q/2}$ for $0<q<1$ in an in-probability sense for linear regression
with fixed design with the $\ell _{q}$-constraint when $M_{n}\gg n$ and $%
M_{n}/(t_{n}^{q}n^{q/2})\geq M_{n}^{\kappa }$ with some $\kappa \in (0,1)$.
From our result, the ABC estimator simultaneously achieves the minimax rates
of convergence for all $0\leq q\leq 1$ and for all $M_{n}\geq 2$ and $t_{n}$
no smaller than order $n^{-1/2},$ and also under the joint constraints when $%
q=0$ and $0<q\leq 1$. We also need to point out that we only work on
estimating the regression mean in this work, but \cite{Raskuttietal2010}
showed that, under additional conditions, these upper rates are also valid
for the estimation of the parameter $\theta $ under the squared error and
verified their minimaxity. Concurrent work by Ye and Zhang \cite{YeZhang2010}
also derived performance bounds on the coefficient estimation that are
optimal in a sense of uniformity over the different designs.

In application, the assumption that $f_{0}\in \mathcal{F}_{q}(t_{n};M_{n})$
or $f_{0}\in \mathcal{F}_{0}(k_{n};M_{n})$ may sometimes be too strong to be
appropriate. Thus, risk bounds that permit model mis-specification, i.e., $%
f_{0}\notin \mathcal{F}_{q}(t_{n};M_{n}),$ are desirable. Part (ii) in the
upper bound theorem (Theorem \ref{UpABC}) shows that the ABC estimator
handles model mis-specification. Indeed, for the different $\ell _{q}$%
-constraints, the risk of the ABC estimator is upper bounded by a multiple
of $d_{n}^{2}(f_{0};\mathcal{F}_{q}(t_{n};M_{n}))$ plus the earlier upper
bounds, respectively. Therefore, model mis-specification or not, our
estimator is minimax rate adaptive over the $\ell _{q,t_{n}}^{M_{n}}$-hulls
without any knowledge about the values of $q,$ $t_{n}$ and $k_{n}$ (as long
as $t_{n}$ is not trivially small).


One limitation of this result, from one theoretical point, is that the
factor is larger than one in front of $d_{n}^{2}(f_{0};\mathcal{F})$. When
the initial estimates need to be obtained based on the same data available,
the multiplying factor being one no longer necessarily has any essential
advantage. However, striving for the right constant is theoretically
attractive when the elements in the dictionary are observed or truly
provided by others.

In that direction, recently, Rigollet and Tsybakov \cite%
{RigolletTsybakov2011}, by considering an estimator based on the
mixing-least-square-estimators algorithm of Leung and Barron \cite%
{LeungBarron2006} with some specific choice of prior probabilities on the
models, have provided in-expectation optimal upper bounds for $\ell _{0}$-
and/or $\ell _{1}$-aggregation. With the power of the oracle inequality (or
the index of resolvability bound), their estimator is shown to be adaptive
over $\ell _{0}$- and $\ell _{1}$-hulls. Their results do not address $\ell
_{q}$-aggregation for $0<q<1.$ We next show that we can have an estimator
that handles all $0\leq q\leq 1$ in generality.

The mixed least squares estimator by the mixing algorithm of Leung and
Barron (2006) is given by 
\begin{equation*}
\hat{f}^{MLS}=\sum_{J\in \Gamma _{n}}w_{J}\hat{Y}_{J}\text{\hspace{0.5cm}
with\hspace{0.5cm}}w_{J}=\frac{\pi _{J}\exp \{-\widehat{R}_{J}/(4\sigma
^{2})\}}{\sum_{J^{^{\prime }}\in \Gamma _{n}}\pi _{J^{^{\prime }}}\exp \{-%
\widehat{R}_{J^{^{\prime }}}/(4\sigma ^{2})\}},
\end{equation*}%
where $\widehat{R}_{J}=n\Vert Y-\hat{Y}_{J}\Vert _{n}^{2}+2r_{J}\sigma
^{2}-n\sigma ^{2}$ is the unbiased risk estimate for $\hat{Y}_{J}$. Let the
prior on model $J$ be chosen as $\pi _{J_{m}}=0.85\left( ((M_{n}-1)\wedge n){%
\binom{M_{n}}{m}}\right) ^{-1}$ for $1\leq m\leq (M_{n}-1)\wedge n$, and $%
\pi _{J_{M_{n}}}=\pi _{J_{0}}=\pi _{\bar{J}}=0.05$.


\begin{theorem}
\label{UpMLS} Suppose $0<\sigma <\infty $ is known. For any $M_{n}\geq 1$, $%
n\geq 1$, the estimator $\hat{f}^{MLS}$ simultaneously has the following
properties.

\begin{itemize}
\item[(i)] For any $0<q\leq 1$, $t_n>0$, 
\begin{eqnarray*}
E(ASE(\hat{f}^{MLS})) &\leq &d_n^{2}(f_{0};\mathcal{F}_{q}(t_n;M_{n})) \\
&+& B_{1}\left\{ 
\begin{array}{ll}
\frac{\sigma^{2}r_{M_{n}}}{n}, & \text{if $m_{\ast }=M_{n}\wedge n$,} \\ 
\sigma^{2-q}t_n^{q}\left( \frac{1+\log \frac{M_{n}}{(nt_n^{2}\tau )^{q/2}}}{n%
}\right) ^{1-q/2}\wedge \frac{\sigma^{2}r_{M_{n}}}{n}, & \text{if $1<m_{\ast
}<M_{n} \wedge n$,}%
\end{array}%
\right.
\end{eqnarray*}%
and 
\begin{eqnarray*}
E(ASE(\hat{f}^{MLS})) &\leq &\left(d_n^{2}(f_{0};\mathcal{F}_{q}(t_n;M_{n}))+%
\frac{B_{1}\left( \sigma^{2}(1+\log M_{n})\wedge \sigma^{2}r_{M_{n}}\right) 
}{n}\right) \\
& & \wedge \left( \Vert \bar{f}_{J_{0}} - f_{0}^n \Vert _{n}^{2}+\frac{%
\widetilde{B}_{1}\sigma^{2}}{n}\right) ,\text{ if $m_{\ast }=1$ }.
\end{eqnarray*}

\item[(ii)] For $1\leq k_{n}\leq M_{n}$, 
\begin{equation*}
E(ASE(\hat{f}^{MLS}))\leq d_n^{2}(f_{0};\mathcal{F}_{0}(k_{n};M_{n}))+B_{2}%
\left( \frac{\sigma^{2}k_{n}\left( 1+\log \frac{M_{n}}{k_{n}}\right) }{n}%
\wedge \frac{\sigma^{2}r_{M_{n}}}{n}\right) .
\end{equation*}

\item[(iii)] For any $0<q\leq 1$, $t_n>0$, and $1\leq k_{n}\leq M_{n}$, 
\begin{eqnarray*}
E(ASE(\hat{f}^{MLS})) &\leq &d_n^{2}(f_{0};\mathcal{F}_{q}(t_n;M_{n})\cap 
\mathcal{F}_{0}(k_{n};M_{n})) \\
&&+B_{3}\left\{ 
\begin{array}{ll}
\frac{\sigma ^{2}k_{n}\left( 1+\log \frac{M_{n}}{k_{n}}\right) }{n}\wedge 
\frac{\sigma ^{2}r_{M_{n}}}{n} & \text{if $m_{\ast }>k_{n}$,} \\ 
\sigma ^{2-q}t_n^{q}\left( \frac{1+\log \frac{M_{n}}{(nt_n^{2}\tau )^{q/2}}}{%
n}\right) ^{1-q/2}\wedge \frac{\sigma ^{2}r_{M_{n}}}{n} & \text{if $%
1<m_{\ast }\leq k_{n}$. }%
\end{array}%
\right.
\end{eqnarray*}%
and 
\begin{eqnarray*}
E(ASE(\hat{f}^{MLS})) &\leq &\left( d_n^{2}(f_{0};\mathcal{F}%
_{q}(t_n;M_{n})\cap \mathcal{F}_{0}(k_{n};M_{n}))+\frac{B_{3}\left( \sigma
^{2}(1+\log M_{n})\wedge \sigma ^{2}r_{M_{n}}\right) }{n}\right) \\
&&\wedge \left( \Vert \bar{f}_{J_{0}} - f_{0}^n \Vert _{n}^{2}+\frac{%
\widetilde{B}_{3}\sigma ^{2}}{n}\right) ,\text{ if $m_{\ast }=1$ }.
\end{eqnarray*}

\item[(iv)] For every $f_{0},$ we have 
\begin{equation*}
E(ASE(\hat{f}^{MLS}))\leq B_{4}\sigma ^{2}.
\end{equation*}
\end{itemize}

For these cases, the constants $\widetilde{B}_{1}$, $B_{2}$, $\widetilde{B}%
_3 $ and $B_{4}$ are pure constants, and $B_{1}$ and $B_{3}$ depend on $q$.
\end{theorem}

\begin{remark}
From (ii) above, by taking $k_{n}=1,$ we have 
\begin{equation*}
E(ASE(\hat{f}^{MLS}))\leq \inf_{1\leq j\leq M_{n}}\Vert
f_{j}^{n}-f_{0}^{n}\Vert _{n}^{2}+B_{2}\left( \frac{\sigma ^{2}\left( 1+\log
M_{n}\right) }{n}\wedge \frac{\sigma ^{2}r_{M_{n}}}{n}\right) ,
\end{equation*}%
where $f_{j}^{n}=(x_{1,j},\ldots ,x_{n,j})^{\prime }$. Thus, we have
achieved aggregation for adaptation as well under the fixed design.
\end{remark}

The risk upper bounds above when $q$ is restricted to be either 0 or 1 or
under both constraints are already given in Theorem 6.1 of \cite%
{RigolletTsybakov2011}. The first four cases given there are clearly
reproduced here (note that their cases 3 and 1 are just special case and
immediate consequence, respectively, of their case 4, given in our bound in
(ii)). Their case 5, a sparse aggregation with $k_{n}$ estimates as studied
in \cite{Yang2004} (page 36) and \cite{Lounici2007} (called $D$-convex
aggregation) is implied by our bound in (iii) with $q$ taken to be 1. In the
case $q=1,$ a minor difference is that if $\Vert \bar{f}_{J_{0}}-f_{0}^{n}%
\Vert _{n}^{2}$ happens to be of a smaller order than $t_{n}\left( \frac{%
1+\log \frac{M_{n}}{(nt_{n}^{2})^{1/2}}}{n}\right) ^{1/2}\wedge \frac{%
r_{M_{n}}}{n},$ then our risk bound in (iii) yields a faster rate of
convergence. In addition, our inclusion of the full projection model among
the candidates guarantees that the risk of our estimator is always bounded,
which is not true for the estimator in \cite{RigolletTsybakov2011}. Our main
contribution here is to handle adaptive $\ell _{q}$-aggregation for the
whole range of $q$ between 0 and 1. Note that the upper bounds in the above
theorem have already been shown to be minimax-rate optimal under the
conditions in Theorem \ref{LowABC}.

\subsection{A comment on the model selection and model mixing approaches}

\vspace{1pt}From the risk bounds in the previous subsection, we see that the
model mixing approach leads to the optimal constant 1 in front of the
approximation error $d_{n}^{2}(f_{0};\mathcal{F})$ for the three choices of $%
\mathcal{F}$, which is not the case for the model selection based estimator.
However, the model selection approach may also have its own advantages.

From the proof of Theorem \ref{UpABC} and proof of Theorem 1 in \cite%
{Yang1999}, besides the given risk bounds, we also have a general
in-probability bound of the form: for any $x>0,$ there are constants $c,$ $%
c^{^{\prime }}$ (absolute constants) and $c^{^{\prime \prime }}$ (depending
on $\lambda $ and $\sigma ^{2}$) such that 
\begin{equation*}
P\left( \frac{ASE(\hat{f}_{\hat{J}})+\frac{\lambda \sigma ^{2}C_{\hat{J}}}{n}%
}{R_{n}(f_{0})}\geq c+x\right) \leq c^{^{\prime }}\exp \left( -c^{^{\prime
\prime }}x\right) ,
\end{equation*}%
where $R_{n}(f_{0})=$ $\inf_{J\in \Gamma _{n}}\left( \left\Vert \bar{f}_{J}-{%
f}_{0}^{n}\right\Vert _{n}^{2}+\frac{\sigma ^{2}r_{J}}{n}+\frac{\lambda
\sigma ^{2}C_{J}}{n}\right) $ is an index of resolvability, which
specializes to the upper bounds in (i) and (ii) of Theorem \ref{UpABC},
respectively in those situations. Thus, we know that not only $ASE(\hat{f}_{%
\hat{J}})$ is at order $R_{n}(f_{0})$ with upper deviation probability
exponentially small (in $x$), but also the complexity of the selected model, 
$\frac{\lambda \sigma ^{2}C_{\hat{J}}}{n}$, is upper bounded in probability
in the same way as well. In particular, for estimating a linear regression
function with the soft or hard (or both) constraint(s) on the coefficients,
the ABC estimator converges at rate $\frac{m_{\ast }\left( 1+\log \frac{M_{n}%
}{m_{\ast }}\right) }{n}\wedge \frac{r_{M_{n}}}{n}$ both in expectation and
with upper deviation probability exponentially small, where $m_{\ast }$ is
the corresponding effective model size in each case. Furthermore, the rank
(the actual number of free-parameters) of the model selected by ABC is right
at order $m_{\ast }\wedge r_{M_{n}}$ with exception probability
exponentially small.

For model mixing estimators based on exponential weighting, however, to our
knowledge, no result has shown that their losses are generally at the
optimal rate in probability. In fact, a negative result is given in \cite%
{Audibert2007} that shows that an exponential weighting based estimator 
\textit{optimal} for aggregation for adaptation (i.e., its risk regret, or
the expected excessive loss, is of order $\frac{\log M_{n}}{n}$) is
necessarily \textit{sub-optimal} in probability (with a non-vanishing
probability its excessive loss is at least at the much larger order of $%
\sqrt{\frac{\log M_{n}}{n}}$) in certain settings.

Thus, we tend to believe that both the model selection and model mixing
approaches have their own theoretical strengths in different ways.

\vspace*{0.1in}

\subsection{When $\protect\sigma$ is unknown}

\label{sigmaunknown}

Needless to say, the assumption that $\sigma $ is fully known is
unrealistic. When $\sigma $ is unknown but is upper bounded by a known
constant $\overline{\sigma }>0$, similar results for rate of convergence can
be obtained with a model selection rule different from ABC.

For this situation, Yang \cite{Yang1999} proposed the ABC$^{\prime }$
criterion: 
\begin{equation*}
ABC^{\prime }(J_{m})=\left( 1+\frac{2r_{J_{m}}}{n-r_{J_{m}}}\right) \left(
\sum_{i=1}^{n}(Y_{i}-\hat{Y}_{i,J_{m}})^{2}+\lambda \overline{\sigma }%
^{2}C_{J_{m}}\right) ,
\end{equation*}%
which is a modification of Akaike's FPE criterion \cite{Akaike1970}. We
define $ABC^{\prime }(\bar{J})=\left( 1+2n\right) \lambda \overline{\sigma }%
^{2}C_{\bar{J}}$ and $ABC^{\prime }(J_{0})=\left( 1+\frac{2}{n-1}\right)
\left( \sum_{i=1}^{n}(Y_{i}-\bar{Y})^{2}+\lambda \overline{\sigma }%
^{2}C_{J_{0}}\right) .$ The list of candidate models and complexity
assignments need to be different for the different situations, as described
below.

\begin{enumerate}
\item When $M_{n}\leq n/2,$ all the subset models, $J_0$ and $\bar{J}$ are
considered with the complexity $C_{J_{m}}= - \log 0.85+\log (M_{n}-1)+\log {%
\binom{M_n }{m}}$ for a subset model with $m<M_{n}$, $C_{J_{M_{n}}}=C_{J_0}
= C_{\bar{J}}=- \log 0.05$.

\item When $M_{n}>n/2$ and $r_{M_{n}}\geq n/2,$ we only consider models with
size $m\leq n/2$, the model $J_0$ and the model $\bar{J}$. Then we assign
the complexity $C_{J_{m}}=- \log 0.8+\log (\lfloor n/2\rfloor )+\log {\binom{%
M_n }{m}} $ for a subset model, $C_{J_0} =C_{\bar{J}}=- \log 0.1$.

\item When $M_{n}>n/2$ and $r_{M_{n}}<n/2,$ we only consider models with
size $m\leq n/2,$ the full model $J_{M_{n}}$, the null model $J_0$, and the
model $\bar{J}$. We assign the complexity $C_{J_{m}}=- \log 0.85+\log
(\lfloor n/2\rfloor )+\log {\binom{M_n }{m}}$ for a subset model, $%
C_{J_{M_{n}}}=C_{J_0} = C_{\bar{J}}=-\log0.05$.
\end{enumerate}

In any of the cases above, let $\Gamma _{n}^{\prime }$ denote the set of all
the models considered. The model chosen by the ABC$^{\prime }$ is 
\begin{equation*}
\hat{J}^{\prime }=\arg \min_{J\in \Gamma _{n}^{\prime }}ABC^{\prime }(J),
\end{equation*}%
producing the ABC$^{\prime }$ estimator $\hat{f}_{\hat{J}^{\prime }}=\hat{Y}%
_{\hat{J}^{\prime }}$.


\begin{theorem} \label{UpABCunknown} 
When $\lambda \geq 40 \log 2$, the ABC$^{\prime }$
estimator $\hat{f}_{\hat{J}^{\prime }}$ simultaneously has the following
properties.

\begin{itemize}
\item[(i)] For $\mathcal{F=F}_{q}(t_{n};M_{n})$ with $0<q\leq 1,$ or $%
\mathcal{F=F}_{0}(k_{n};M_{n})$ with $1\leq k_{n}\leq M_{n},$ or $\mathcal{%
F=F}_{q}(t_{n};M_{n})\cap \mathcal{F}_{0}(k_{n};M_{n})$ with $0<q\leq 1$ and 
$1\leq k_{n}\leq M_{n},$ we have%
\begin{equation*}
\sup_{f_{0}\in \mathcal{F}}E(ASE(\hat{f}_{\hat{J}^{\prime }}))\leq B\Phi ^{%
\mathcal{F}},
\end{equation*}%
where the constant $B$ depends only on $q$, $\lambda $, $\overline{\sigma },$
$\sigma $ for the first and third cases of $\mathcal{F},$ and depends only
on $\lambda $, $\overline{\sigma },$ $\sigma $ for the second case.

\item[(ii)] In general, for an arbitrary ${f}_{0}^{n},$ we have%
\begin{eqnarray*}
\hspace*{-0.5in} &&E(ASE(\hat{f}_{\hat{J}^{\prime }})) \\
\hspace*{-0.5in} &\leq &B\left( \left\Vert \bar{f}_{J_{M_{n}}}-{f}%
_{0}^{n}\right\Vert _{n}^{2}+\inf_{J_{m}:1\leq m<M_{n}}\left( \left\Vert 
\bar{f}_{J_{m}}-\bar{f}_{J_{M_{n}}}\right\Vert _{n}^{2}+\frac{\sigma
^{2}r_{J_{m}}}{n}+\frac{\sigma ^{2}\log (M_{n}\wedge n)}{n}\right. \right. \\
&&+\left. \left. \frac{\sigma ^{2}\log {\binom{M_{n}}{m}}}{n}\right) \wedge 
\frac{\sigma ^{2}r_{M_{n}}}{n}\right) \wedge B\left( \left( \Vert \bar{f}%
_{J_{0}}-f_{0}^{n}\Vert _{n}^{2}+\frac{\sigma ^{2}}{n}\right) \wedge \sigma
^{2}\right) ,
\end{eqnarray*}
\end{itemize}

where the constant $B$ depends only on $\lambda ,$ $\overline{\sigma },$ $%
\sigma$.
\end{theorem}

\begin{remark}
For the results in (i), as seen before, when $f_{0}$ is not in the
respective class of linear combinations, an obvious modification is needed
by adding a multiple of the approximation error $d_{n}^{2}(f_{0};\mathcal{F}%
) $ in the risk bound.
\end{remark}

\vspace{1pt}When $0<\sigma <\infty $ is fully unknown, a model selection
method by Baraud, Giraud and Huet \cite{Baraudetal2009} can be used to
obtain results on $\ell _{q}$-regression.

They consider a different modification of the FPE criterion \cite{Akaike1970}%
: 
\begin{equation*}
BGH(J_{m})=\left( 1+\frac{pen(J_{m})}{n-r_{J_{m}}}\right) \left(
\sum_{i=1}^{n}(Y_{i}-\hat{Y}_{i,J_{m}})^{2}\right) ,
\end{equation*}%
where $pen(J_{m})$ is a penalty assigned to the model $J_{m}.$ They devise a
new form for $pen(J_{m})$ (Section 4.1 in \cite{Baraudetal2009}) to yield a
nice oracle inequality (Corollary 1) that does not require any knowledge of $%
\sigma ,$ but at the expense of excluding some large models in the
consideration. When $M_{n}\leq \left( n-7\right) \wedge {\varsigma n}$ for
some $0<{\varsigma <1,}$ we consider all subset models in the model
selection process. When $M_{n}$ is large, we consider only subset models
with $n-r_{J_{m}}\geq 7$ and $m\vee \log {\binom{M_{n}}{m}\leq \varsigma n}$
for a fixed $0<{\varsigma <1.}$ Combining the tools developed in this and
their papers, we have the following result.

\begin{theorem}
\label{Baraud} The BGH estimator $\hat{f}_{\hat{J}}$ has the following
properties.

\begin{itemize}
\item[(i)] When $M_{n}\leq \left( n-7\right) \wedge {\varsigma n},$ for $%
\mathcal{F=F}_{q}(t_{n};M_{n})$ with $0<q\leq 1,$ or $\mathcal{F=F}%
_{0}(k_{n};M_{n})$ with $1\leq k_{n}\leq M_{n},$ or $\mathcal{F=F}%
_{q}(t_{n};M_{n})\cap \mathcal{F}_{0}(k_{n};M_{n})$ with $0<q\leq 1$ and $%
1\leq k_{n}\leq M_{n},$ we have%
\begin{equation*}
\sup_{f_{0}\in \mathcal{F}}E(ASE(\hat{f}_{\hat{J}}))\leq B\Phi ^{\mathcal{F}%
},
\end{equation*}%
where the constant $B$ depends only on $q$ and ${\varsigma }$ for the first
and third cases of $\mathcal{F},$ and depends on ${\varsigma }$ for the
second case.

\item[(ii)] For a general $M_{n},$ if $m_{\ast }$ satisfies $m_{\ast }\leq
n-7$ and $m_{\ast }\vee \log {\binom{M_{n}}{m_{\ast }}\leq \varsigma n,}$ we
have%
\begin{equation*}
\sup_{f_{0}\in \mathcal{F}_{q}(t_{n};M_{n})}E(ASE(\hat{f}_{\hat{J}}))\leq
B\left\{ 
\begin{array}{ll}
\sigma ^{2-q}t_{n}^{q}\left( \frac{1+\log \frac{M_{n}}{(nt_{n}^{2}\tau
)^{q/2}}}{n}\right) ^{1-q/2} & \text{if $m_{\ast }>1$,} \\ 
t_{n}^{2}\vee \frac{\sigma ^{2}}{n} & \text{if $m_{\ast }=1$,}%
\end{array}%
\right.
\end{equation*}%
where $B$ depends only on $q$ and ${\varsigma .}$ If $k_{n}$ satisfies $%
k_{n}\leq n-7$ and $k_{n}\vee \log {\binom{M_{n}}{k_{n}}\leq \varsigma n,}$
we have%
\begin{equation*}
\sup_{f_{0}\in \mathcal{F}_{0}(k_{n};M_{n})}E(ASE(\hat{f}_{\hat{J}}))\leq
B^{^{\prime }}\frac{\sigma ^{2}k_{n}\left( 1+\log \frac{M_{n}}{k_{n}}\right) 
}{n},
\end{equation*}%
where $B^{^{\prime }}$ is a constant that depends only on ${\varsigma }$.
\end{itemize}
\end{theorem}


\begin{remark}
As before, when $f_{0}$ is not in the respective class, a multiple of the
approximation error $d_{n}^{2}(f_{0};\mathcal{F})=\inf_{f\in \mathcal{F}%
}\Vert f-f_{0}\Vert _{n}^{2}$ needs to be added in the aggregation risk
bound.
\end{remark}

\vspace{1pt}From the above theorem, we see that when $\sigma $ is fully
unknown, as long as $M_{n}\leq \left( n-7\right) \wedge {\varsigma n}$ for
some $0<{\varsigma <1,}$ similar risk bounds to those in Theorem \ref%
{UpABCunknown} for $\ell _{q}$-regression hold. However, when $M_{n}$ is
larger, the previous risk bounds are seriously compromised: 1) the possible
improvement in risk due to low rank of the full model is no longer
guaranteed; 2) the previous upper rates determined by the effective model
size $m_{\ast }$ or $k_{n}$ are valid only when those model sizes are not
excluded from consideration by the BGH criterion; 3) The risk is no longer
guaranteed to be always uniformly bounded. Indeed, due to the restriction on
the model sizes to be considered, the final risk here can be arbitrarily
large. It turns out that this last aspect is not due to technical deficiency
in the analysis, but it is a necessary price to pay for not knowing $\sigma $
at all (see \cite{Verzelen2010}).


\section{Discussion}

\label{discussion}

Since early 1990s, sparse estimation has been recognized as an important
tool for multi-dimensional function estimation. Emergence of
high-dimensional statistical problems in the information age has prompted an
increasing attention on the topic from theoretical, computational and
applied perspectives. We focus only on a theoretical standpoint in the
discussion below.

To our knowledge, several lines of research on sparse function estimation in
1990s produced theoretical foundations that still provide essential
understandings on ways to explore sparsity and associated price to pay when
pursuing sparse estimation from minimax perspectives. It has been discovered
that for some function classes, sparse representations (in contrast to
traditional full approximation) result in faster rates of convergence, which
alleviate the curse of dimensionality when the problem size is large. Such
function classes include, for example, Besov classes (e.g., \cite%
{Donohoetal1996}), Jones-Barron classes (\cite{Barron1994,
JuditskyNemirovski2000}) and may also be defined directly in terms of sparse
approximation (e.g., \cite{YangBarron1998}, Section III.D). Regarding
methods to achieve the optimal sparse estimation, wavelet thresholding with
one or more orthonormal dictionaries and model selection with a descriptive
complexity penalty term added to the sum of negative maximized likelihood
(or a general contrast function) and a multiple of the model dimension have
yielded successful theoretical advancements. Oracle inequalities/index of
resolvability bounds have been derived that readily give minimax-rate
adaptive estimators for various scenarios. In linear representation, $\ell
_{1}$-constraints on the coefficients have been long known to be associated
with fast rate of convergence for both orthogonal and non-orthogonal bases
by model selection or aggregation methods, as mentioned in the introduction
of this paper.

It is worth noticing that these research works usually target nonparametric
settings. In the past few years, the situation of a large number of
naturally observed predictors has attracted much attention, shifting the
focus to much simpler linear modeling. As pointed out earlier, the work in
the 1990s on model selection has direct implications for the
high-dimensional linear regression. For example, if the sum of the absolute
values of the linear coefficients is bounded ($\ell _{1}$-constraint), then
the rate of convergence is bounded by $\left( \log n/n\right) ^{1/2}$ as
long as $M_{n}$ increases only polynomially in $n.$ If only $k_{n}$ terms
have non-zero coefficients ($\ell _{0}$-constraint), then the rate of
convergence is of order $k_{n}(1+\log (M_{n}/k_{n}))/n$ based on model
selection with mild conditions on the predictors. However, such subset
selection based estimators pose computational challenges in real
applications.

In the direction of using the $\ell _{1}$-constraints in constructing
estimators, algorithmic and theoretical results have been well developed.
Both the Lasso and the Dantzig selector have been shown to achieve the rate $%
k_{n}\log (M_{n})/n$ under different conditions on correlations of
predictors and the hard sparsity constraint on the linear coefficients (see 
\cite{GeerBuhlmann2009} for a discussion about the sufficient conditions for
deriving oracle inequalities for the Lasso). Our upper bound results do not
require any of those conditions, but we do assume the sparse Riesz condition
for deriving the lower bounds. Computational issues aside, we have seen that
the approach of model selection/combination with descriptive complexity
penalty has provided the most general adaptive estimators that automatically
exploit sparsity natures of the target function in terms of linear
approximations subject to $\ell _{q}$-constraints.

Donoho and Johnstone \cite{DonohoJohnstone1994} derived insightful general
asymptotic minimax risk expressions for estimating the mean vector in $\ell
_{q}$-balls ($0<q<\infty $) under $\ell _{p}$ loss ($p\geq 1$) in a Gaussian
sequence framework. The work by Raskutti et al. \cite{Raskuttietal2010} and
by Rigollet and Tsybakov \cite{RigolletTsybakov2011} are directly related to
our work in the fixed design case. The former successfully obtains optimal
non-adaptive in-probability loss bounds for their main scenario that $M_{n}$
is much larger than $n$ for general $0\leq q\leq 1$ when the true regression
function is assumed to be in the $\ell _{q,t_{n}}^{M_{n}}$-hull. In
contrast, our estimators are adaptive and the risk bounds hold without
restrictions on $M_{n}$ or the \textquotedblleft norm\textquotedblright\
parameter $t_{n},$ also allowing the true regression function to be really
arbitrary. The work of Rigollet and Tsybakov \cite{RigolletTsybakov2011}
nicely shows the adaptive aggregation capability of model mixing over $\ell
_{0}$ and $\ell _{1}$-balls. Our results are valid over the whole range of $%
0\leq q\leq 1.$ For lower bounds, our formulation is somewhat different from
theirs. In addition, unlike those results, we have also provided results
when the error variance is unknown but upper bounded by a known constant or
fully unknown. Furthermore, our model selection based estimators have
optimal convergence rates also in terms of upper deviation probability,
which may not hold for the model mixing estimators. We need to point out
that both \cite{Raskuttietal2010} and \cite{RigolletTsybakov2011} have given
results on related problems that we do not address in this work.

In our results, the effective model size $m_{\ast }$ (as defined in Section %
\ref{aninsight}) plays a key role in determining the minimax rate of $\ell
_{q}$-aggregation for $0<q\leq 1.$ With the extended definition of the
effective model size $m_{\ast }$ to be simply the number of nonzero
components $k_{n}$ when $q=0$ and re-defining $m_{\ast }$ to be $m_{\ast
}\wedge k_{n}$ under both $\ell _{q}$- ($0<q\leq 1$) and $\ell _{0}$%
-constraints, the minimax rate of aggregation is unified to be the simple
form $1\wedge \frac{m_{\ast }\left( 1+\log \left( \frac{M_{n}}{m_{\ast }}%
\right) \right) }{n}$.

Risk bounds for selection/mixing least squares estimators from a countable
collection of linear models (such as given in \cite{Yang1999,
LeungBarron2006}), together with sparse approximation error bounds, are
essential for our approach to devise minimax optimal sparse estimation for
fixed design. When the predictors are taken as some initial estimates, the
selection/mixing methods can be regarded as aggregation methods with the
risk bounds as aggregation risk bounds. In a strict sense, however, these
results are not totally satisfactory for at least two reasons. First, the
evaluation of performance only at the design points that have been seen
already has limited value: i) The strengths of the candidate procedures may
not be reflected at all on such a measure; ii) A small ASE on the design
points does not mean good behaviors on future predictor values. Second, when
the initial estimates are not given (which is almost always the case), to
combine arbitrary estimators, data splitting is typically necessary to come
up with the candidate estimates and use the rest of the sample for weight
assignment. Then, the final risk bounds, unfortunately, depend on how the
data are split. In contrast, for the random design case, this is not an
issue. We have also seen that because ASE cares only about the performance
at the design points, given the i.i.d. normal error assumption, there is
absolutely no condition needed on the true regression function, as pointed
out in a remark to Theorem 1 in \cite{Yang1999}. For random design, however,
we have made the sup-norm bound assumption, but the risk bounds guarantee
optimal future performance as long as the sampling distribution is unchanged.

Regarding aggregation, we notice that the $\ell _{q}$-aggregation includes
as special cases the state-of-art aggregation problems, namely aggregation
for adaptation, convex and $D$-convex aggregations, linear aggregation, and
subset selection aggregation, and all of them can be defined (or essentially
so) by considering linear combinations under $\ell _{0}$- and/or $\ell _{1}$%
-constraints. Our investigation provides optimal rates of aggregation, which
not only agrees with (and, in some cases, improves over) previous findings
for the mostly studied aggregation problems, but also holds for a much
larger set of linear combination classes. Indeed, we have seen that $\ell
_{0}$-aggregation includes aggregation for adaptation over the initial
estimates (or model selection aggregation) ($\ell _{0}(1)$-aggregation),
linear aggregation when $M_{n}\leq n$ ($\ell _{0}(M_{n})$-aggregation), and
aggregation to achieve the best performance of linear combination of $k_{n}$
estimates in the dictionary for $1<k_{n}<M_{n}$ (sometimes called subset
selection aggregation) ($\ell _{0}(k_{n})$-aggregation). When $M_{n}$ is
large, aggregating a subset of the dictionary under a $\ell _{q}$-constraint
for $0<q\leq 1$ can be advantageous, which is just $\ell _{0}(k_{n})\cap
\ell _{q}(t_{n})$-aggregation. Since the optimal rates of aggregation as
defined in \cite{Tsybakov2003} can differ substantially in different
directions of aggregation and typically one does not know which direction
works the best for the unknown regression function, multi-directional or
universal aggregation is important so that the final estimator is
automatically conservative and aggressive, whichever is better (see \cite%
{Yang2004}). Our aggregation strategy is indeed multi-directional, achieving
the optimal rates over all $\ell _{q}$-aggregation for $0\leq q\leq 1$ and $%
\ell _{0}\cap \ell _{q}$-aggregation for all $0<q\leq 1.$

One interesting observation is that aggregation for adaptation is
essentially a special case of $\ell_{q}$-aggregation, yet our way of
achieving the simultaneous $\ell_{q}$-aggregation is by methods of
aggregation for adaptation through model selection/combination.

Aggregation of estimates and regression estimation problems are closely
related. For aggregation, besides that the predictors to be aggregated are
from some initial estimations (and thus are not directly observed), the
emphases are: i) One is unwilling to make assumptions on relationships
between the initial estimates so that they can have arbitrary dependence;
ii) One is unwilling to make specific assumptions on the true regression
function beyond that it is uniformly bounded and hence allow model
mis-specification. In this game, there is little interest on the true or
optimal coefficients in the representation of the regression function in
terms of the initial estimates.

Obviously, there are other directions of aggregation that one may pursue.
The $\ell _{q}$-aggregation strategy that relies on aggregating subset
choices of the initial estimates, as in \cite{Yang2004}, while producing the
most general aggregation risk bounds so far, follows a global aggregation
paradigm, i.e., the linear coefficients are globally determined. It is
conceivable that sometimes localized weights may provide better
estimation/prediction performance (see, e.g., \cite{Yang2008}). Much more
work is needed here to result in practically effective localized aggregation
methods.

Aggregation of estimates, as an important step in combining statistical
procedures, has proven to bring theoretically elegant and practically
feasible methods for regression estimation/prediction. It is an important
vehicle to share strengths of different function estimation methodologies to
produce adaptively optimal and robust estimators that work well under
minimal conditions. Aggregation by mixing certainly cannot replace model
selection when selection of an estimator among candidates or a set of
predictors is essential for interpretation or business/operational decisions.

Our focus in this work is of a theoretical nature to provide an
understanding of the fundamental theoretical issues about $\ell _{q}$%
-aggregation or linear regression under $\ell _{q}$-constraints.
Computational aspects will be studied in the future.


\section{General oracle inequalities for random design}

\label{generaloracle}

Consider the setting in Section 3.2.

\begin{theorem}
\label{Oracle} Suppose $A_{\mathbf{E}-\mathbf{G}}$ holds for the \textbf{E}-%
\textbf{G} strategy, respectively. Then, the following oracle inequalities
hold for the estimator $\hat{f}_{F_{n}}$.

(i) For \textbf{T-C} and \textbf{T-Y} strategies, 
\begin{eqnarray*}
&&R(\hat{f}_{F_{n}};f_{0};n) \\
&\leq &c_{0}\inf_{1\leq m\leq M_{n}\wedge n}\left(
c_{1}\inf_{J_{m}}d^{2}(f_{0};\mathcal{F}_{J_{m}})+c_{2}\frac{m}{n_{1}}+c_{3}%
\frac{1 + \log {\binom{M_{n}}{m}+}\log (M_{n}\wedge n)-\log (1-p_{0})}{%
n-n_{1}}\right) \\
&&\wedge c_{0}\left( \Vert f_{0}\Vert ^{2}+c_{3}\frac{1-\log p_{0}}{n-n_{1}}%
\right) ,
\end{eqnarray*}%
where $c_{0}=1$, $c_{1}=c_{2}=C_{L,\sigma}$, $c_{3}=\frac{2}{\lambda _{C}}$
for the \textbf{T-C} strategy; $c_{0}=C_{Y}$, $c_{1}=c_{2}=C_{L,\sigma}$, $%
c_{3}=\sigma^{2}$ for the \textbf{T-Y} strategy.

(ii) For \textbf{AC-C} and \textbf{AC-Y} strategies, 
\begin{eqnarray*}
&&R(\hat{f}_{F_{n}};f_{0};n) \\
&\leq &c_{0}\inf_{1\leq m\leq M_{n}\wedge n}\left( R(f_{0},m,n)+c_{2}\frac{m%
}{n_{1}}+c_{3}\frac{1+\log {\binom{M_{n}}{m}+}\log (M_{n}\wedge n)-\log
(1-p_{0})}{n-n_{1}}\right) \\
&&\wedge c_{0}\left( \Vert f_{0}\Vert ^{2}+c_{3}\frac{1-\log p_{0}}{n-n_{1}}%
\right) ,
\end{eqnarray*}%
where 
\begin{equation*}
R(f_{0},m,n)=c_{1}\inf_{J_{m}}\inf_{s\geq 1}\left( d^2 (f_0; \mathcal{F}%
^L_{J_m, s}) +2 c_{3}\frac{\log (1+s)}{n-n_{1}}\right) ,
\end{equation*}%
and $c_{0}=c_{1}=1$, $c_{2}=8c(\sigma^2 + 5 L^2)$, $c_{3}=\frac{2}{\lambda
_{C}} $ for the \textbf{AC-C} strategy; $c_{0}=C_{Y}$, $c_{1}=1$, $%
c_{2}=8c(\sigma^2 + 5 L^2)$, $c_{3}=\sigma^{2}$ for the \textbf{AC-Y}
strategy.
\end{theorem}

From the theorem, the risk $R(\hat{f}_{F_{n}};f_{0};n)$ is upper bounded by
a multiple of the best trade-off of the different sources of errors
(approximation error, estimation error due to estimating the linear
coefficients, and error associated with searching over many models of the
same dimension). For a model $J,$ let $IR(f_{0};J)$ generically denote the
sum of these three sources of errors. Then, the best trade-off is $%
IR(f_{0})=\inf_{J}IR(f_{0};J),$ where the infimum is over all the candidate
models. Following the terminology in \cite{BarronCover1991}, $IR(f_{0})$ is
the so-called index of resolvability of the true function $f_{0}$ by the
estimation method over the candidate models. We call $IR(f_{0};J)$ the index
of resolvability at model $J.$ The utility of the index of resolvability is
that for $f_{0}$ with a given characteristic, an evaluation of the index of
resolvability at the best $J$ immediately tells us how well the unknown
function is \textquotedblleft resolved\textquotedblright\ by the estimation
method at the current sample size. Thus, accurate index of resolvability
bounds often readily show minimax optimal performance of the model selection
based estimator.

\begin{proof}
(i) For the \textbf{T-C} strategy, 
\begin{eqnarray*}
&&R(\hat{f}_{F_{n}};f_{0};n) \\
&\leq &\inf_{1\leq m\leq M_{n}\wedge n}\left\{ C_{L,\sigma
}\inf_{J_{m}}d^{2}(f_{0}; \mathcal{F}_{J_{m}})+C_{L,\sigma }\frac{m}{n_1}+%
\frac{2}{\lambda _{C}} \left( \frac{ \log (M_{n}\wedge n) + \log {\binom{%
M_{n}}{m}} - \log (1-p_{0}) }{n - n_1} \right) \right\} \\
&&\wedge \left\{ \Vert f_{0}\Vert ^{2}-\frac{2}{\lambda _{C}}\frac{\log p_{0}%
}{n - n_1}\right\} .
\end{eqnarray*}

For the \textbf{T-Y} strategy, 
\begin{eqnarray*}
&&R(\hat{f}_{F_{n}};f_{0};n) \\
&\leq &C_{Y}\inf_{1\leq m\leq M_{n}\wedge n}\left\{ C_{L,\sigma
}\inf_{J_{m}}d^{2}(f_{0};\mathcal{F}_{J_{m}})+C_{L,\sigma }\frac{m}{n_1}%
+\sigma^{2} \left( \frac{1 + \log (M_{n}\wedge n) + \log {\binom{M_{n}}{m}}
- \log (1-p_{0}) }{n - n_1} \right) \right\} \\
&&\wedge C_{Y}\left\{ \Vert f_{0}\Vert ^{2}+\sigma^{2}\frac{1-\log p_{0}}{n
- n_1}\right\} .
\end{eqnarray*}

(ii) For the \textbf{AC-C} strategy, 
\begin{eqnarray*}
&&R(\hat{f}_{F_{n}};f_{0};n) \\
&\leq &\inf_{1\leq m\leq M_{n} \wedge n}\left\{ \inf_{J_{m}}\inf_{s\geq
1}\left( d^2 (f_0; \mathcal{F}^L_{J_m, s})+c(2 \sigma^{\prime} + H)^2\frac{m%
}{n_1}+\frac{2}{\lambda _{C}}\left( \frac{ \log (M_{n}\wedge n)+\log {\binom{%
M_{n}}{m}} -\log (1-p_{0}) }{n - n_1} \right. \right. \right. \\
& & \left. \left. \left. + \frac{2\log (1+s)}{n - n_1}\right) \right)
\right\} \wedge \left\{ \Vert f_{0}\Vert ^{2}-\frac{2}{\lambda _{C}}\frac{%
\log p_{0}}{n - n_1}\right\} \\
&\leq &\inf_{1\leq m\leq M_{n} \wedge n}\left\{ \inf_{J_{m}}\inf_{s\geq
1}\left( d^2 (f_0; \mathcal{F}^L_{J_m, s})+8c(\sigma^2 + 5 L^2) \frac{m}{n_1}%
+\frac{2}{\lambda _{C}}\left( \frac{ \log (M_{n}\wedge n) + \log {\binom{%
M_{n}}{m}} -\log (1-p_{0}) }{n - n_1} \right. \right. \right. \\
& & \left. \left. \left. + \frac{2\log (1+s)}{n - n_1}\right) \right)
\right\} \wedge \left\{ \Vert f_{0}\Vert ^{2}-\frac{2}{\lambda _{C}}\frac{%
\log p_{0}}{n - n_1}\right\} .
\end{eqnarray*}

For the \textbf{AC-Y} strategy, 
\begin{eqnarray*}
&&R(\hat{f}_{F_{n}};f_{0};n) \\
&\leq &C_{Y}\inf_{1\leq m\leq M_{n}\wedge n}\left\{ \inf_{J_{m}}\inf_{s\geq
1}\left( d^2 (f_0; \mathcal{F}^L_{J_m, s}) +c(2\sigma^{\prime} + H)^2 \frac{m%
}{n_1}+\sigma ^{2}\left( \frac{1+\log (M_{n}\wedge n)+\log {\binom{M_{n}}{m}}%
}{n - n_1}\right. \right. \right. \\
&&\left. \left. \left. +\frac{- \log (1-p_{0})+2 \log (1+s)}{n - n_1}\right)
\right) \right\} \wedge C_{Y}\left\{ \Vert f_{0}\Vert ^{2}+\sigma ^{2}\frac{%
1-\log p_{0}}{n - n_1}\right\} \\
&\leq &C_{Y}\inf_{1\leq m\leq M_{n}\wedge n}\left\{ \inf_{J_{m}}\inf_{s\geq
1}\left( d^2 (f_0; \mathcal{F}^L_{J_m, s}) +8c(\sigma^2 + 5 L^2) \frac{m}{n_1%
}+\sigma ^{2}\left( \frac{1+\log (M_{n}\wedge n) + \log {\binom{M_{n}}{m}}}{%
n - n_1}\right. \right. \right. \\
&&\left. \left. \left. +\frac{ - \log (1-p_{0})+2 \log (1+s)}{n - n_1}%
\right) \right) \right\} \wedge C_{Y}\left\{ \Vert f_{0}\Vert ^{2}+\sigma
^{2}\frac{1-\log p_{0}}{n - n_1}\right\} .
\end{eqnarray*}

\end{proof}

\begin{remark}
Similar oracle inequalities hold for the estimator $\hat{f}_A$ under the
linear regression setting with random design: $d^2(f_0; \mathcal{F}_{J_m}) $
is replaced by $d^2 (f_0; \mathcal{G}_{J_m})$, and $\sum_{j \in J_m}
\theta_j f_j$ is replaced by $\sum_{j \in J_m} \theta_j x_j$ in the above
theorem.
\end{remark}

\vspace*{.1in}

\section{Proofs}

\label{proofs}

\textit{Proof of Theorem \ref{Th1}.} 
\begin{proof}
(i) Because $\{e_j\}_{j=1}^{N_{\epsilon}}$ is an ${\epsilon}$-net of $%
\mathcal{F}_q(t_n)$ if and only if $\{t_n^{-1} e_j\}_{j=1}^{N_{\epsilon}}$
is an ${\epsilon}/t_n$-net of $\mathcal{F}_q(1)$, we only need to prove the
theorem for the case $t_n=1$. Recall that for any positive integer $k$, the
unit ball of $\ell_q^{M_n}$ can be covered by $2^{k-1}$ balls of radius $%
\epsilon_k$ in $\ell_1$ distance, where 
\begin{equation*}
\epsilon_k\le c\left\{%
\begin{array}{ll}
{1} & {1\le k\le \log_2(2M_n)} \\ 
{\left(\frac{\log_2(1+\frac{2M_n}{k})}{k}\right)^{\frac1q-1}} & {%
\log_2(2M_n)\le k\le 2M_n} \\ 
{2^{-\frac{k}{2M_n}}(2M_n)^{1-\frac1q}} & {k\ge 2M_n}%
\end{array}%
\right.
\end{equation*}
(c.f., \cite{EdmundsTriebel1998}, page 98). Thus, there are $2^{k-1} $
functions $g_j $, $1\le j\le 2^{k-1}$, such that 
\begin{equation*}
\mathcal{F}_q(1)\subset \bigcup_{j=1}^{2^{k-1}}(g_j+\mathcal{F}%
_1(\epsilon_k)).
\end{equation*}
For any $g\in \mathcal{F}_1(\epsilon_k)$, $g$ can be expressed as $%
g=\sum_{i=1}^{M_n}c_if_i$ with $\sum_{i=1}^{M_n}|c_i|\le \epsilon_k$. We
define a random function $U$, such that 
\begin{equation*}
\mathbb{P}(U=\mathrm{sign} (c_i) \epsilon_k f_i)=|c_i|/ \epsilon_k, \ \ 
\mathbb{P}(U=0)=1-\sum_{i=1}^{M_n}|c_i| / \epsilon_k.
\end{equation*}
Then we have $\|U\|_2\le \epsilon_k $ a.s. and $\mathbb{E} U=g$ under the
randomness just introduced. Let $U_1, U_2,...,U_m$ be i.i.d. copies of $U$,
and let $V=\frac1m\sum_{i=1}^m U_i$. We have 
\begin{eqnarray*}
\mathbb{E}\|V-g\|_2 = \sqrt{ \frac{1}{m} \| \mathbb{V}ar (U) \|_2 } \leq 
\sqrt{ \frac{1}{m} \mathbb{E} \| U \|_2^2 } \leq \frac{\epsilon_k}{\sqrt{m}}.
\end{eqnarray*}
In particular, there exists a realization of $V$, such that $\|V-g\|_2\le
\epsilon_k / \sqrt{m}$. Note that $V$ can be expressed as $\epsilon_k
m^{-1}(k_1f_1+k_2f_2+\cdots + k_{M_n}f_{M_n})$, where $k_1$, $k_2$, ..., $%
k_{M_n}$ are integers, and $|k_1|+|k_2|+\cdots+|k_{M_n}|\le m$. Thus, the
total number of different realizations of $V$ is upper bounded by $\binom{{%
2M_n+m}}{{m}}$. Furthermore, $\|V\|_0\le m$.

If $\log_2(2M_n)\le k\le 2M_n$, we choose $m$ to be the largest integer such
that ${\binom{{2M_n+m}}{{m}}}\le 2^k$. Then we have 
\begin{equation*}
\frac1m\le \frac{c^{\prime }}{k}\log_2\left(1 + \frac{2M_n}{k} \right)
\end{equation*}
for some positive constant $c^{\prime }$. Hence, $\mathcal{F}_q(1)$ can be
covered by $2^{2k-1}$ balls of radius 
\begin{equation*}
\epsilon_k \sqrt{c^{\prime}k^{-1}\log_2\left( 1 + \frac{2M_n}{k} \right)}
\end{equation*}
in $L^2$ distance.

If $k\ge 2M_n$, we choose $m=M_n$. Then $\mathcal{F}_q(1)$ can be covered by 
$2^{k-1}{\binom{{2M_n+m}}{{m}}}$ balls of radius $\epsilon_k M_n^{-1/2}$ in $%
L^2$ distance. Consequently, there exists a positive constant $c^{\prime
\prime }$ such that $\mathcal{F}_q(1)$ can be covered by $2^{l-1}$ balls of
radius $r_l$, where 
\begin{equation*}
r_{l} \le c^{\prime \prime }\left\{%
\begin{array}{ll}
{1} & {1\le l \le \log_2(2M_n)}, \\ 
{l^{\frac12-\frac1q}[\log_2(1+\frac{2M_n}{l})]^{\frac{1}{q}-\frac12}} & {%
\log_2(2M_n)\le l \le 2M_n}, \\ 
{2^{-\frac{l}{2M_n}}(2M_n)^{\frac12-\frac1q}} & {l \ge 2M_n}.%
\end{array}%
\right.
\end{equation*}
For any given $0<{\epsilon}<1$, by choosing the smallest $l$ such that $%
r_{l}<{\epsilon}/2$, we find an ${\epsilon}/2$-net $\{u_i\}_{i=1}^N$ of $%
\mathcal{F}_q(1)$ in $L^2$ distance, where 
\begin{equation*}
N=2^{l-1}\le \left\{%
\begin{array}{ll}
{\exp\left(c^{\prime \prime \prime }{\epsilon}^{-\frac{2q}{2-q}%
}\log(1+M_n^{\frac1q-\frac12}{\epsilon})\right)} & {{\epsilon}%
>M_n^{\frac12-\frac1q}}, \\ 
{\exp\left(c^{\prime \prime \prime }M_n\log(1+M_n^{\frac12-\frac1q}{%
\epsilon^{-1}})\right)} & {{\epsilon}<M_n^{\frac12-\frac1q}},%
\end{array}%
\right.
\end{equation*}
and $c^{\prime \prime \prime }$ is some positive constant.

It remains to show that for each $1 \leq i \leq N$, we can find a function $%
e_i$ so that $\|e_i\|_0\le 5{\epsilon}^{2q/(q-2)}+1$ and $\|e_i-u_i\|_2\le {%
\epsilon}/2$.

Suppose $u_i=\sum_{j=1}^{M_n}c_{ij}f_j$, $1 \leq i \leq N$, with $%
\sum_{j=1}^{M_n}|c_{ij}|^q\le 1$. Let $L_i=\{j: |c_{ij}|>{\epsilon}%
^{2/(2-q)}\}$. Then, $|L_i|{\epsilon}^{2q/(2-q)}\le \sum |c_{ij}|^q\le 1$,
which implies $|L_i|\le {\epsilon}^{2q/(q-2)}$ and also 
\begin{equation*}
\sum_{j\notin L_i}|c_{ij}|\le \sum_{j\notin L_i}|c_{ij}|^q[{\epsilon}%
^{2/(2-q)}]^{1-q}\le {\epsilon}^{\frac{2-2q}{2-q}}.
\end{equation*}
Define $v_i=\sum_{j\in L_i}c_{ij}f_j$ and $w_i=\sum_{j\notin L_i}c_{ij}f_j$.
We have $w_i\in \mathcal{F}_1({\epsilon}^{\frac{2-2q}{2-q}})$. By the
probability argument above, we can find a function $w^{\prime }_i$ such that 
$\|w^{\prime }_i\|_0\le m$ and $\|w_i-w^{\prime }_i\|_2\le {\epsilon}^{\frac{%
2-2q}{2-q}}/\sqrt{m}$. In particular, if we choose $m$ to be the smallest
integer such that $m\ge 4{\epsilon}^{2q/(q-2)}$. Then, $\|w_i-w^{\prime
}_i\|_2\le {\epsilon}/2$.

We define $e_i=v_i+w^{\prime }_i$, we have $\|u_i-e_i\|_2\le {\epsilon} / 2$%
, and then we can show that 
\begin{equation*}
\|e_i\|_0=\|v_i\|_0+\|w^{\prime }_i\|_0\le |L_i|+m\le 5{\epsilon}%
^{2q/(q-2)}+1.
\end{equation*}

(ii) Let $f_{\theta }^{\ast }=\sum_{j=1}^{M_{n}}c_{j}f_{j}=\arg
\inf_{f_{\theta }\in \mathcal{F}_{q}(t_{n})}\Vert f_{\theta }-f_{0}\Vert
^{2} $ be the best approximation of $f_{0}$ over the class $\mathcal{F}%
_{q}(t_{n}) $. For any $1\leq m\leq M_{n}$, let $L^{\ast
}=\{j:|c_{j}|>t_{n}m^{-1/q}\}$. Because $\sum_{j=1}^{M_{n}}|c_{j}|^{q}\leq
t_{n}^{q}$, we have $|L^{\ast }|t_{n}^{q}/m<\sum |c_{j}|^{q}\leq t_{n}^{q}$.
So, $|L^{\ast }|<m$. Also, 
\begin{equation*}
\sum_{j\notin L^{\ast }}|c_{j}|\leq \sum_{j\notin L^{\ast
}}|c_{j}|^{q}[t_{n}(1/m)^{1/q}]^{1-q}=\sum_{j\notin L^{\ast
}}|c_{j}|^{q}t_{n}^{1-q}(1/m)^{(1-q)/q}\leq t_{n}m^{1-1/q}:=D.
\end{equation*}%
Define $v^{\ast }=\sum_{j\in L^{\ast }}c_{j}f_{j}$ and $w^{\ast
}=\sum_{j\notin L^{\ast }}c_{j}f_{j}$. We have $w^{\ast }\in \mathcal{F}%
_{1}(D)$. Define a random function $U$ so that $\mathbb{P}(U=D\text{sign}%
(c_{j})f_{j})=|c_{j}|/D$, $j\notin L^{\ast }$ and $\mathbb{P}%
(U=0)=1-\sum_{j\notin L^{\ast }}|c_{j}|/D$. Thus, $\mathbb{E}U=w^{\ast }$,
where $\mathbb{E}$ denotes expectation with respect to the randomness $%
\mathbb{P}$ (just introduced). Also, $\Vert U\Vert \leq D\sup_{1\leq j\leq
M_{n}}\Vert f_{j}\Vert \leq D$. Let $U_{1},U_{2},...,U_{m}$ be i.i.d. copies
of $U$, then $\forall \mathbf{x}\in \mathcal{X},$ 
\begin{equation*}
\mathbb{E}\left( f_{0}(\mathbf{x})-v^{\ast }(\mathbf{x})-\frac{1}{m}%
\sum_{i=1}^{m}U_{i}(\mathbf{x})\right) ^{2}=\left( f_{\theta }^{\ast }(%
\mathbf{x})-f_{0}(\mathbf{x})\right) ^{2}+\frac{1}{m}\mathbb{V}ar\left( U(%
\mathbf{x})\right) .
\end{equation*}%
Together with Fubini, 
\begin{equation*}
\mathbb{E}\left\Vert f_{0}-v^{\ast }-\frac{1}{m}\sum_{i=1}^{m}U_{i}\right\Vert ^{2}\leq \Vert f_{\theta }^{\ast }-f_{0}\Vert ^{2}+\frac{1}{m}E\Vert
U\Vert ^{2}\leq \Vert f_{\theta }^{\ast }-f_{0}\Vert ^{2}+t_{n}^{2}m^{1-2/q}.
\end{equation*}%
In particular, there exists a realization of $v^{\ast }+\frac{1}{m}%
\sum_{i=1}^{m}U_{i}$, denoted by $f_{\theta ^{m}}$, such that $\Vert
f_{\theta ^{m}}-f_{0}\Vert ^{2}\leq \Vert f_{\theta }^{\ast }-f_{0}\Vert
^{2}+t_{n}^{2}m^{1-2/q}$. Note that $\Vert f_{\theta ^{m}}\Vert _{0}\leq
2m-1 $. If we consider $\widetilde{m}=\lfloor (m+1)/2\rfloor $ instead, we
have $2\widetilde{m}-1\leq m$ and $\widetilde{m}\geq m/2.$ The conclusion
then follows.

\end{proof}

\vspace*{0.2in} \textit{Proof of Theorem \ref{UpAgg}.} 
\begin{proof}
To derive the upper bounds, we only need to examine the index of
resolvability for each strategy. The natures of the constants in Theorem \ref%
{UpAgg} follow from Theorem \ref{Oracle}.

(i) For \textbf{T-} strategies, according to Theorem \ref{Th1} and the
general oracle inequalities in Theorem \ref{Oracle}, for each $1\leq m\leq
M_{n}\wedge n$, there exists a subset $J_{m}$ and the best $f_{\theta
^{m}}\in \mathcal{F}_{J_{m}}$ such that 
\begin{eqnarray*}
R(\hat{f}_{F_{n}};f_{0};n) &\leq &c_{0}\left( c_{1}\Vert f_{\theta
^{m}}-f_{0}\Vert ^{2}+2c_{2}\frac{m}{n}+2c_{3}\frac{1+\log {\binom{M_{n}}{m}+%
}\log (M_{n}\wedge n)-\log (1-p_{0})}{n}\right) \\
&&\wedge c_{0}\left( \Vert f_{0}\Vert ^{2}+2c_{3}\frac{1-\log p_{0}}{n}%
\right) .
\end{eqnarray*}%
Under the assumption that $f_{0}$ has sup-norm bounded, the index of
resolvability evaluated at the null model $f_{\theta }\equiv 0$ leads to the
fact that the risk is always bounded above by $C_{0}\left( \Vert f_{0}\Vert
^{2}+\frac{C_2 \sigma ^{2}}{n}\right) $ for some constant $C_{0}, C_2>0.$

For $\mathcal{F} = \mathcal{F}_q (t_n)$, and when $m_* = m^{\ast} = M_{n} <
n $, evaluating the index of resolvability at the full model $J_{M_{n}}$, we
get 
\begin{equation*}
R(\hat{f}_{F_{n}};f_{0};n)\leq c_{0}c_{1}d^{2}(f_{0};\mathcal{F}_{q}(t_n))+%
\frac{CM_{n}}{n} \hspace{0.1cm}\text{ with } \hspace{0.1cm} \frac{CM_{n}}{n}=%
\frac{Cm_{\ast }\left( 1+\log \left( \frac{M_{n}}{m_{\ast }}\right) \right) 
}{n}.
\end{equation*}
Thus, the upper bound is proved when $m_{\ast }=m^{\ast }=M_{n}.$

For $\mathcal{F} = \mathcal{F}_q (t_n)$, and when $m_{\ast}=m^{\ast }=n <
M_n,$ then clearly $m_{\ast }\left( 1+\log \left( \frac{M_{n}}{m_{\ast }}%
\right) \right)/n$ is larger than 1, and then the risk bound given in the
theorem in this case holds.

For $\mathcal{F}=\mathcal{F}_{q}(t_{n})$, and when $1\leq m_{\ast }\leq
m^{\ast }<M_{n}\wedge n$, for $1\leq m<M_{n},$ and from Theorem \ref{Th1},
we have 
\begin{eqnarray*}
R(\hat{f}_{F_{n}};f_{0};n) &\leq &c_{0}\left( c_{1}d^{2}(f_{0};\mathcal{F}%
_{q}(t_{n}))+c_{1}2^{2/q-1}t_{n}^{2}m^{1-2/q}+2c_{2}\frac{m}{n} \right. \\
&&\left. +2c_{3}\frac{1+\log {\binom{M_{n}}{m}+}\log (M_{n}\wedge n)}{n} -
2c_{3}\frac{\log (1-p_{0})}{n}\right) .
\end{eqnarray*}%
Since $\log {\binom{M_{n}}{m}}\leq m\log \left( \frac{eM_{n}}{m}\right)
=m\left( 1+\log \frac{M_{n}}{m}\right) $, then 
\begin{eqnarray*}
R(\hat{f}_{F_{n}};f_{0};n) &\leq &c_{0}c_{1}d^{2}(f_{0};\mathcal{F}%
_{q}(t_{n}))+C\left( t_{n}^{2}m^{1-2/q}+\frac{m\left( 1+\log \frac{M_{n}}{m}%
\right) }{n}+\frac{\log (M_{n}\wedge n)}{n}\right) \\
&\leq &c_{0}c_{1}d^{2}(f_{0};\mathcal{F}_{q}(t_{n}))+C^{^{\prime }}\left(
t_{n}^{2}m^{1-2/q}+\frac{m\left( 1+\log \frac{M_{n}}{m}\right) }{n}\right) ,
\end{eqnarray*}%
where $C$ and $C^{\prime }$ are constants that do not depend on $n$, $t_{n}$%
, and $M_{n}$ (but may depend on $q,$ $\sigma ^{2},$ $p_{0}$ and $L$). Choosing 
$m=m_{\ast },$ we have 
\begin{equation*}
t_{n}^{2}m^{1-2/q}+\frac{m\left( 1+\log \frac{M_{n}}{m}\right) }{n}\leq
C^{\prime \prime }\frac{m_{\ast }\left( 1+\log \left( \frac{M_{n}}{m_{\ast }}%
\right) \right) }{n}.
\end{equation*}%
The upper bound for this case then follows.

For $\mathcal{F} = \mathcal{F}_0 (k_n)$, by evaluating the index of
resolvability from Theorem \ref{Oracle} at $m=k_{n}$, the upper bound
immediately follows.

For $\mathcal{F} = \mathcal{F}_q (t_n) \cap \mathcal{F}_0 (k_n)$, both $\ell
_{q}$- and $\ell _{0}$-constraints are imposed on the coefficients, the
upper bound will go with the faster rate from the tighter constraint. The
result follows.

(ii) For \textbf{AC-} strategies, three constraints $\Vert \theta \Vert
_{1}\leq s$ ($s>0$), $\Vert \theta \Vert _{q}\leq t_{n}$ ($0\leq q\leq 1$, $%
t_{n}>0$) and $\Vert f_{\theta }\Vert _{\infty }\leq L$ are imposed on the
coefficients. Notice that $\Vert \theta \Vert _{1}\leq \Vert \theta \Vert
_{q}$ when $0<q\leq 1$, then the $\ell _{1}$-constraint is satisfied by
default as long as $s\geq t_{n}$ and $\Vert \theta \Vert _{q}\leq t_{n}$
with $0<q\leq 1$. Using similar arguments as used for \textbf{T-}strategies,
the desired upper bounds can be easily derived.

\end{proof}

\vspace*{.2in} 
\textit{Global metric entropy and local metric entropy.} The tools developed
in Yang and Barron \cite{YangBarron1999} allow us to derive minimax lower
bounds for $\ell_{q}$-aggregation of estimates or regression under $\ell_{q}$%
-constraints. Both global and local entropies of the regression function
classes are relevant. The following lower bound result slightly generalizes
Lemma 1 in \cite{Yang2004}.

Consider estimating a regression function $f_{0}$ in a general function
class $\mathcal{F}$ based on i.i.d. observations $(\mathbf{X}%
_{i},Y_{i})_{i=1}^{n}$ from the model%
\begin{equation}
Y=f_{0}(\mathbf{X})+\sigma \cdot \varepsilon ,  \label{globallocal}
\end{equation}%
where $\sigma >0$ and $\varepsilon $ follows a standard normal distribution
and is independent of $\mathbf{X}$.

Given $\mathcal{F},$ we say $G\subset \mathcal{F}$ is an \textit{$\epsilon $%
-packing set} in $\mathcal{F}$ ($\epsilon >0$) if any two functions in $G$
are more than $\epsilon $ apart in the $L_{2}$ distance. Let $0<\alpha <1$
be a constant.

\textsc{Definition 1:} (\textit{Global metric entropy}) The packing $%
\epsilon $-entropy of $\mathcal{F}$ is the logarithm of the largest $%
\epsilon $-packing set in $\mathcal{F}$. The packing $\epsilon $-entropy of $%
\mathcal{F}$ is denoted by $M(\epsilon ).$

\textsc{Definition 2:} (\textit{Local metric entropy}) The $\alpha $-local $%
\epsilon $-entropy at $f\in \mathcal{F}$ is the logarithm of the largest $%
\left( \alpha \epsilon \right) $-packing set in $\mathcal{B} (f,\epsilon
)=\{f^{^{\prime }}\in \mathcal{F}:\parallel f^{^{\prime }}-f\parallel \leq
\epsilon \}$. The $\alpha $-local $\epsilon $-entropy at $f$ is denoted by $%
M_{\alpha }(\epsilon \mid f)$. The $\alpha $-local $\epsilon $-entropy of $%
\mathcal{F}$ is defined as $M_{\alpha }^{\text{loc}}(\epsilon )=\max_{f\in 
\mathcal{F}}M_{\alpha }(\epsilon \mid f).$

Suppose that $M_{\alpha }^{\text{loc }}(\epsilon )$ is lower bounded by $%
\underline{M}_{\alpha}^{\text{loc}}(\epsilon )$ (a continuous function), and
assume that $M(\epsilon )$ is upper bounded by $\overline{M}(\epsilon )$ and
lower bounded by $\underline{M}(\epsilon )$ (with $\overline{M}(\epsilon )$
and $\underline{M}(\epsilon )$ both being continuous).

Suppose there exist $\epsilon _{n}$, $\overline{\epsilon }_{n}$, and $%
\underline{\epsilon }_{n}$ such that 
\begin{eqnarray}
\underline{M}_{\alpha}^{\text{loc}}(\sigma \epsilon _{n}) &\geq& n \epsilon
_{n}^{2}+2\log 2, \\
\overline{M}(\sqrt{2} \sigma \overline{\epsilon }_{n})&=&n \overline{%
\epsilon }_{n}^{2},  \label{globalupper} \\
\underline{M}( \sigma \underline{\epsilon }_{n})&=& 4n \overline{\epsilon }%
_{n}^{2}+2\log 2.  \label{globallower}
\end{eqnarray}

\begin{proposition}
\label{YangBarron} (Yang and Barron \cite{YangBarron1999}) The minimax risk
for estimating $f_{0}$ from model (\ref{globallocal}) in the function class $%
\mathcal{F}$ is lower-bounded as the following 
\begin{equation*}
\inf_{\hat{f}}\sup_{f_{0}\in \mathcal{F}}E\Vert \hat{f}-f_{0}\Vert ^{2}\geq 
\frac{\alpha ^{2}\sigma ^{2}\epsilon _{n}^{2}}{8},
\end{equation*}%
\begin{equation*}
\inf_{\hat{f}}\sup_{f_{0}\in \mathcal{F}}E\Vert \hat{f}-f_{0}\Vert ^{2}\geq 
\frac{\sigma ^{2}\underline{\epsilon }_{n}^{2}}{8}.
\end{equation*}%
Let $\underline{\mathcal{F}}$ be a subset of $\mathcal{F}$. If a packing set
in $\mathcal{F}$ of size at least $\exp (\underline{M}_{\alpha }^{\text{loc}%
}(\sigma \epsilon _{n}))$ or $\exp (\underline{M}(\sigma \underline{\epsilon 
}_{n}))$ is actually contained in $\underline{\mathcal{F}}$, then $\inf_{%
\hat{f}}\sup_{f_{0}\in \underline{\mathcal{F}}}E\Vert \hat{f}-f_{0}\Vert
^{2} $ is lower bounded by $\frac{\alpha ^{2}\sigma ^{2}\epsilon _{n}^{2}}{8}
$ or $\frac{\sigma ^{2}\underline{\epsilon }_{n}^{2}}{8},$ respectively.
\end{proposition}

\begin{proof}
The result is essentially given in \cite{YangBarron1999}, but not in
the concrete forms. The second lower bound is given in \cite{Yang2004}. We
briefly derive the first one.

Let $N$ be an $\left( \alpha \epsilon _{n}\right) $-packing set in $\mathcal{%
B}(f,\sigma \epsilon _{n})=\{f^{^{\prime }}\in \mathcal{F}:\parallel
f^{^{\prime }}-f\parallel \leq \sigma \epsilon _{n}\}.$ Let $\Theta $ denote
a uniform distribution on $N.$ Then, the mutual information between $\Theta $
and the observations $(\mathbf{X}_{i},Y_{i})_{i=1}^{n}$ is upper bounded by $%
\frac{n}{2}\epsilon _{n}^{2}$ (see Yang and Barron \cite{YangBarron1999},
Sections 7 and 3.2) and an application of Fano's inequality to the
regression problem gives the minimax lower bound%
\begin{equation*}
\frac{\alpha ^{2}\sigma ^{2}\epsilon _{n}^{2}}{4}\left( 1-\frac{I\left(
\Theta ;(\mathbf{X}_{i},Y_{i})_{i=1}^{n}\right) +\log 2}{\log |N|}\right) ,
\end{equation*}%
where $|N|$ denote the size of $N.$ By our way of defining $\epsilon _{n},$
the conclusion of the first lower bound follows.

For the last statement, we prove for the global entropy case and the
argument for the local entropy case similarly follows. Observe that the
upper bound on $I\left( \Theta ;(\mathbf{X}_{i}, Y_i)_{i=1}^{n}\right) $ by $%
\log (|G|)+n\overline{\epsilon }_{n}^{2},$ where $G$ is an $\overline{%
\epsilon }_{n}$-net of $\mathcal{F}$ under the square root of the
Kullback-Leibler divergence (see \cite{YangBarron1999}, page 1571),
continues to be an upper bound on $I\left( \underline{\Theta };(\mathbf{X}%
_{i}, Y_i)_{i=1}^{n}\right) ,$ where $\underline{\Theta }$ is the uniform
distribution on a packing set in $\underline{\mathcal{F}}.$ Therefore, by
the derivation of Theorem 1 in \cite{YangBarron1999}, the same lower bound
holds for $\underline{\mathcal{F}}$ as well.

\end{proof}
\newpage
 \textit{Proof of Theorem \ref{LowAgg}.}

\begin{proof}
Assume $f_{0}\in \mathcal{F}$ in each case of $\mathcal{F}$ so that $%
d^{2}(f_{0};\mathcal{F})=0$. Without loss of generality, assume $\sigma =1$. 

(i) We first derive the lower bounds without $L_{2}$ or $L_{\infty }$ upper
bound assumption on $f_{0}$. To prove case 1 (i.e., $\mathcal{F}=\mathcal{F}%
_{q}(t_{n})$), it is enough to show that 
\begin{equation*}
\inf_{\hat{f}}\sup_{f_{0}\in \mathcal{F}_{q}(t_{n})}E\Vert \hat{f}%
-f_{0}\Vert ^{2}\geq C_{q}\left\{ 
\begin{array}{ll}
\frac{M_{n}}{n} & \text{if }\widetilde{\text{$m$}}\text{$^{\ast }=M_{n}$,}
\\ 
t_{n}^{q}\left( \frac{1+\log \frac{M_{n}}{(nt_{n}^{2})^{q/2}}}{n}\right)
^{1-q/2} & \text{if $1<\widetilde{m}_{\ast }\leq $}\widetilde{\text{$m$}}%
\text{$^{\ast }<M_{n}$,} \\ 
t_{n}^{2} & \text{if }\widetilde{m}_{\ast }\text{$=1$, }%
\end{array}%
\right. 
\end{equation*}%
in light of the fact that, by definition, when $\widetilde{\text{$m$}}^{\ast
}=M_{n}$, $\widetilde{m}_{\ast }=M_{n}$ and when $1<\widetilde{m}_{\ast
}\leq \widetilde{\text{$m$}}^{\ast }<M_{n}$, we have $\frac{\widetilde{m}%
_{\ast }(1+\log \frac{M_{n}}{\widetilde{\text{$m$}}_{\ast }})}{n}$ is upper
and lower bounded by multiples (depending only on $q$) of $t_{n}^{q}\left( 
\frac{1+\log \frac{M_{n}}{(nt_{n}^{2})^{q/2}}}{n}\right) ^{1-q/2}.$ Note
that $\widetilde{\text{$m$}}^{\ast }$ and $\widetilde{m}_{\ast }$ are
defined as $m$$^{\ast }$ and $m_{\ast }$ except that no ceiling of $n$ is
imposed there.

Given that the basis functions are orthonormal, the $L_{2}$ distance on $%
\mathcal{F}_{q}(t_n)$ is the same as the $\ell_{2}$ distance on the
coefficients in $B_{q}(t_n;M_{n})=\{\theta :\Vert \theta \Vert _{q}\leq
t_n\}.$ Thus, the entropy of $\mathcal{F}_{q}(t_n)$ under the $L_{2}$
distance is the same as that of $B_{q}(t_n;M_{n})$ under the $\ell_{2}$
distance.

When $\widetilde{\text{$m$}}^{\ast }=M_{n}$, we use the lower bound tool in
terms of local metric entropy. Given the $\ell _{q}$-$\ell _{2}$%
-relationship $\Vert \theta \Vert _{q}\leq {M_{n}}^{1/q-1/2}\Vert \theta
\Vert _{2}$ for $0<q\leq 2,$ for $\epsilon \leq \sqrt{M_{n}/n},$ taking $%
f_{0}^{\ast }\equiv 0$, we have 
\begin{equation*}
\mathcal{B}(f_{0}^{\ast };\epsilon )=\{f_{\theta }:\Vert f_{\theta
}-f_{0}^{\ast }\Vert \leq \epsilon ,\Vert \theta \Vert _{q}\leq
t_{n}\}=\{f_{\theta }:\Vert \theta \Vert _{2}\leq \epsilon ,\Vert \theta
\Vert _{q}\leq t_{n}\}=\{f_{\theta }:\Vert \theta \Vert _{2}\leq \epsilon \},
\end{equation*}%
where the last equality holds because when $\epsilon \leq \sqrt{M_{n}/n},$
for $\Vert \theta \Vert _{2}\leq \epsilon $, $\Vert \theta \Vert _{q}\leq
t_{n}$ is always satisfied. Consequently, for $\epsilon \leq \sqrt{M_{n}/n},$
the $(\epsilon /2)$-packing of $\mathcal{B}(f_{0}^{\ast };\epsilon )$ under
the $L_{2}$ distance is equivalent to the $(\epsilon /2)$-packing of $%
B_{\epsilon }=\{\theta :\Vert \theta \Vert _{2}\leq \epsilon \}$ under the $%
\ell _{2}$ distance. Note that the size of the maximum packing set is at
least the ratio of volumes of the balls $B_{\epsilon }$ and $B_{\epsilon /2},
$ which is $2^{M_{n}}.$ Thus, the local entropy $M_{1/2}^{\text{loc}%
}(\epsilon )$ of $\mathcal{F}_{q}(t)$ under the $L_{2}$ distance is at least 
$\underline{M}_{1/2}^{\text{loc}}(\epsilon )=M_{n}\log 2$ for $\epsilon \leq 
\sqrt{M_{n}/n}.$ The minimax lower bound for the case of $\widetilde{\text{$m
$}}^{\ast }=M_{n}$ then directly follows from Proposition \ref{YangBarron}.

When $1<\widetilde{m}_{\ast }\leq \widetilde{\text{$m$}}^{\ast }<M_{n}$, the
use of global entropy is handy. Applying the minimax lower bound in terms of
global entropy in Proposition \ref{YangBarron}, with the metric entropy
order for larger $\epsilon $ (which is tight in our case of orthonormal
functions in the dictionary) from Theorem \ref{Th1}, the minimax lower rate
is readily obtained. Indeed, for the class $\mathcal{F}_{q}(t_{n}),$ with ${%
\epsilon }>t_{n}M_{n}^{\frac{1}{2}-\frac{1}{q}},$ there are constants $%
c^{\prime }$ and $\underline{c}^{\prime }$ (depending only on $q$) such that 
\begin{equation*}
\underline{c}^{\prime }\left( t_{n}\epsilon ^{-1}\right) ^{\frac{2q}{2-q}%
}\log (1+M_{n}^{\frac{1}{q}-\frac{1}{2}}t_{n}^{-1}\epsilon )\leq \underline{M%
}(\epsilon )\leq \overline{M}(\epsilon )\leq c^{\prime }\left( t_{n}\epsilon
^{-1}\right) ^{\frac{2q}{2-q}}\log (1+M_{n}^{\frac{1}{q}-\frac{1}{2}%
}t_{n}^{-1}\epsilon ).
\end{equation*}%
Thus, we see that $\underline{\epsilon }_{n}$ determined by (8.4) is lower
bounded by $c^{^{\prime \prime \prime }}t_{n}^{\frac{q}{2}}\left( (1+\log 
\frac{M_{n}}{(nt_{n}^{2})^{q/2}})/n\right) ^{\frac{1}{2}-\frac{q}{4}},$
where $c^{^{\prime \prime \prime }}$ is a constant depending only on $q.$

When $\widetilde{m}_{\ast }=1$, note that with $f_{0}^{\ast }=0$ and $%
\epsilon \leq t_{n},$ 
\begin{equation*}
\mathcal{B}(f_{0}^{\ast };\epsilon )=\{f_{\theta }:\Vert \theta \Vert
_{2}\leq \epsilon ,\Vert \theta \Vert _{q}\leq t_{n}\}\supset \{f_{\theta
}:\Vert \theta \Vert _{q}\leq \epsilon \}.
\end{equation*}%
Observe that the $(\epsilon /2)$-packing of $\{f_{\theta }:\Vert \theta
\Vert _{q}\leq \epsilon \}$ under the $L_{2}$ distance is equivalent to the $%
(1/2)$-packing of $\{f_{\theta }:\Vert \theta \Vert _{q}\leq 1\}$ under the
same distance. Thus, by applying Theorem \ref{Th1} with $t_{n}=1$ and $%
\epsilon =1/2$, we know that the $(\epsilon /2)$-packing entropy of $%
\mathcal{B}(f_{0}^{\ast };\epsilon )$ is lower bounded by $\underline{c}%
^{^{\prime \prime }}\log (1+\frac{1}{2}M_{n}^{1/q-1/2})$ for some constant $%
\underline{c}^{^{\prime \prime }}$depending only on $q,$ which is at least a
multiple of $nt_{n}^{2}$ when $\widetilde{\text{$m$}}^{\ast }\leq \left(
1+\log \frac{M_{n}}{\widetilde{\text{$m$}}\text{$^{\ast }$}}\right) ^{q/2}$.
Therefore we can choose $0<\delta <1$ small enough (depending only on $q$)
such that 
\begin{equation*}
\underline{c}^{^{\prime \prime }}\log (1+\frac{1}{2}M_{n}^{1/q-1/2})\geq
n\delta ^{2}t_{n}^{2}+2\log 2.
\end{equation*}%
The conclusion then follows from applying the first lower bound of
Proposition \ref{YangBarron}.

To prove case 2 (i.e., $\mathcal{F} = \mathcal{F}_0 (k_n)$), noticing that
for $M_{n}/2\leq k_{n}\leq M_{n},$ we have $(1 + \log2 )/ 2 M_{n} \leq k_{n}
\left(1 + \log \frac{M_{n}}{k_{n}} \right) \leq M_{n},$ together with the
monotonicity of the minimax risk in the function class, it suffices to show
the lower bound for $k_{n}\leq M_{n}/2.$ Let $B_{k_{n}}(\epsilon )=\{\theta
:\Vert \theta \Vert _{2}\leq \epsilon ,\Vert \theta \Vert _{0}\leq k_{n}\}.$
As in case 1, we only need to understand the local entropy of the set $%
B_{k_{n}}(\epsilon )$ for the critical $\epsilon $ that gives the claimed
lower rate. Let $\eta =\epsilon /\sqrt{k_{n}}.$ Then $B_{k_{n}}(\epsilon )$
contains the set $D_{k_{n}}(\eta ),$ where%
\begin{equation*}
D_{k}(\eta )=\{\theta =\eta I:I\in \{1,0,-1\}^{M_{n}},\Vert I\Vert _{0}\leq
k\}.
\end{equation*}%
Clearly $\Vert \eta I_{1}-\eta I_{2}\Vert _{2} \geq \eta \left(
d_{HM}(I_{1},I_{2})\right) ^{1/2},$ where $d_{HM}(I_{1},I_{2})$ is the
Hamming distance between $I_{1},I_{2}\in \{1,0,-1\}^{M_{n}}.$ From Lemma 4
of \cite{Raskuttietal2010} (the result there actually also holds when
requiring the pairwise Hamming distance to be strictly larger than $k/2$;
see also the derivation of a metric entropy lower bound in \cite{Kuhn2001}),
there exists a subset of $\{I:I\in \{1,0,-1\}^{M_{n}},\Vert I\Vert _{0}\leq
k\}$ with more than $\exp \left( \frac{k}{2} \log \frac{2(M_{n}-k)}{k}%
\right) $ points that have pairwise Hamming distance larger than $k/2$.
Consequently, we know the local entropy $M_{1 / \sqrt{2}}^{loc}(\epsilon )$
of $\mathcal{F}_{0}(k_{n}) $ is lower bounded by $\frac{k_{n}}{2}\log \frac{%
2(M_{n}-k_{n})}{k_{n}}.$ The result follows.

To prove case 3 (i.e., $\mathcal{F}_{q}(t_{n})\cap \mathcal{F}_{0}(k_{n})$),
for the larger $k_{n}$ case, from the proof of case 1, we have used fewer
than $k_{n}$ nonzero components to derive the minimax lower bound there.
Thus, the extra $\ell _{0}$-constraint does not change the problem in terms
of lower bound. For the smaller $k_{n}$ case, note that for $\theta $ with $%
\Vert \theta \Vert _{0}\leq k_{n},$ $\Vert \theta \Vert _{q}\leq
k_{n}^{1/q-1/2}\Vert \theta \Vert _{2}\leq k_{n}^{1/q-1/2} \cdot $ $\sqrt{%
Ck_{n}\left( 1+\log \frac{M_{n}}{k_{n}}\right) /n}$ for $\theta $ with $%
\Vert \theta \Vert _{2}\leq \sqrt{Ck_{n}\left( 1+\log \frac{M_{n}}{k_{n}}%
\right) /n}$ for some constant $C>0$. Therefore the $\ell _{q}$-constraint
is automatically satisfied when $\Vert \theta \Vert _{2}$ is no larger than
the critical order $\sqrt{k_{n}\left( 1+\log \frac{M_{n}}{k_{n}}\right) /n}$%
, which is sufficient for the lower bound via local entropy techniques. The
conclusion follows.

(ii) Now, we turn to the lower bounds under the $L_{2}$ norm condition. When
the regression function $f_{0}$ satisfies the boundedness condition in $%
L_{2} $ norm, the estimation risk is obviously upper bounded by $L^{2}$ by
taking the trivial estimator $\hat{f}=0.$ In all of the lower boundings in
(i) through local entropy argument, if the critical radius $\epsilon $ is of
order $1$ or lower, the extra condition $\Vert f_{0}\Vert \leq L$ does not
affect the validity of the lower bound. Otherwise, we take $\epsilon $ to be 
$L$. Then, since the local entropy stays the same, it directly follows from the
first lower bound in Proposition \ref{YangBarron} that $L^{2}$ is a lower
order of the minimax risk. The only case remained is that of $\left( 1+\log 
\frac{M_{n}}{m^{\ast }}\right) ^{q/2}\leq m^{\ast }<M_{n}.$ If $%
t_{n}^{q}\left( (1+\log \frac{M_{n}}{(nt^{2})^{q/2}})/n\right) ^{1-q/2}$ is
upper bounded by a constant, from the proof of the lower bound of the metric
entropy of the $\ell _{q}$-ball in \cite{Kuhn2001}, we know that the
functions in the special packing set satisfy the $L_{2}$ bound. Indeed,
consider $\{f_{\theta }:\theta \in D_{m_{n}}(\eta )\}$ with $m_{n}$ being a
multiple of $\left( nt_{n}^{2}/\left( 1+\log \frac{M_{n}}{(nt_{n}^{2})^{q/2}}%
\right) \right) ^{q/2}$ and $\eta $ being a (small enough) multiple of $%
\sqrt{(1+\log \frac{M_{n}}{(nt_{n}^{2})^{q/2}})/n}.$ Then these $f_{\theta }$
have $\Vert f_{\theta }\Vert $ upper bounded by a multiple of $%
t_{n}^{q}\left( (1+\log \frac{M_{n}}{(nt_{n}^{2})^{q/2}})/n\right) ^{1-q/2}$
and the minimax lower bound follows from the last statement of Proposition %
\ref{YangBarron}. If $t_{n}^{q}\left( (1+\log \frac{M_{n}}{(nt_{n}^{2})^{q/2}%
})/n\right) ^{1-q/2}$ is not upper bounded, we reduce the packing radius to $%
L$ (i.e., choose $\eta $ so that $\eta \sqrt{m_{n}}$ is bounded by a
multiple of $L$). Then the functions in the packing set satisfy the $L_{2}$
bound and furthermore, the number of points in the packing set is of a
larger order than $nt_{n}^{q}\left( (1+\log \frac{M_{n}}{(nt_{n}^{2})^{q/2}}%
)/n\right) ^{1-q/2}$. Again, adding the $L_{2}$ condition on $f_{0}\in 
\mathcal{F}_{q}(t)$ does not increase the mutual information bound in our
application of Fano's inequality. We conclude that the minimax risk is lower
bounded by a constant.

(iii) Finally, we prove the lower bounds under the sup-norm bound condition.
For 1), under the direct sup-norm assumption, the lower bound is obvious.
For the general $M_{n}$ case 2), note that the functions $f_{\theta }$'s in
the critical packing set satisfies that $\Vert \theta \Vert _{2}\leq
\epsilon $ with $\epsilon $ being a multiple of $\sqrt{\frac{k_{n}\left(
1+\log \frac{M_{n}}{k_{n}}\right) }{n}}.$ Then together with $\Vert \theta
\Vert _{0}\leq k_{n},$ we have $\Vert \theta \Vert _{1}\leq \sqrt{k_{n}}%
\Vert \theta \Vert _{2},$ which is bounded by assumption. The lower bound
conclusion then follows from the last part of Proposition \ref{YangBarron}.
To prove the results for the case $M_{n} / \left( 1 + \log \frac{M_n}{k_n}
\right) \leq bn,$ as in \cite{Tsybakov2003}, we consider the special
dictionary $F_{n}=\{f_{i}:1\leq i\leq M_{n}\}$ on $[0,1],$ where 
\begin{equation*}
f_{i}(\mathbf{x})=\sqrt{M_{n}}I_{[\frac{i-1}{M_{n}},\frac{i}{M_{n}})}(%
\mathbf{x}),\hspace{1cm}i=1,...,M_{n}.
\end{equation*}%
Clearly, these functions are orthonormal. By the last statement of
Proposition \ref{YangBarron}, we only need to verify that the functions in
the critical packing set in each case do have the sup-norm bound condition
satisfied. Note that for any $f_{\theta }$ with $\theta \in D_{k_{n}}(\eta )$
(as defined earlier), we have $\Vert f_{\theta }\Vert \leq \eta \sqrt{k_{n}}$
and $\Vert f_{\theta }\Vert _{\infty }\leq \eta \sqrt{M_{n}}.$ Thus, it
suffices to show that the critical packing sets for the previous lower
bounds without the sup-norm bound can be chosen with $\theta $ in $%
D_{k_{n}}(\eta )$ for some $\eta =O\left( M_{n}^{-1/2}\right) .$ Consider $%
\eta $ to be a (small enough) multiple of $\sqrt{\left( 1+\log \frac{M_{n}}{%
k_{n}}\right) /n}=O\left( M_{n}^{-1/2}\right) $ (which holds under the
assumption $\frac{M_{n}}{1+\log \frac{M_{n}}{k_{n}}}\leq bn$). From the
proof of part (ii) without constraint, we know that there is a subset of $%
D_{k_{n}}(\eta )$ that with more than $\exp (\frac{k_{n}}{2}\log \frac{%
2(M_{n}-k_{n})}{k_{n}})$ points that are separated in $\ell _{2}$ distance
by at least $\sqrt{k_{n}\left( 1+\log \frac{M_{n}}{k_{n}}\right) /n}$.

\end{proof}

\vspace*{.2in} \textit{Proof of Theorem \ref{UpReg}.}

\begin{proof}
For linear regression with random design, we assume the true regression
function $f_0$ belongs to $\mathcal{F}^L_q (t_n; M_n)$, or $\mathcal{F}^L_0
(k_n; M_n) $, or both, thus $d^2(f_0,\mathcal{F})$ is equal to zero for all
cases (except for \textbf{AC-} strategies when $\mathcal{F} = \mathcal{F}%
^L_0 (k_n; M_n)$, which we discuss later).

(i) For \textbf{T-} strategies and $\mathcal{F}=\mathcal{F}%
_{q}^{L}(t_{n};M_{n})$. For each $1\leq m\leq M_{n}\wedge n$, according to
the general oracle inequalities in Theorem \ref{Oracle} , the adaptive
estimator $\hat{f}_{A}$ has 
\begin{eqnarray*}
\sup_{f_{0}\in \mathcal{F}}R(\hat{f}_{A};f_{0};n) & \leq & c_{0}\left( 2c_{2}%
\frac{m}{n}+2c_{3}\frac{1+\log \binom{M_{n}}{m}+\log (M_{n}\wedge n)-\log
(1-p_{0})}{n}\right) \\
& & \wedge c_{0}\left( \Vert f_{0}\Vert ^{2}-2c_{3}\frac{\log p_{0}}{n}%
\right) .
\end{eqnarray*}

When $m_* = m^* = M_n <n$, the full model $J_{M_n}$ results in an upper
bound of order $M_n / n$.

When $m_* = m^* = n < M_n$, we choose the null model and the upper bound is
simply of order one.

When $1 < m_* \leq m^* < M_n \wedge n$, the similar argument of Theorem \ref%
{UpAgg} leads to an upper bound of order $1 \wedge \frac{m_*}{n} \left( 1 +
\log \frac{M_n}{m_*} \right)$. Since $(n t_n^2)^{q / 2} \left( 1 + \log 
\frac{M_n}{(n t_n^2)^{q / 2}} \right)^{-q / 2} \leq m_* \leq 4 (n t_n^2)^{q
/ 2} \left( 1 + \log \frac{M_n}{2 (n t_n^2)^{q / 2}} \right)^{-q / 2}$, then
the upper bound is further upper bounded by $c_q t_n^q \cdot$ $\left( \frac{%
1 + \log \frac{M_n}{(n t_n^2)^{q / 2}}}{n} \right)^{1 - q / 2}$ for some
constant $c_q $ only depending on $q$.

When $m_* = 1$, the null model leads to an upper bound of order $\| f_0 \|^2
+ \frac{1}{n} \leq t_n^2 + \frac{1}{n} \leq 2 (t_n^2 \vee \frac{1}{n})$ if $%
f_0 \in \mathcal{F}^L_q (t_n; M_n)$.

For $\mathcal{F} = \mathcal{F}^L_{0} (k_n; M_n)$ or $\mathcal{F} = \mathcal{F%
}^L_q (t_n; M_n) \cap \mathcal{F}^L_{0} (k_n; M_n)$, one can use the same
argument as in Theorem \ref{UpAgg}.

(ii) For \textbf{AC-} strategies, for $\mathcal{F}=\mathcal{F}%
_{q}^{L}(t_{n};M_{n})$ or $\mathcal{F}=\mathcal{F}_{q}^{L}(t_{n};M_{n})\cap 
\mathcal{F}_{0}^{L}(k_{n};M_{n})$, again one can use the same argument as in
the proof of Theorem \ref{UpAgg}. For $\mathcal{F}=\mathcal{F}%
_{0}^{L}(k_{n};M_{n})$, the approximation error is $\inf_{s\geq 1}\left(
\inf_{\{\theta :\Vert \theta \Vert _{1}\leq s,\Vert \theta \Vert _{0}\leq
k_{n},\Vert f_{\theta }\Vert _{\infty }\leq L\}}\Vert f_{\theta }-f_{0}\Vert
^{2}+2c_{3}\frac{\log (1+s)}{n}\right) \leq \inf_{\{\theta :\Vert \theta
\Vert _{1}\leq \alpha _{n},\Vert \theta \Vert _{0}\leq k_{n},\Vert f_{\theta
}\Vert _{\infty }\leq L\}}\Vert f_{\theta }-f_{0}\Vert ^{2}+2c_{3}\frac{\log
(1+\alpha _{n})}{n}=2c_{3}\frac{\log (1+\alpha _{n})}{n}$ if $f_{0}\in 
\mathcal{F}_{0}^{L}(k_{n};M_{n})$. The upper bound then follows.

\end{proof}

\vspace*{0.2in} \textit{Proof of Theorem \ref{LowReg}.}

\begin{proof}
Without loss of generality, we assume $\sigma^2 = 1$ for the error variance.
First, we give a simple fact. Let $B_{k}(\eta )=\{\theta :\Vert \theta \Vert
_{2}\leq \eta ,\Vert \theta \Vert _{0}\leq k\}$ and $\mathcal{B}%
_{k}(f_{0};\epsilon )=\{f_{\theta }:\Vert f_{\theta }\Vert \leq \epsilon
,\Vert \theta \Vert _{0}\leq k\}$ (take $f_{0}=0$). Then, under Assumption
SRC with $\gamma =k$, the $\frac{\underline{a}}{2\overline{a}}$-local $%
\epsilon $-packing entropy of $\mathcal{B}_{k}(f_{0};\epsilon )$ is lower
bounded by the $\frac{1}{2}$-local $\eta $-packing entropy of $B_{k}(\eta )$
with $\eta =\epsilon / \overline{a}$.

(i) The proof is essentially the same as that of Theorem \ref{LowAgg}. When $%
m^{\ast }=M_{n},$ the previous lower bounding method works with a slight
modification. When $\left( 1+\log \frac{M_{n}}{m^{\ast }}\right)
^{q/2}<m^{\ast }<M_{n},$ we again use the global entropy to derive the lower
bound based on Proposition \ref{YangBarron}. The key is to realize that in
the derivation of the metric entropy lower bound for $\{\theta :$ $\Vert
\theta \Vert _{q}\leq t_n\}$ in \cite{Kuhn2001}, an optimal size packing set
is constructed in which every member has at most $m_{\ast }$ non-zero
coefficients. Assumption SRC with $\gamma =m_{\ast }$ ensures that the $%
L_{2} $ distance on this packing set is equivalent to the $\ell_{2}$
distance on the coefficients and then we know the metric entropy of $%
\mathcal{F}_{q}(t_n;M_{n}) $ under the $L_{2}$ distance is at the order
given. The result follows as before. When $m^{\ast }\leq \left( 1+\log \frac{%
M_{n}}{m^{\ast }}\right) ^{q/2},$ observe that $\mathcal{F}%
_{q}(t_n;M_{n})\supset \{\beta x_{j}:|\beta |\leq t_n\}$ for any $1\leq
j\leq M_{n}.$ The use of the local entropy result in Proposition \ref%
{YangBarron} readily gives the desired result.

(ii) As in the proof of Theorem \ref{LowAgg}, without loss of generality, we
can assume $k_{n}\leq M_{n}/2.$ Together with the simple fact given at the
beginning of the proof, for $B_{k_{n}}(\epsilon /\overline{a})=\{\theta :$$%
\Vert \theta \Vert _{2}\leq \epsilon /\overline{a},\Vert \theta \Vert
_{0}\leq k_{n}\},$ with $\eta ^{\prime }=\epsilon /(\overline{a}\sqrt{k_{n}}%
),$ we know $B_{k_{n}}(\epsilon /\overline{a})$ contains the set 
\begin{equation*}
\{\theta =\eta ^{\prime }I:I\in \{1,0,-1\}^{M_{n}},\Vert I\Vert _{0}\leq
k_{n}\}.
\end{equation*}%
For $\theta _{1}=\eta ^{\prime }I_{1},\theta _{2}=\eta ^{\prime }I_{2}$ both
in the above set, by Assumption SRC, $\Vert f_{\theta _{1}}-f_{\theta
_{2}}\Vert ^{2}\geq \underline{a}^{2}\eta ^{^{\prime
}2}d_{HM}(I_{1},I_{2})\geq \underline{a}^{2}\epsilon ^{2}/(2\overline{a}%
^{2}) $ when the Hamming distance $d_{HM}(I_{1},I_{2})$ is larger than $%
k_{n}/2.$ With the derivation in the proof of part (i) of Theorem \ref%
{LowAgg} (case 2), we know the local entropy $M_{\underline{a}/(\sqrt{2}%
\overline{a})}^{loc}(\epsilon )$ of $\mathcal{F}_{0}(k_{n};M_{n})\cap
\{f_{\theta }:\Vert \theta \Vert _{2}\leq a_{n}\}$ with $a_{n}\geq \epsilon $
is lower bounded by $\frac{k_{n}}{2}\log \frac{2(M_{n}-k_{n})}{k_{n}}.$
Then, under the condition $a_{n}\geq C\sqrt{k_{n}\left( 1+\log \frac{M_{n}}{%
k_{n}}\right) /n}$ for some constant $C,$ the minimax lower rate $%
k_{n}\left( 1+\log \frac{M_{n}}{k_{n}}\right) /n$ follows from a slight
modification of the proof of Theorem \ref{LowAgg} with $\epsilon =C^{\prime }%
\sqrt{k_{n}\left( 1+\log \frac{M_{n}}{k_{n}}\right) /n}$ for some constant $%
C^{\prime }>0.$ When $0<a_{n}<C\sqrt{k_{n}\left( 1+\log \frac{M_{n}}{k_{n}}%
\right) /n},$ with $\epsilon $ of order $a_{n},$ the lower bound follows.

(iii) For the larger $k_{n}$ case, from the proof of part (i) of the
theorem, we have used fewer than $k_{n}$ nonzero components to derive the
minimax lower bound there. Thus, the extra $\ell_{0}$-constraint does not
change the problem in terms of lower bound. For the smaller $k_{n}$ case,
note that for $\theta $ with $\Vert \theta \Vert _{0}\leq k_{n},$ $\Vert
\theta \Vert _{q}\leq k_{n}^{1/q-1/2}\Vert \theta \Vert _{2}\leq
k_{n}^{1/q-1/2} \sqrt{ C k_n \left( 1 + \log \frac{M_n}{k_n} \right) / n }$
for $\theta $ with $\Vert \theta \Vert _{2}\leq \sqrt{ C k_n \left( 1 + \log 
\frac{M_n}{k_n} \right) / n }.$ Therefore the $\ell_{q}$-constraint is
automatically satisfied when $\Vert \theta \Vert _{2}$ is no larger than the
critical order $\sqrt{ k_n \left( 1 + \log \frac{M_n}{k_n} \right) / n }$,
which is sufficient for the lower bound via local entropy techniques. The
conclusion follows.

\end{proof}

\vspace*{0.2in} \textit{Proof of Theorem \ref{MinimaxReg}.}

\begin{proof}
(i) We only need to derive the lower bound part. Under the assumptions that $%
\sup_{j}\Vert X_{j}\Vert _{\infty }\leq L_{0}<\infty $ for some constant $%
L_{0}>0$, for a fixed $t_{n}=t>0,$ we have $\forall f_{\theta }\in \mathcal{F%
}_{q}(t_{n};M_{n})$, $\Vert f_{\theta }\Vert _{\infty }\leq \sup_{j}\Vert
X_{j}\Vert _{\infty }\cdot \sum_{j=1}^{M_{n}}|\theta _{j}|\leq L_{0}\Vert
\theta \Vert _{1}\leq L_{0}\Vert \theta \Vert _{q}\leq L_{0}t$. Then the
conclusion follows directly from Theorem \ref{LowReg} (Part (i)). Note that
when $t_{n}$ is fixed, the case $m_{\ast }=1$ needs not to be separately
considered.

(ii) For the upper rate part, we use the \textbf{AC-C} upper bound. For $%
f_{\theta }$ with $\Vert \theta \Vert _{\infty }\leq L_{0},$ clearly, we
have $\Vert \theta \Vert _{1}\leq M_{n}L_{0},$ and consequently, since $\log
(1+M_{n}L_{0})$ is upper bounded by a multiple of $k_{n}\left( 1+\log \frac{%
M_{n}}{k_{n}}\right) ,$ the upper rate $\frac{k_{n}}{n}\left( 1+\log \frac{%
M_{n}}{k_{n}}\right) \wedge 1$ is obtained from Theorem \ref{UpReg}. Under
the assumptions that $\sup_{j}\Vert X_{j}\Vert _{\infty }\leq L_{0}<\infty $
and $k_{n}\sqrt{\left( 1+\log \frac{M_{n}}{k_{n}}\right) /n}\leq \sqrt{K_{0}}
$, we know that $\forall f_{\theta }\in \mathcal{F}_{0}(k_{n};M_{n})\bigcap
\{f_{\theta }:\Vert \theta \Vert _{2}\leq a_{n}\}$ with $a_{n}=C\sqrt{%
k_{n}\left( 1+\log \frac{M_{n}}{k_{n}}\right) /n}$ for some constant $C>0$,
the sup-norm of $f_{\theta }$ is upper bounded by 
\begin{equation*}
\Vert \sum_{j=1}^{M_{n}}\theta _{j}x_{j}\Vert _{\infty }\leq L_{0}\Vert
\theta \Vert _{1}\leq L_{0}\sqrt{k_{n}}a_{n}=CL_{0}k_{n}\sqrt{\frac{1+\log 
\frac{M_{n}}{k_{n}}}{n}}\leq C\sqrt{K_{0}}L_{0}.
\end{equation*}%
Then the functions in $\mathcal{F}_{0}(k_{n};M_{n})\bigcap \{f:\Vert \theta
\Vert _{2}\leq a_{n}\}$ have sup-norm uniformly bounded. Note that for
bounded $a_{n},$ $\Vert \theta \Vert _{2}\leq a_{n}$ implies that $\Vert
\theta \Vert _{\infty }\leq a_{n}.$ Thus, the extra restriction $\Vert
\theta \Vert _{\infty }\leq L_{0}$ does not affect the minimax lower rate
established in part (ii) of Theorem \ref{LowReg}.

(iii) The upper and lower rates follow similarly from Theorems \ref{UpReg}
and \ref{LowReg}. The details are thus skipped.

\end{proof}


Now we turn to the setup in Section \ref{linearregressionfixed} with $\sigma
^{2}$ known.

\begin{proposition}
\label{Yang1999a} (Yang \cite{Yang1999}, Theorem 1) When $\lambda \geq
5.1\log 2$, we have 
\begin{equation*}
E(ASE(\hat{f}_{\hat{J}}))\leq B\inf_{J\in \Gamma _{n}}\left( \left\Vert \bar{%
f}_{J}-{f}_{0}^{n}\right\Vert _{n}^{2}+\frac{\sigma ^{2}r_{J}}{n}+\frac{%
\lambda \sigma ^{2}C_{J}}{n}\right) ,
\end{equation*}%
where $B>0$ is a constant that depends only on $\lambda $.
\end{proposition}

\vspace*{.2in} \textit{Proof of Theorem \ref{UpABC}.}

\begin{proof}
The general case (ii) is easily derived based on our estimation procedure
and Proposition \ref{Yang1999a}.

To prove (i), when $\mathcal{F} = \mathcal{F}_q (t_n; M_n)$, according to
the upper bound in (ii) and Theorem \ref{Th1}, when $f_0^n \in \mathcal{F}_q
(t_n; M_n)$, for any $1 \leq m \leq (M_n - 1) \wedge n$, there exists a
subset $J_m$ and $f_{\theta^m} \in \mathcal{F}_{J_m}$ such that 
\begin{eqnarray*}
& & E (ASE (\hat{f}_{\hat{J}})) \\
& \leq & B \left( \left( \left\| {f}_{\theta^{m}} -{f}_0^n \right\|_n^2 + 
\frac{ \sigma^2 r_{J_m}}{n} + \frac{\sigma^2 \log ( M_n \wedge n) }{n} + 
\frac{\sigma^2 \log {\binom{M_n }{m}} }{n} \right) \wedge \frac{\sigma^2
r_{M_n}}{n} \right) \\
& & \wedge B \left( \left( \| \bar{f}_{J_0} - f_0^n \|_n^2 + \frac{\sigma^2 
}{n} \right) \wedge \sigma^2 \right) \\
& \leq & B \left( \left( 2^{2 / q - 1} t_n^{2} m^{1 - 2 / q} + \frac{
\sigma^2 r_{J_m}}{n} + \frac{\sigma^2 \log ( M_n \wedge n) }{n} +\frac{%
\sigma^2 \log {\binom{M_n }{m}}}{n} \right) \wedge \frac{\sigma^2 r_{M_n}}{n}
\right) \\
& & \wedge B \left( \left( t_n^2 + \frac{\sigma^2 }{n} \right) \wedge
\sigma^2 \right).
\end{eqnarray*}
Since $\log {\binom{M_n }{m}} \leq m \left( 1 + \log \frac{M_n}{m} \right)$
and $\log M_n \leq m \left( 1 + \log \frac{M_n}{m} \right)$, then for models
with size $1 \leq m \leq (M_n - 1) \wedge n$, we have 
\begin{eqnarray*}
E (ASE (\hat{f}_{\hat{J}})) & \leq & B^{\prime }\left( \left( t_n^{2} m^{1 -
2 / q} + \frac{\sigma^2 r_{J_m}}{n} + \frac{\sigma^2 m \left( 1 + \log \frac{%
M_n }{m} \right)}{n} \right) \wedge \frac{\sigma^2 r_{M_n}}{n} \right) \\
& & \wedge B \left( \left( t_n^2 + \frac{\sigma^2 }{n} \right) \wedge
\sigma^2 \right),
\end{eqnarray*}
where $B^{\prime }$ only depends on $q$ and $\lambda$.

When $m_{\ast }=m^{\ast }=M_{n}\wedge n$, the full model $J_{M_{n}}$ leads
to an upper bound of order $\frac{\sigma ^{2}r_{M_{n}}}{n}$. When $1<m_{\ast
}\leq m^{\ast }<M_{n}\wedge n$, we get the desired upper bounds by
evaluating the risk bounds choosing $J_{m_{\ast }}$ and $J_{M_{n}}$. When $%
m_{\ast }=1$, models $J_{0}$ and $J_{M_{n}}$ result in the desired upper
bound.

The arguments for cases $\mathcal{F}=\mathcal{F}_{0}(k_{n};M_{n})$ and $%
\mathcal{F}=\mathcal{F}_{q}(t_{n};M_{n})\cap \mathcal{F}_{0}(k_{n};M_{n})$
are similar to those of Theorem \ref{UpAgg} and above with $r_{J_{m}}$
replacing $m$ in the upper bounds.

\end{proof}

\vspace*{.2in} \textit{Proof of Theorem \ref{LowABC}.}

\begin{proof}
Without loss of generality, assume the error variance $\sigma^2 = 1$. Let $%
P_{f}(y^{n})=\prod _{i=1}^{n}\frac{1}{\sqrt{2\pi }}\exp \left( \right.$ $%
\left. -\frac{1}{2}(y_{i}-f(\mathbf{x}_{i})^{2} )\right) $ denote the joint
density of $Y^{n}=(Y_{1},...,Y_{n})^{^{\prime }},$ where the components are
independent with mean $f(\mathbf{x}_{i})$ and variance $1,$ $1\leq i\leq n.$
Then the Kullback-Leibler distance between $P_{f_{1}}(y^{n})$ and $%
P_{f_{2}}(y^{n})$ is%
\begin{equation*}
D\left( P_{f_{1}}(y^{n})\parallel P_{f_{2}}(y^{n})\right) =\frac{1}{2}%
\sum_{i=1}^{n}\left( f_{1}(\mathbf{x}_{i})-f_{2}(\mathbf{x}_{i})\right) ^{2}.
\end{equation*}%
To prove the lower bounds, instead of the global $L_{2}$ distance on the
regression functions, we need to work with the distance $d(f_{1},f_{2})=%
\sqrt{\sum_{i=1}^{n}\left( f_{1}(\mathbf{x}_{i})-f_{2}(\mathbf{x}%
_{i})\right) ^{2}}.$

First consider the case $\mathcal{F}=\mathcal{F}_{q}(t_{n};M_{n})$. Let $%
B_{k}(\eta )=\{\theta :\Vert \theta \Vert _{2}\leq \eta ,\Vert \theta \Vert
_{0}\leq k\}$ and $\mathcal{B}_{k}(f_{0};\epsilon )=\{f_{\theta }:\Vert
f_{\theta }\Vert _{n}\leq \epsilon ,\Vert \theta \Vert _{0}\leq k\}$ ($%
f_{0}=0$). Then, under Assumption SRC$^{\prime }$ with $\gamma =k$, the $%
\frac{\underline{a}}{2\overline{a}}$-local $\epsilon $-packing entropy of $%
\mathcal{B}_{k}(f_{0};\epsilon )$ is lower bounded by the $\frac{1}{2}$%
-local $\eta $-packing entropy of $B_{k}(\eta )$ with $\eta =\frac{\epsilon 
}{\overline{a}}$. When $\gamma =m_{\ast }$, the proof is the same as the
proof of Theorem \ref{LowReg}.

Now consider the case $\mathcal{F}=\mathcal{F}_{0}(k_{n};M_{n})$ and again
assume $k_{n}\leq M_{n}/2$ as in the proof of Theorem \ref{LowReg}. When
Assumption SRC$^{\prime }$ holds with $\gamma =k_{n}$, the lower bound is of
order $\frac{k_{n}(1+\log M_{n}/k_{n})}{n}$ as before in the random design
case. The proof for the last case $\mathcal{F}=\mathcal{F}%
_{q}(t_{n};M_{n})\cap \mathcal{F}_{0}(k_{n};M_{n})$ is similarly done as in
the proof of Theorem \ref{LowReg}.

\end{proof}

\newpage
 \textit{Proof of Theorem \ref{UpMLS}.}
\begin{proof} 
According to Corollary 6 from \cite{LeungBarron2006}, we have 
\begin{equation*}
E(ASE(\hat{f}^{MLS}))\leq \inf_{J\in \Gamma _{n}}\left( \Vert \bar{f}%
_{J}-f_{0}^{n}\Vert _{n}^{2}+\frac{\sigma ^{2}r_{J}}{n}+\frac{4\sigma
^{2}\log (1/\pi _{J})}{n}\right) ,
\end{equation*}%
which is basically the same as Proposition \ref{Yang1999a} with $B=1$. Thus,
the rest of the proof is basically the same as that of Theorem \ref{UpABC}.

\end{proof}

\vspace*{0.2in} To prove Theorem \ref{UpABCunknown} , we need an oracle
inequality, which improves Theorem 4 of \cite{Yang1999}, where only a
convergence in probability result is given. Suppose that only the subset
models $J_{m}$ with rank $r_{J_{m}}\leq n/2$ are considered (which is
automatically satisfied when $M_{n}\leq n/2$). Let $\Gamma $ denote these
models. (More generally, a risk bound similar to the following holds if we
consider models with size no more than $(1-\rho )n$ for any small $\rho >0$%
.) Let $C_{J}$ be the descriptive complexity of the model $J$ in $\Gamma .$


\begin{proposition}
\label{Yang1999b} When $\lambda \geq 40\log 2$, the selected model $\hat{J}%
^{\prime }$ by ABC$^{\prime }$ satisfies 
\begin{equation*}
E(ASE(\hat{f}_{\hat{J}^{\prime }}))\leq B\inf_{J\in \Gamma }\left(
\left\Vert \bar{f}_{J}-{f}_{0}^{n}\right\Vert _{n}^{2}+\frac{\sigma ^{2}r_{J}%
}{n}+\frac{\lambda \sigma ^{2}C_{J}}{n}\right) ,
\end{equation*}%
where $B$ is a constant that depends on $\lambda ,$ $\overline{\sigma }^{2},$
and $\sigma ^{2}$.
\end{proposition}

\begin{remark}
If we add models with rank $r_{J}>n/2$ into the competition, as long as the
complexity assignment over all the models is valid (i.e., satisfying the
summability condition), if we can show that for these added models, $%
ABC^{\prime }(J)$ are also upper and lower bounded with high probabilities
as in (\ref{ABCLB}) and (\ref{ABCUB}), then the risk bound in the
proposition continues to hold.
\end{remark}

\begin{proof}
Let $e_{n}=(\varepsilon _{1},\ldots ,\varepsilon _{n})^{\prime }$.
For ease in writing, we simplify $\left\Vert \cdot \right\Vert _{n}^{2}$ to $%
\left\Vert \cdot \right\Vert ^{2}$ in this proof. From page 495 in \cite{Yang1999}, for each candidate model $J$, we have 
\begin{equation*}
ABC^{\prime }(J)=\Vert
A_{J}f_{0}^{n}\Vert ^{2}+r_{J}\left( \frac{2}{n-r_{J}}\left( \Vert Y_{n}-%
\hat{Y}_{J}\Vert ^{2}+\lambda \overline{\sigma }^{2}C_{J}\right) -\sigma
^{2}\right) +\lambda \overline{\sigma }^{2}C_{J}+2\text{rem}_{1}(J)+\text{rem%
}_{2}(J),
\end{equation*}
where $\Vert A_{J}f_{0}^{n}\Vert ^{2}=\Vert \overline{f}_{J}-{f}%
_{0}^{n}\Vert ^{2}$, $\text{rem}_{1}(J)=e_{n}^{\prime
}(f_{0}^{n}-M_{J}f_{0}^{n})$ and $\text{rem}_{2}(J)=r_{J}-e_{n}^{\prime
}M_{J}e_{n}$. Note also that $\Vert Y_{n}-\hat{Y}_{J}\Vert ^{2}+\lambda 
\overline{\sigma }^{2}C_{J}=\Vert A_{J}f_{n}\Vert ^{2}+(n-r_{J})\sigma
^{2}+\left( e_{n}^{\prime }A_{J}e_{n}-(n-r_{J})\sigma ^{2}\right)
+2e_{n}^{\prime }A_{J}f_{n}+\lambda \overline{\sigma }^{2}C_{J}$. Let 
\begin{equation*}
T(J)=\Vert A_{J}f_{0}^{n}\Vert ^{2}+(n-r_{J})\sigma ^{2}+\lambda \overline{%
\sigma }^{2}C_{J},\mbox{ and }nR_{n}(J)=\Vert A_{J}f_{0}^{n}\Vert
^{2}+r_{J}\sigma ^{2}+\lambda \overline{\sigma }^{2}C_{J}.
\end{equation*}%
As is shown in the proof of Theorem 1, \cite{Yang1999}, if $\lambda >h(\tau
_{1},\tau _{2})=\max (\sup_{\xi \geq 0}((2(\log 2)\xi )^{1/2}/\tau _{1}-\xi
),\sup_{\rho \geq 0}(\rho /\tau _{2}-1)2(\log 2)/(\rho -\log (\rho +1)))$
for some constants $\tau _{1}$ and $\tau _{2}$ with $2\tau _{1}+\tau _{2}<1,$
then for any $\delta >0$, with probability no less than $1-5\delta $, $|%
\text{rem}_{1}(J)|\leq \tau _{1}(nR_{n}(J)+g_{1}(\delta ))$, $|\text{rem}%
_{2}(J)|\leq \tau _{2}(nR_{n}(J)+g_{2}(\delta ))$, and $|e_{n}^{\prime
}A_{J}e_{n}-(n-r_{J})\sigma ^{2}|\leq \tau _{2}(T(J)+g_{2}(\delta ))$, where 
$g_{1}(\delta )=g_{2}(\delta )=\lambda \log _{2}(1/\delta )$. Then with
probability no less than $1-5\delta $, we have 
\begin{eqnarray}
ABC^{\prime }(J) &\geq &\Vert A_{J}f_{0}^{n}\Vert ^{2}+r_{J}\left( \frac{%
2(T(J)-\tau _{2}(T(J)+g_{2}(\delta ))-2\tau _{1}(nR_{n}(J)+g_{1}(\delta )))}{%
n-r_{J}}-\sigma ^{2}\right)  \notag \\
&&-2\tau _{1}(nR_{n}(J)+g_{1}(\delta ))-\tau _{2}(nR_{n}(J)+g_{2}(\delta
))+\lambda \overline{\sigma }^{2}C_{J}  \notag \\
&\geq &\Vert A_{J}f_{0}^{n}\Vert ^{2}+r_{J}\left( \frac{2(1-(2\tau _{1}+\tau
_{2}))T(J)}{n-r_{J}}-\frac{2(2\tau _{1}g_{1}(\delta )+\tau _{2}g_{2}(\delta
))}{n-r_{J}}-\sigma ^{2}\right)  \notag \\
&&-(2\tau _{1}+\tau _{2})nR_{n}(J)-(2\tau _{1}g_{1}(\delta )+\tau
_{2}g_{2}(\delta ))+\lambda \overline{\sigma }^{2}C_{J}  \notag \\
&\geq &\Vert A_{J}f_{0}^{n}\Vert ^{2}+r_{J}(1-(4\tau _{1}+2\tau _{2}))\sigma
^{2}-\frac{2r_{J}(2\tau _{1}g_{1}(\delta )+\tau _{2}g_{2}(\delta ))}{n-r_{J}}
\notag \\
&&-(2\tau _{1}+\tau _{2})nR_{n}(J)-(2\tau _{1}g_{1}(\delta )+\tau
_{2}g_{2}(\delta ))+\lambda \overline{\sigma }^{2}C_{J}  \notag \\
&\geq &(1-(6\tau _{1}+3\tau _{2}))nR_{n}(J)-\frac{n+r_{J}}{n-r_{J}}(2\tau
_{1}g_{1}(\delta )+\tau _{2}g_{2}(\delta )).  \label{ABCLB}
\end{eqnarray}%
Suppose $6\tau _{1}+3\tau _{2}<1$. Let $J_{n}$ be the candidate model that
minimizes $R_{n}(J)$. Then with exception probability less than $5\delta $,
we have 
\begin{eqnarray*}
ABC^{\prime }(J_{n}) &\leq &\Vert A_{J_{n}}f_{0}^{n}\Vert
^{2}+r_{J_{n}}\left( \frac{2(1+(2\tau _{1}+\tau _{2}))T(J_{n})}{n-r_{J_{n}}}%
-\sigma ^{2}\right) +(2\tau _{1}+\tau _{2})nR_{n}(J_{n}) \\
&&+\frac{n+r_{J_{n}}}{n-r_{J_{n}}}(2\tau _{1}g_{1}(\delta )+\tau
_{2}g_{2}(\delta ))+\lambda \overline{\sigma }^{2}C_{J_{n}}.
\end{eqnarray*}%
Since $%
T(J_{n})/(n-r_{J_{n}})=(1+r_{J_{n}}/(n-r_{J_{n}}))R_{n}(J_{n})+(1-r_{J_{n}}/(n-r_{J_{n}}))\sigma ^{2}\leq 2R_{n}(J_{n})+\sigma ^{2} 
$, then 
\begin{equation}
ABC^{\prime }(J_{n})\leq (5+14\tau _{1}+7\tau _{2})nR_{n}(J_{n})+\frac{%
n+r_{J_{n}}}{n-r_{J_{n}}}(2\tau _{1}g_{1}(\delta )+\tau _{2}g_{2}(\delta )).
\label{ABCUB}
\end{equation}%
Thus, for any $\delta >0$, when the sample size is large enough, we have
that with probability no less than $1-5\delta $, 
\begin{eqnarray*}
nR_{n}(\hat{J}^{\prime }) &\leq &\frac{ABC^{\prime }(\hat{J}^{\prime })+%
\frac{n+r_{\hat{J}^{\prime }}}{n-r_{\hat{J}^{\prime }}}(2\tau
_{1}g_{1}(\delta )+\tau _{2}g_{2}(\delta ))}{1-(6\tau _{1}+3\tau _{2})} \\
&\leq &\frac{ABC^{\prime }(J_{n})+\frac{n+r_{\hat{J}^{\prime }}}{n-r_{\hat{J}%
^{\prime }}}(2\tau _{1}g_{1}(\delta )+\tau _{2}g_{2}(\delta ))}{1-(6\tau
_{1}+3\tau _{2})} \\
&\leq &\frac{(5+14\tau _{1}+7\tau _{2})nR_{n}(J_{n})+\frac{n+r_{J_{n}}}{%
n-r_{J_{n}}}(2\tau _{1}g_{1}(\delta )+\tau _{2}g_{2}(\delta ))+\frac{n+r_{%
\hat{J}^{\prime }}}{n-r_{\hat{J}^{\prime }}}(2\tau _{1}g_{1}(\delta )+\tau
_{2}g_{2}(\delta ))}{1-(6\tau _{1}+3\tau _{2})}.
\end{eqnarray*}%
Thus, with probability at least $1-5\delta ,$ 
\begin{eqnarray*}
R_{n}(\hat{J}^{\prime })/R_{n}(J_{n}) &\leq &\frac{5+14\tau _{1}+7\tau _{2}}{%
1-(6\tau _{1}+3\tau _{2})}+\frac{\left( \frac{n+r_{J_{n}}}{n-r_{J_{n}}}+%
\frac{n+r_{\hat{J}^{\prime }}}{n-r_{\hat{J}^{\prime }}}\right) (2\tau
_{1}g_{1}(\delta )+\tau _{2}g_{2}(\delta ))}{(1-(6\tau _{1}+3\tau
_{2}))nR_{n}(J_{n})} \\
&\leq &\frac{5+14\tau _{1}+7\tau _{2}}{1-(6\tau _{1}+3\tau _{2})}+\frac{%
\left( \frac{n+r_{J_{n}}}{n-r_{J_{n}}}+\frac{n+r_{\hat{J}^{\prime }}}{n-r_{%
\hat{J}^{\prime }}}\right) (2\tau _{1}g_{1}(\delta )+\tau _{2}g_{2}(\delta ))%
}{(1-(6\tau _{1}+3\tau _{2}))\sigma ^{2}} \\
&\leq &\frac{5+14\tau _{1}+7\tau _{2}}{1-(6\tau _{1}+3\tau _{2})}+\frac{%
6(2\tau _{1}g_{1}(\delta )+\tau _{2}g_{2}(\delta ))}{(1-(6\tau _{1}+3\tau
_{2}))\sigma ^{2}}.
\end{eqnarray*}%
Let 
\begin{equation*}
\widetilde{W}=b_{n}^{-1}\left( \frac{R_{n}(\hat{J}^{\prime })}{R_{n}(J_{n})}-%
\frac{5+14\tau _{1}+7\tau _{2}}{1-(6\tau _{1}+3\tau _{2})}\right) \text{ and 
}b_{n}=\frac{6(2\tau _{1}+\tau _{2})\lambda }{(1-(6\tau _{1}+3\tau
_{2}))\sigma ^{2}}.
\end{equation*}%
Then $P\left( \widetilde{W}\geq -\log _{2}\delta \right) \leq 5\delta $ for $%
0<\delta <1$. Since $E(\widetilde{W}^{+})=\int_{0}^{\infty }P(\widetilde{W}%
\geq t)dt\leq 5\int_{0}^{\infty }2^{-t}dt=5/\ln 2$ and $R_{n}(J_{n})\leq (%
\overline{\sigma }^{2}/\sigma ^{2})\inf_{J\in \Gamma }R_{n}(f_{0};J)$ where $%
R_{n}(f_{0};J)=\Vert \overline{f}_{J}-f_{0}^{n}\Vert _{n}^{2}+r_{J}\sigma
^{2}/n+\lambda \sigma ^{2}C_{J}/n$, then we have 
\begin{eqnarray*}
\frac{E\left( R_{n}(\hat{J}^{\prime })\right) }{\inf_{J\in \Gamma
}R_{n}(f_{0};J)} &=&\frac{E\left( R_{n}(\hat{J}^{\prime })\right) }{%
R_{n}(J_{n})}\cdot \frac{R_{n}(J_{n})}{\inf_{J\in \Gamma }R_{n}(f_{0};J)} \\
&\leq &\left( \frac{5+14\tau _{1}+7\tau _{2}}{1-(6\tau _{1}+3\tau _{2})}+%
\frac{30(2\tau _{1}+\tau _{2})\lambda }{(\ln 2)(1-(6\tau _{1}+3\tau
_{2}))\sigma ^{2}}\right) \cdot \left( \frac{\overline{\sigma }^{2}}{\sigma
^{2}}\right) .
\end{eqnarray*}%
So $E(ASE(\hat{f}_{\hat{J}^{\prime }}))\leq B\inf_{J\in \Gamma
}R_{n}(f_{0};J)$, where the constant $B$ depends on $\tau _{1}$, $\tau _{2},$
$\overline{\sigma },$ and $\sigma $. Minimizing $h(\tau _{1},\tau _{2})$
over $\tau _{1}>0$ and $\tau _{2}>0$ in the region $6\tau _{1}+3\tau _{2}<1$%
, one finds a minimum value less than $40\log 2$. Thus, the results of the
theorem hold when $\lambda \geq 40\log 2$.

\end{proof}

Proposition \ref{Yang1999b} may not provide optimal risk rate when $%
r_{M_{n}} $ is small, or when $r_{M_{n}}$ is larger than $n/2$ (in which
case the risk bound on $E(ASE(\hat{J}^{\prime }))$ can be arbitrarily large
because the approximation errors can be arbitrarily large when the models
are restricted to be of size $n/2$ or smaller). The issue can be resolved by
considering the full model $J_{M_{n}}$ and the full projection model $\bar{J}
$ in the candidate model list, as described before Theorem \ref{UpABCunknown} .

\vspace*{.2in} \textit{Proof of Theorem \ref{UpABCunknown} :} 

\begin{proof}
Observe that for the full projection model $\bar{J}$, with the chosen 
$C_{\bar{J}}$, we have that%
\begin{equation*}
(1-(6\tau _{1}+3\tau _{2}))nR_{n}(\bar{J})\leq ABC^{\prime }(\bar{J})\leq
\xi nR_{n}(\bar{J})=\xi \left( n\sigma ^{2}+\lambda \overline{\sigma }^{2}C_{%
\bar{J}}\right)
\end{equation*}%
for some constant $\xi >0$ that depends only on $\lambda ,$ $\overline{%
\sigma }^{2}$ and $\sigma ^{2}$. From the remark after Proposition \ref%
{Yang1999b}, we have the following risk bounds for the three situations.
Below $B$ and $B^{\prime }$ are constants depending only on $\lambda $, $%
\overline{\sigma }^2,$ and $\sigma ^{2}$.

\begin{enumerate}
\item When $M_{n}\leq n/2,$ we have the general risk bound 
\begin{eqnarray*}
&&E(ASE(\hat{f}_{\hat{J}^{\prime }})) \\
&\leq &B^{\prime} \left( \inf_{J_{m}:1\leq m<M_{n}}\left( \left\Vert \bar{f}%
_{J_{m}}-{f}_{0}^{n}\right\Vert _{n}^{2}+\frac{\sigma ^{2}r_{J_{m}}}{n}+%
\frac{\lambda \sigma ^{2}C_{J_{m}}}{n}\right) \wedge \left( \left\Vert \bar{f%
}_{J_{M_{n}}}-{f}_{0}^{n}\right\Vert _{n}^{2}+\frac{\sigma ^{2}r_{M_{n}}}{n}%
\right) \right. \\
&&\left. \wedge R_{n}(\bar{J})\wedge R_{n}(J_{0})\right) \\
&\leq &B^{\prime }\left( \Vert \bar{f}_{J_{M_{n}}}-f_{0}^{n}\Vert
_{n}^{2}+\inf_{J_{m}:1\leq m<M_{n}}\left( \Vert \bar{f}_{J_{m}}-\bar{f}%
_{J_{M_{n}}}\Vert _{n}^{2}+\frac{\sigma ^{2}r_{J_{m}}}{n}+\frac{\sigma
^{2}\log (M_{n}-1)}{n}\right. \right. \\
&&\left. \left. +\frac{\sigma ^{2}\log {\binom{M_{n}}{m}}}{n}\right) \wedge 
\frac{\sigma ^{2}r_{M_{n}}}{n}\right) \wedge B^{\prime }\left( \left( \Vert 
\bar{f}_{J_{0}}-f_{0}^{n}\Vert _{n}^{2}+\frac{\sigma ^{2}}{n}\right) \wedge
\sigma ^{2}\right) .
\end{eqnarray*}%
For $f_{0}^{n}\in \mathcal{F}_{q}(t_{n};M_{n})$, from above, by an argument
similar to that in Theorem \ref{UpABC}, for any $1\leq m<M_{n}$, there
exists a subset $J_{m}$ and $f_{\theta ^{m}}\in \mathcal{F}_{J_{m}}$ such
that 
\begin{eqnarray}
E(ASE(\hat{f}_{\hat{J}^{\prime }})) &\leq &B^{\prime }\left( \left(
t_{n}^{2}m^{1-2/q}+\frac{\sigma ^{2}r_{J_{m}}}{n}+\frac{\sigma ^{2}m\left(
1+\log \frac{M_{n}}{m}\right) }{n}\right) \wedge \frac{\sigma ^{2}r_{M_{n}}}{%
n}\right)  \notag \\
&&\wedge B^{\prime }\left( \left( t_{n}^{2}+\frac{\sigma ^{2}}{n}\right)
\wedge \sigma ^{2}\right) .  \label{UpABC'}
\end{eqnarray}%
When $m_{\ast }=m^{\ast }=M_{n}$, the full model $J_{M_{n}}$ leads to an
upper bound of order $\frac{\sigma ^{2}r_{M_{n}}}{n}$. When $1<m_{\ast
}<M_{n}$, we get the desired upper bound by taking the smaller value of the
index of resolvability at $J_{m_{\ast }}$ and $J_{M_{n}}$. When $m_{\ast }=1$%
, the smaller value of the index of resolvability at $J_{0}$ and $J_{M_{n}}$
results in the given upper bound.

The arguments for cases $\mathcal{F}=\mathcal{F}_{0}(k_{n};M_{n})$ and $%
\mathcal{F}=\mathcal{F}_{q}(t_n;M_{n})\cap \mathcal{F}_{0}(k_{n};M_{n})$ are
similar to those of Theorem \ref{UpABC}.

\item When $M_{n}>n/2$ and $r_{M_{n}}\geq n/2,$ evaluating the index of
resolvability gives 
\begin{eqnarray*}
&&E(ASE(\hat{f}_{\hat{J}^{\prime }})) \\
&\leq &B\left( \inf_{J_{m}:1\leq m\leq n/2}\left( \left\Vert \bar{f}_{J_{m}}-%
{f}_{0}^{n}\right\Vert _{n}^{2}+\frac{\sigma ^{2}r_{J_{m}}}{n}+\frac{\lambda
\sigma ^{2}C_{J_{m}}}{n}\right) \wedge R_{n}(\bar{J})\wedge
R_{n}(J_{0})\right) \\
&\leq &B^{\prime }\inf_{J_{m}:1\leq m\leq n/2}\left( \Vert \bar{f}_{J_{m}}-{f%
}_{0}^{n}\Vert _{n}^{2}+\frac{\sigma ^{2}r_{J_{m}}}{n}+\frac{\sigma ^{2}\log
\lfloor n/2\rfloor }{n}+\frac{\sigma ^{2}\log {\binom{M_{n}}{m}}}{n}\right)
\\
&&\wedge B^{\prime }\left( \left( \Vert \bar{f}_{J_{0}}-f_{0}^{n}\Vert
_{n}^{2}+\frac{\sigma ^{2}}{n}\right) \wedge \sigma ^{2}\right) .
\end{eqnarray*}%
In this case, for the full model, clearly, we have $\Vert \bar{f}%
_{J_{M_{n}}}-f_{0}^{n}\Vert _{n}^{2}+\frac{\sigma ^{2}r_{M_{n}}}{n}\geq 
\frac{1}{2}\sigma ^{2},$ which cannot be better than the model $\bar{J}$ up
to a constant factor. We next show that adding the models with size $%
n/2<m<M_{n}$ does not help either in terms of the rate in the risk bound. If 
$r_{J_{m}}\geq r_{M_{n}}/2$, then obviously $\Vert \bar{f}_{J_{m}}-{f}%
_{0}^{n}\Vert _{n}^{2}+\frac{\sigma ^{2}r_{J_{m}}}{n}+\frac{\sigma ^{2}\log
\lfloor n/2\rfloor }{n}+\frac{\sigma ^{2}\log {\binom{M_{n}}{m}}}{n}\geq 
\frac{1}{4}\sigma ^{2}.$ For $r_{J_{m}}<r_{M_{n}}/2$, if $n/2<m\leq M_{n}/2,$
then there exists a smaller model with size $\widetilde{m}\leq n/2$ that has
the same approximation error and rank, but smaller complexity $C_{J_{%
\widetilde{m}}}$ (i.e., $C_{J_{\widetilde{m}}}\leq C_{J_{m}}$), where $%
C_{J_{m}}=\log (n\wedge M_{n})+\log {\binom{M_{n}}{m}}$ when $m>n/2$. If $%
m>M_{n}/2$ (and $r_{J_{m}}<r_{M_{n}}/2$), then due to the monotonicity of
the function ${\binom{M_{n}}{m}}$ in $m\geq M_{n}/2,$ since there must be
more than $r_{M_{n}}/2$ terms left out in the model, we must have $\log {%
\binom{M_{n}}{m}\geq \log \binom{M_{n}}{M_{n}-\lfloor r_{M_{n}}/2\rfloor }%
\geq }\lfloor r_{M_{n}}/2\rfloor \log \frac{M_{n}}{\lfloor
r_{M_{n}}/2\rfloor },$ which is at least of order $n$ under the condition $%
r_{M_{n}}\geq n/2.$ Putting the above facts together, we conclude that
adding the models with size $n/2<m\leq M_{n}$ does not affect the validity
of the risk bound given in part (ii) of Theorem \ref{UpABCunknown} (note
that $\log \lfloor n/2\rfloor $ is of the same order as $\log (M_{n}\wedge
n) $ in our case). Then, the general risk upper bound becomes (with $%
B^{\prime } $ enlarged by an absolute constant factor) 
\begin{eqnarray*}
&&B^{\prime }\inf_{J_{m}:1\leq m<M_{n}}\left( \Vert \bar{f}%
_{J_{m}}-f_{0}^{n}\Vert _{n}^{2}+\frac{\sigma ^{2}r_{J_{m}}}{n}+\frac{\sigma
^{2}\log (M_{n}\wedge n)}{n}+\frac{\sigma ^{2}\log {\binom{M_{n}}{m}}}{n}%
\right) \\
&&\wedge B^{\prime }\left( \left( \Vert \bar{f}_{J_{0}}-f_{0}^{n}\Vert
_{n}^{2}+\frac{\sigma ^{2}}{n}\right) \wedge \sigma ^{2}\right) \wedge
B^{\prime }\left( \Vert \bar{f}_{J_{M_{n}}}-f_{0}^{n}\Vert _{n}^{2}+\frac{%
\sigma ^{2}r_{M_{n}}}{n}\right) \\
&\leq &B^{\prime }\left( \Vert \bar{f}_{J_{M_{n}}}-f_{0}^{n}\Vert
_{n}^{2}+\inf_{J_{m}:1\leq m<M_{n}}\left( \Vert \bar{f}_{J_{m}}-\bar{f}%
_{J_{M_{n}}}\Vert _{n}^{2}+\frac{\sigma ^{2}r_{J_{m}}}{n}+\frac{\sigma
^{2}\log (M_{n}\wedge n)}{n}\right. \right. \\
&&\left. \left. +\frac{\sigma ^{2}\log {\binom{M_{n}}{m}}}{n}\right) \wedge 
\frac{\sigma ^{2}r_{M_{n}}}{n}\right) \wedge B^{\prime }\left( \left( \Vert 
\bar{f}_{J_{0}}-f_{0}^{n}\Vert _{n}^{2}+\frac{\sigma ^{2}}{n}\right) \wedge
\sigma ^{2}\right) .
\end{eqnarray*}
For $f_{0}^{n}\in \mathcal{F}_{q}(t_{n};M_{n})$ and any $1\leq m<M_{n}$,
there exists a subset $J_{m}$ and $f_{\theta ^{m}}\in \mathcal{F}_{J_{m}}$
such that the inequality (\ref{UpABC'}) holds. When $m_{\ast }=m^{\ast
}=M_{n}\wedge n$, the full projection model $\bar{J}$ leads to an upper
bound of order $\sigma ^{2}$. When $1<m_{\ast }<M_{n}\wedge n$, we get the
desired upper bounds by choosing $J_{m_{\ast }}$ and $\bar{J}$ to evaluate
the index of resolvability. When $m_{\ast }=1$, models $J_{0}$ and $\bar{J}$
result in the desired upper bound.

\item When $M_{n}>n/2$ and $r_{M_{n}}<n/2,$ the full model is already
included, and, similarly as above, the models with $n/2<m<M_{n}$ can be
included in the minimization set of the general risk bound. Indeed, if $%
r_{M_{n}}=1,$ the statement is trivial. If $r_{J_{m}}\geq r_{M_{n}}/2$, then 
$\Vert \bar{f}_{J_{m}}-{f}_{0}^{n}\Vert _{n}^{2}+\frac{\sigma ^{2}r_{J_{m}}}{%
n}+\frac{\sigma ^{2}\log \lfloor n/2\rfloor }{n}+\frac{\sigma ^{2}\log {%
\binom{M_{n}}{m}}}{n}\geq \left\Vert \bar{f}_{J_{M_{n}}}-{f}%
_{0}^{n}\right\Vert _{n}^{2}+\frac{\sigma ^{2}J_{M_{n}}}{2n},$ which means
that the model cannot beat the full model up to a constant factor. For $%
r_{J_{m}}<r_{M_{n}}/2$, if $m>M_{n}/2,$ then we again have $\log {\binom{%
M_{n}}{m}\geq \log \binom{M_{n}}{M_{n}-\lfloor r_{M_{n}}/2\rfloor }\geq }%
\lfloor r_{M_{n}}/2\rfloor \log \frac{M_{n}}{\lfloor r_{M_{n}}/2\rfloor }.$
Thus there exists a model in $\Gamma _{n}^{^{\prime }}$ with the same rank
of $r_{J_{m}}\leq n/2$ and approximation error, and its complexity is at
most at the same order as $J_{m}.$ Then with the same arguments for the case
of $r_{M_{n}}\geq n/2$, we again conclude that adding the models with size $%
n/2<m\leq M_{n}$ does not affect the validity of the risk bound given in
part (ii) of Theorem \ref{UpABCunknown}. Thus, the general risk bound is 
\begin{eqnarray*}
&&E(ASE(\hat{f}_{\hat{J}^{\prime }})) \\
&\leq &B\left\{ \left( \left\Vert \bar{f}_{J_{M_{n}}}-{f}_{0}^{n}\right\Vert
_{n}^{2}+\inf_{J_{m}:1\leq m\leq n/2}\left( \left\Vert \bar{f}_{J_{m}}-\bar{f%
}_{J_{M_{n}}}\right\Vert _{n}^{2}+\frac{\sigma ^{2}r_{J_{m}}}{n}+\frac{%
2\lambda \sigma ^{2}C_{J_{m}}}{n}\right) \wedge \right. \right. \\
&&\left. \left. \frac{\sigma ^{2}r_{M_{n}}}{n}\right) \right\} \wedge
B\left( \left( \Vert \bar{f}_{J_{0}}-f_{0}^{n}\Vert _{n}^{2}+\frac{\sigma
^{2}}{n}\right) \wedge \sigma ^{2}\right) \\
&\leq &B^{\prime }\left( \Vert \bar{f}_{J_{M_{n}}}-f_{0}^{n}\Vert
_{n}^{2}+\inf_{J_{m}:1\leq m<M_{n}}\left( \Vert \bar{f}_{J_{m}}-\bar{f}%
_{J_{M_{n}}}\Vert _{n}^{2}+\frac{\sigma ^{2}r_{J_{m}}}{n}+\frac{\sigma
^{2}\log (M_{n}\wedge n)}{n}\right. \right. \\
&&\left. \left. +\frac{\sigma ^{2}\log {\binom{M_{n}}{m}}}{n}\right) \wedge 
\frac{\sigma ^{2}r_{M_{n}}}{n}\right) \wedge B^{\prime }\left( \left( \Vert 
\bar{f}_{J_{0}}-f_{0}^{n}\Vert _{n}^{2}+\frac{\sigma ^{2}}{n}\right) \wedge
\sigma ^{2}\right) .
\end{eqnarray*}%
For $f_{0}^{n}\in \mathcal{F}_{q}(t_{n};M_{n})$ and any $1\leq m<M_{n}$,
there exists a subset $J_{m}$ and $f_{\theta ^{m}}\in \mathcal{F}_{J_{m}}$
such that the inequality (\ref{UpABC'}) holds. When $m_{\ast }=m^{\ast
}=M_{n}\wedge n$, the full model $J_{M_{n}}$ leads to an upper bound of
order $\frac{\sigma ^{2}r_{M_{n}}}{n}$. When $1<m_{\ast }<M_{n}\wedge n$, we
get the desired upper bounds by choosing $J_{m_{\ast }}$ and $J_{M_{n}}$
when evaluating the index of resolvability. When $m_{\ast }=1 $, taking
models $J_{0}$ and $J_{M_{n}}$ results in the desired upper bound.
\end{enumerate}

\end{proof}

\vspace{.2in} \textit{Proof of Theorem \ref{Baraud}:}

\begin{proof}
The proof is similar to that of Theorem \ref{UpABCunknown} except that we
use the oracle inequality (4.7) in \cite{Baraudetal2009} instead of that in
Proposition \ref{Yang1999b} (and there is no need to consider the different
scenarios). Note that if $M_{n}\leq \left( n-7\right) \wedge {\varsigma n}, $
then $m\vee \log {\binom{M_{n}}{m}<\varsigma n}$ for all $1\leq m\leq $ $%
M_{n}.$ Thus all subset models are allowed by the BGH criterion. When $M_{n}$
is larger, however, the conditions required in Corollary 1 of \cite%
{Baraudetal2009} may invalidate the choice of $m_{\ast }$ or $k_{n}$ when it
is too large, hence the upper bound assumption on $m_{\ast }$ and $k_{n}$.
We skip the details of the proof.

\end{proof}

\newpage

\section*{Acknowledgments}

We thank Yannick Baraud for helpful discussions and for pointing out the
work of Audibert and Catoni (2010) that helped us to remove a logarithmic
factor in some of our previous aggregation risk bounds and for his
suggestion of handling the case of fully unknown error variance for fixed
design. Sandra Paterlini conducted part of this research while visiting the
School of Mathematics, University of Minnesota.

\end{document}